\documentclass[11pt]{amsart}
\usepackage{verbatim, amscd, epsfig,amsmath,amsthm,amstext,amssymb,hyperref,bbm,txfonts,float  
}

\theoremstyle{plain}
\newtheorem*{theorem*}{Theorem} 
\newtheorem{theorem}{Theorem}[section]
\newtheorem{lemma}[theorem]{Lemma}
\newtheorem{cor}[theorem]{Corollary}

\newtheorem{prop}[theorem]{Proposition}
\newtheorem{rem}[theorem]{Remark}

\theoremstyle{definition}
\newtheorem{definition}[theorem]{Definition}

\renewcommand{\Re}{{\rm Re}\,}
\renewcommand{\Im}{{\rm Im}\,}

\newcommand{\R}{\mathbb{ R}}
\newcommand{\C}{\mathbb{ C}}
\newcommand{\Cc}{\mathfrak{C}}
\newcommand{\Z}{\mathbb{ Z}}

\renewcommand{\H}{\mathbb{ H}}

\newcommand{\N}{\mathbb{ N}}
\renewcommand{\P}{\mathbb{ P}}

\newcommand{\HP}{\H\P}
\newcommand{\CP}{\C\P}
\newcommand{\CcP}{\Cc\P}
\newcommand{\invers}{^{-1}}

\newcommand{\Hh}{\mathcal{H}}
\newcommand{\ii}{{\bf i\,}}

\DeclareMathOperator{\ord}{ord}
\DeclareMathOperator{\Gl}{GL}
\DeclareMathOperator{\SO}{SO}
\DeclareMathOperator{\Oo}{O}

\DeclareMathOperator{\im}{im}
\DeclareMathOperator{\Span}{span}
\DeclareMathOperator{\Id}{Id}

\DeclareMathOperator{\Gr}{Gr}
\DeclareMathOperator{\res}{res}

\newcommand{\zo}{^{(0,1)}}
\newcommand{\oz}{^{(1,0)}}

\newcommand{\trivial}[1]{\underline{\H}^{#1}}
\newcommand{\trivialC}[1]{\underline{\C}^{#1}}

\newcommand{\ttrivial}[1]{\widetilde{\underline{\H}}^{#1}}

\newcommand{\todoo}[1]{}
\newcommand{\done}[1]{}

\begin{document}
\title[Simple factor dressing and the L\'opez-Ros deformation]
{Simple factor dressing and the L\'opez-Ros deformation of
     minimal surfaces in Euclidean 3--space} \author{K. Leschke and K. Moriya}

 \address{K. Leschke, Department of Mathematics,
  University of Leicester, University Road, Leicester LE1 7RH, United
  Kingdom}
\address{K. Moriya, Division of Mathematics, Faculty of Pure and Applied Sciences, University of Tsukuba,
  1-1-1 Tennodai, Tsukuba-shi, Ibaraki-ken 305-8571, Japan}  
\email{k.leschke@le.ac.uk, moriya@math.tsukuba.ac.jp}

\thanks{First author  partially supported by DFG SPP 1154 ``Global
  Differential Geometry'' and JSPS KAKENHI  Grant-in-Aids for Scientific
  Research (C), Grant Number
  24540090. Both authors  supported by JSPS
  KAKENHI Grant-in-Aids for Scientific
  Research (C),  Grant Numbers 22540064 and 25400063}
\maketitle

\begin{abstract}
  The aim of this paper is to give a new link between integrable
  systems and minimal surface theory.  The dressing operation uses the
  associated family of flat connections of a harmonic map to construct
  new harmonic maps.  Since a minimal surface in 3--space is a
  Willmore surface, its conformal Gauss map is harmonic and a dressing
  on the conformal Gauss map can be defined. We study the induced
  transformation on minimal surfaces in the simplest case,
  the simple factor dressing, and show that the well--known
  L\'opez--Ros deformation of minimal surfaces is a special case of
  this transformation. We express the simple factor dressing and the
  L\'opez--Ros deformation explicitly in terms of the minimal surface
  and its conjugate surface.  In particular, we can control periods
  and end behaviour of the simple factor dressing. This allows to
  construct new examples of doubly--periodic minimal surfaces arising
  as simple factor dressings of Scherk's first surface.

\end{abstract}


\section{Introduction}

Minimal surfaces, that is, surfaces with vanishing mean curvature,
first implicitly appeared as solutions to the Euler-Lagrange equation
of the area functional in \cite{Lagrange} by Lagrange.  The classical
theory flourished through contributions of leading mathematicians
including, amongst others, Catalan, Bonnet, Serre, Riemann,
Weierstrass, Enneper, Schwarz and Plateau.  By now, the class of
minimal surfaces belongs to the best investigated and understood
classes in surface theory.  One of the reasons for the success of its
theory is the link to Complex Analysis: since a minimal conformal
immersion $f: M \to\R^3$ from a Riemann surface $M$ into 3--space is a
harmonic map, minimal surfaces are exactly the real parts holomorphic
curves $\Phi: M \to \Cc^3$ into complex 3--space. Due to the
conformality of $f$, the holomorphic  map $\Phi$ is a null curve  with
respect to the standard symmetric bilinear form on $\Cc^3$. A
particularly important aspect of this approach is that the
\emph{Enneper--Weierstrass representation} formula,  \cite{enneper},
\cite{weierstrass}, allows to construct all holomorphic null curves,
and thus all minimal surfaces, from the \emph{Weierstrass data} $(g,
\omega)$ where $g$ is a meromorphic function and $\omega$ a
holomorphic 1--form.  For details on the use of the holomorphic null
curve and the associated Enneper--Weierstrass representation as well
as historical background we refer the reader to standard works on
minimal surfaces, such as \cite{nitsche}, 
\cite{hoffman_karcher}, \cite{lopez_martin}, \cite{perez_ros},
\cite{dierkes}, \cite{meeks_perez_book}.

For the purposes of this paper, it is however useful to point out two obvious
ways to construct new minimal surfaces from a given minimal surface
$f: M \to \R^3$ and its holomorphic null curve $\Phi$: firstly,
multiplying $\Phi$ by $e^{-i\theta}$, $\theta\in\R$, one obtains the
\emph{associated family} of minimal surfaces $f_{\cos \theta, \sin\theta}
=\Re(e^{-i\theta} \Phi)$ as the real parts of the holomorphic null
curves $e^{-i\theta} \Phi$. The associated family of minimal surfaces
was introduced by  Bonnet, \cite{Bonnet}, in the study of surfaces
parametrised by a curvilinear coordinate.  An interesting feature of
the associated family is that it is an isometric deformation of
minimal surfaces which preserves the Gauss map.  The converse was
shown by  Schwarz, \cite{Schwarz}: if two simply-connected
minimal surfaces are isometric, then, by a suitable rigid motion, they
belong to the same associated family.

The second transformation, the so--called \emph{Goursat
  transformation} \cite{goursat}, is given by any orthogonal matrix
$\mathcal{A}\in \Oo(3, \Cc)$: since $\mathcal A$ preserves the
standard symmetric bilinear form on $\Cc^3$, the holomorphic map
$\mathcal{A}\Phi$ is a null curve, and $\Re(\mathcal{A}\Phi)$ is a
minimal surface in $\R^3$.  As pointed out by P\'erez and Ros,
\cite{perez_ros}, an interesting special case is known as the
\emph{L\'opez--Ros deformation}.  To show that any complete, embedded
genus zero minimal surface with finite total curvature is a catenoid
or a plane, L\'opez and Ros \cite{lopez_ros} used a deformation of the
Weierstrass data which preserves completeness and finite total
curvature.  This L\'opez-Ros deformation has been later used in
various aspects of minimal surface theory, e.g., in the study of
properness of complete embedded minimal surfaces,
\cite{meeks_perez_ros}, the discussion of symmetries of embedded genus
$k$--helicoids, \cite{bernstein_breiner}, and in an approach to the
Calabi-Yau problem, \cite{ferrer_martin_umehara_yamada}.

On the other hand, by the Ruh--Vilms theorem the Gauss map of a
minimal surface is a harmonic map $N: M \to S^2$ from a Riemann
surface $M$ into the 2-sphere, \cite{ruh_vilms}. Harmonic maps from
Riemann surfaces into compact Lie groups and symmetric spaces, or more
generally, between Riemannian manifolds, have been extensively studied
in the past.  Harmonic maps are critical points of the energy
functional and include a wide range of examples such as geodesics,
minimal surfaces, Gauss maps of surfaces with constant mean curvature
and classical solutions to non--linear sigma models in the physics of
elementary particles. Surveys on the remarkable progress in this topic may
be found in \cite{eells_lemaire1}, \cite{eells_lemaire2},
\cite{guest}, \cite{helein_wood}, \cite{ohnita}.

One of the big breakthroughs in the theory of harmonic maps was the
observation from theoretical physicists that a harmonic map equation
is an integrable system, \cite{pohlmeyer}, \cite{zakharov_mikhailov},
\cite{zakharov_shabat}: The harmonicity condition of a map from a
Riemann surface into a suitable space can be expressed as a
Maurer--Cartan equation. This equation allows to introduce the
spectral parameter to obtain the \emph{associated family of
  connections}. The condition for the map to be harmonic is then
expressed by the condition that every connection in the family is a
flat connection.  This way, the harmonic map equation can be
formulated as a Lax equation with parameter. Starting with the work of
Uhlenbeck \cite{uhlenbeck} integrable systems methods have been highly
successful in the geometric study of harmonic maps from Riemann
surfaces into suitable spaces, e.g., \cite{hitchin-harmonic},
\cite{Uhl92}, \cite{Annals}, \cite{duke}, \cite{DPW},
\cite{terng_uhlenbeck}. In particular, the theory can be used to
describe the moduli spaces of surface classes which are given
 in terms of a harmonicity condition, such as
constant mean curvature surfaces, e.g., \cite{pin&ster},
\cite{sym_bob}, isothermic surfaces, e.g.,
\cite{cieslinski1995isothermic}, \cite{bjpp}, \cite{fran_epos},
Hamiltonian Stationary Lagrangians, e.g., \cite{pascal_frederic},
\cite{hsl}, and Willmore surfaces, e.g., \cite{Helein1},
\cite{schmidt}, \cite{bohle}, \cite{burstall_quintino}.

We recall the methods of integrable systems which are relevant for our
paper: given a $\C_*$--family of flat connections $d_\lambda$ of the
appropriate form, one can construct a harmonic map from it. In
particular, the associated family $d_\lambda$ of flat connections of a
harmonic map gives an element of  the \emph{associated family of harmonic maps} by,
up to a gauge by a $d_\mu$--parallel endomorphism, using the family
$d_{\mu\lambda}$ for some fixed $\mu\in\C_*$. The
\emph{dressing operation} was introduced by Uhlenbeck and Terng, \cite{uhlenbeck},
\cite{terng_uhlenbeck}: as pointed out to us by Burstall, in the case
of a harmonic map $N: M \to S^2$  the dressing is given
by a gauge  $\hat d_\lambda = r_\lambda\cdot d_\lambda$ of   $d_\lambda$
by   a $\lambda$--dependent dressing matrix $r_\lambda$, \cite{simple_factor_dressing}. The
\emph{dressing} of $N$ is then the harmonic map $\hat N$ that has
$\hat d_\lambda$ as its associated family of flat connections. In general, it is hard to find explicit
dressing matrices and compute the resulting harmonic map. However, if
$r_\lambda$ has a simple pole $\mu\in\C_*$ and is given by a
$d_\mu$--parallel bundle, then the so--called \emph{simple factor
  dressing} can be computed explicitly,  e.g., \cite{terng_uhlenbeck},
\cite{dorfmeister_kilian}, \cite{simple_factor_dressing}. 

Parallel bundles of the associated family of flat connections also
play an important role in Hitchin's classification of harmonic tori in
terms of spectral data, \cite{hitchin-harmonic}, and in applications
of his methods to constant mean curvature and Willmore tori, e.g.,
\cite{pin&ster}, \cite{schmidt}.  The holonomy representation of the family
$d_\lambda$ with respect to a chosen base point on the torus is
abelian and hence has simultaneous eigenlines. From the corresponding
eigenvalues one can
define the \emph{spectral curve} $\Sigma$, a hyperelleptic curve over
$\CP^1$ (which is independent of the chosen base point), together with
a holomorphic line bundle over $\Sigma$, given by the eigenlines of
the holonomy (these depend on the base point, and sweep out a subtorus
of the Jacobian of the spectral curve). Conversely, the spectral data
can be used to construct the harmonic tori in terms of
theta--functions on the spectral curve $\Sigma$. This idea can be
extended to a more general notion \cite{taimanov_weierstrass}
of a spectral curve for conformal tori $f: T^2 \to
S^4$. Geometrically, this multiplier spectral curve arises as a
desingularisation of the set of all Darboux transforms of $f$ where
one uses a generalisation of the notion of Darboux transforms for
isothermic surfaces to conformal surfaces \cite{conformal_tori}.  In
the case when the conformal immersion is a constant mean curvature or
Willmore surface, one obtains as special cases the so--called
\emph{$\mu$--Darboux transforms} which are given by parallel sections
of the associated family of flat connections. In particular, the
(normalisations of the) eigenline spectral curve for the harmonic
Gauss map of a constant mean curvature torus $f$ is, \cite{cmc}, the
multiplier spectral curve of $f$. A similar result holds for
(constrained) Willmore surfaces, \cite{bohle}.
  
As mentioned above, by the Ruh--Vilms theorem the Gauss map of a
minimal immersion $f: M \to\R^3$ is harmonic, and thus, the various
operations discussed above can be applied to its Gauss map. However,
as opposed to the case of an immersion with constant non--vanishing
mean curvature, the Gauss map does not uniquely determine the minimal
surface. Thus, although the associated family and the dressing
operation for the harmonic Gauss map of a minimal surface can be
defined, \cite{dorfmeister_pedit_toda}, the investigation of minimal
surfaces with these dressed harmonic Gauss maps complicates.  On the
other hand, Meeks, P\'erez and Ros \cite{meeks_perez_ros2},
\cite{meeks_perez}, use algebro--geometric solutions to the KdV
equation to show that the only properly embedded minimal planar domains with
infinite topology are the Riemann minimal examples.  The same Lam\'e
potentials appear in the study of the spectral curve of an Euclidean
minimal torus with two planar ends and translational periods
\cite{bohle_taimanov}.  This indicates that applying integrable system
methods may lead to a further development of minimal surface
theory. Conversely, getting a better understanding of the special case
of minimal surfaces may also give insights into the more general
methods from integrable systems.  

The aim of our paper is to provide further evidence that concepts on
minimal surfaces may in fact be special cases of the harmonic map
theory: the L\'opez--Ros deformation is a special case of a simple
factor dressing of a minimal surface.

To avoid the issue that a minimal surface is not uniquely determined
by its Gauss map, we will work with the conformal Gauss map which
determines a minimal surface in 3--space uniquely.  Since minimal
surfaces are Willmore the conformal Gauss map is harmonic, too.  We
will briefly recall the construction of the associated families
$d_\lambda$ and $d^S_\lambda$ of flat connections for both the
harmonic Gauss map $N$ and the conformal Gauss map $S$ of a minimal
surface in our setup. Both are closely related: parallel sections of
$d^S_\lambda$ can be expressed in terms of parallel sections of
$d_\lambda$ and generalisations $f_{p,q}$, $p, q\in S^3$, of the
associated family of minimal surfaces $f_{\cos\theta, \sin\theta}$.
It turns out that this new family $f_{p,q}$, the
\emph{right--associated family}, is in fact a family of minimal
surfaces in 4--space which contains the classical associated
family. In view of this natural appearance of minimal surfaces in
4--space, we will develop our theory more generally for minimal
surfaces in 4--space and restrict to the case of minimal surfaces in
3--space when appropriate.  As in the case of a harmonic map $N: M \to
S^2$ one can define the \emph{associated family of harmonic maps} of
the harmonic conformal Gauss map of a Willmore surface, \cite{tokyo}.  In the
case of a minimal surface, we show that the harmonic maps in the
associated family of the conformal Gauss map are indeed the conformal
Gauss maps of the associated family of minimal surfaces.

Moreover, due to the harmonicity of the conformal Gauss map of a
Willmore surface, a dressing operation on Willmore surfaces can be
defined \cite{burstall_quintino}.  In particular, for the most simple
dressing operation given by a dressing matrix with a simple pole, the
so--called \emph{simple factor dressing}, the new harmonic map can be
computed explicitly and is the conformal Gauss map of a new Willmore
surface in the 4-sphere, \cite{willmore_harmonic},
\cite{burstall_quintino}.

In the case of a minimal surface $f: M \to \R^4\subset S^4$ we only
consider simple factor dressings which preserve the Euclidean
structure and show that in this case, the simple factor dressing of the
conformal Gauss map of $f$ is indeed the conformal Gauss map of a minimal
surface in 4--space. In fact, the simple factor dressing can be given
explicitly in terms of the minimal surface $f$, its conjugate and the
parameters $(\mu, m, n)$ where $\mu\in\C\setminus\{0\}$ is the pole of
the simple factor dressing, and $m, n \in S^3$ determine the
$d^S_\mu$--stable bundle which is needed in the definition of the
dressing matrix. Even for surfaces in 3--space, the simple factor
dressing will in general give surfaces in 4--space. However, for
$n=m$, the simple factor dressing of a minimal surface $f: M \to \R^3$
will be in 3--space and the Gauss map of a simple factor dressing is
the simple factor dressing of the Gauss map of $f$. In the simplest
case when $n=m =1$ and $\mu\in\R$ the simple factor dressing is the
minimal surface
\begin{equation}
\label{eq: sfd mu real}
f^\mu = \begin{pmatrix} f_1 \\ f_2 \cosh s - f_3^* \sinh s\\ f_3\cosh
  s + f_2^* \sinh s
\end{pmatrix}
\end{equation}
where $s =-\ln|\mu|$ and $f_l, f^*_l$ are the coordinate functions of
$f$ and a conjugate $f^*$ of $f$. In this case, we see immediately
that $f^\mu $ is a Goursat transformation of $f$ with holomorphic null
curve $\mathcal L^\mu\Phi$ where $\Phi = f+\ii f^*$ is the holomorphic
null curve of $f$ and
\[
\mathcal L^\mu = \begin{pmatrix} 1 & 0 & 0\\
0&\cosh s & \ii \sinh s\\ 0 & -\ii \sinh s & \cosh s
\end{pmatrix}\in \Oo(3,\Cc)\,.
\]
Indeed, we prove more generally that every simple factor dressing of a
minimal surface $f: M \to \R^4$ with parameters $(\mu, m, n)$ is a
Goursat transformation. In particular, we show that this implies that
the simple factor dressing preserves completeness. If the Goursat
transform is single--valued on $M$ then finite total curvature is
preserved, too.

In the case when $m=n\in S^3$, the orthogonal matrix of the Goursat
transformation is given as $\mathcal{R}_{m,m} \mathcal L^\mu
\mathcal{R}_{m,m}\invers\in \Oo(3, \Cc)$ where the rotation matrix
$\mathcal R_{m,m}$ in 3--space is given by $m\in S^3\subset \R^4$:
decomposing 
$
m
= (\cos \theta, q \sin\theta )
$
 with $q\in \R^3, ||q||=1$, the matrix
$\mathcal R_{m,m}$ is the rotation along the axis given by $q$ about
the angle $2\theta$. In other words, the simple factor dressing in
$\R^3$ with parameters $(\mu, m, m)$ is obtained from the simple factor
dressing (\ref{eq: sfd mu real}) with parameter $\mu$ applied to the
(inverse of the) rotation given by $m$.

The L\'opez--Ros deformation of a minimal surface $f: M \to \R^3$ is
usually given in terms of the Weierstrass data. We recall an explicit
form of the L\'opez--Ros deformation in terms of the minimal surface
and its conjugate surface: the L\'opez--Ros deformation $f_\sigma$ of
$f$ with parameter $\sigma \in\R, \sigma>0,$ is indeed given by
\[ 
f_\sigma= \begin{pmatrix} f_1 \cosh s - f_2^* \sinh s\\ f_2\cosh
  s + f_1^* \sinh s \\ f_3
\end{pmatrix}
\]
where $s= \ln \sigma$.  In other words, since
\[
f_\sigma = \Re(\mathcal R_{m,m}\
\mathcal L^\mu \mathcal R_{m,m} \invers\Phi)
\]
with $\mu = - \frac 1 \sigma$ and $m= \frac 12(1, -1, -1, -1)\in S^3$,
the Lopez--Ros deformation with parameter $\sigma$ is the simple
factor dressing of $f$ with parameters $(\mu, m, m)$.

We investigate the periods of the simple factor dressings in terms of
the periods of the holomorphic null curve, and give conditions on the
parameters $(\mu, m, n)$ for a simple factor dressing to be
single--valued.   We discuss the end behaviour of the simple factor
dressing on minimal surfaces in 3--space with finite total curvature
ends: the simple factor dressing preserves planar ends for all
parameters and, due to the special form of the Goursat transformation,
catenoidal ends if the parameters of the simple factor dressings are
chosen so that the simple factor dressing is single--valued.

Previous results seemed to indicate that the $\mu$--Darboux
transformation, which is used in the geometric understanding of the
spectral curve of conformal tori, preserves a surface class.  For
example, $\mu$--Darboux transforms of constant mean curvature surfaces
$f: M \to\R^3$, i.e., Darboux transforms which are given by parallel
sections of the flat connection $d_\mu$ in the associated family of
the Gauss map $N$ of $f$, have constant mean curvature \cite{cmc} and
similarly, $\mu$--Darboux transforms of Hamiltonian Stationary
Lagrangians are Hamiltonian Stationary Lagrangians,
\cite{hsl}. However, for minimal surfaces we show that $\mu$--Darboux
transforms are in general not minimal but are still given by complex
holomorphic data.  More precisely, a minimal surface has an
\emph{associated Willmore surface} which is the twistor projection of
a holomorphic curve in complex projective 3-space, and a
$\mu$--Darboux transform of $f$ is the associated Willmore surfaces of
an element of the right--associated family $f_{p,q}$ of $f$.  In case
of a minimal surface $f: M \to\R^3$ the associated Willmore surface
$f^\flat$ is the conformal immersion in 4--space which is given by
\[
f^\flat =\begin{pmatrix}
-<f, N> \\
f\times N - f^*
\end{pmatrix}\,,
\]
where $N$ is the Gauss map of $f$ and $f^*$ is a conjugate of $f$. 

We conclude the paper by demonstrating our results for various
well-known minimal surfaces, including the catenoid, surfaces with one
planar end and Scherk's first surface.

In particular, the simple factor dressings of the catenoid which are
again periodic are reparametrisations of the catenoid if they are
surfaces in 3--space. This immediately follows from our result that
the simple factor dressing of a catenoidal end is catenoidal, provided
the simple factor dressing is single--valued.  Since planar ends are
preserved for any parameters, all simple factor dressings of surfaces
with one planar end have one planar end, too.

Using our closing conditions, we show that the L\'opez--Ros
deformation of Scherk's first surface gives doubly--periodic minimal
surfaces. Moreover, for any rational number $q>0$ we show that the
simple factor dressing (\ref{eq: sfd mu real}) with parameter
$\mu=-\frac 1{\sqrt{q}}$ is doubly--periodic, thus we obtain a family
of new examples of doubly--periodic (non--embedded) minimal surfaces.

The authors would like to thank Wayne Rossman and Nick Schmitt for
directing their attention towards the L\'opez-Ros deformation and the
Goursat transformation. Parts of this research were conducted while
the first author was visiting the Department of Mathematics at the
University of Tsukuba and the OCAMI at Osaka University. The first
author would like to thank the members of both institutions for their
hospitality during her stay, and the University of Leicester for
granting her study leave.


 

\section{Minimal surfaces}
We first recall some basic facts on minimal surfaces in Euclidean
space which will be needed in the following whilst setting up our
notation.  Although we are mostly interested in minimal surfaces in
$\R^3$, some of our transforms will be surfaces in $\R^4$. Therefore,
we will study more generally minimal  immersions in $\R^4$ and
specialise to the case of minimal surfaces in 3--space when
appropriate.

\subsection{Minimal surfaces in $\R^4$}
Let $f: M \to \R^4$ be a conformal (branched)
immersion from a Riemann surface $M$ into 4--space. If $f$ is minimal,
then $f$ is harmonic, i.e.,
\[
d*df =0
\]
where we put $*\omega(X) = \omega(J_{TM}X)$ for a 1--form
$\omega\in\Omega^1(TM)$, $X\in TM$. Here, $J_{TM}$ is the complex
structure of the Riemann surface $M$, thus, $*$ is the negative Hodge
star operator. In particular, $*df$ is closed if  $f$ is harmonic and  there
exists a \emph{conjugate surface} $f^*$   on the universal cover
$\tilde M$ of $M$, given up to translation by
\[
df^* = -*df\,.
\]
Note that $f^*$ is minimal, and so is the \emph{associated family},
e.g. \cite{eisenhart},
\[
f_{\cos\theta,\sin\theta} = f\cos\theta + f^*\sin\theta, \quad
\theta\in \R\,.
\]

We model Euclidean 4--space by the quaternions $\R^4=\H$, and the
Euclidean 3--space in $\R^4$ by the
 imaginary quaternions $\R^3=\Im\H$.
The conformality of an immersion $f: M \to\R^4$ gives
\cite[p. 10]{coimbra} the \emph{left} and \emph{right normal} $N, R: M
\to S^2=\{n\in\Im\H \mid n^2=-1\}$ of $f$ by
\begin{equation}
\label{eq:left and right normal}
*df = N df = -df R\,.
\end{equation}
Then 
the \emph{mean curvature vector} $\Hh$ of $f: M \to\R^4$
satisfies \cite[p. 39]{coimbra}
\[
\bar \Hh df = \frac12(*dR + RdR), \quad \text{or, equivalently,} \quad df\bar\Hh = -\frac 12(*dN + NdN)\,.
\]
Since $\Hh$ is normal we have $N\Hh = \Hh R$.  We put $H = -R\bar\Hh$
and denote by
\[
(dR)' = \frac 12(dR - R*dR), \quad (dR)'' = \frac 12(dR + R*dR)
\]
the $(1,0)$ and $(0,1)$--part of $dR$ with respect to the complex
structure $R$. Then the equation of the mean curvature vector becomes 
\begin{equation}
\label{eq:H R}
Hdf    =(dR)'\,.
\end{equation}

Similarly, there is also an equation for the  mean curvature vector in terms of
the left normal:
\begin{equation}
\label{eq:H}
dfH = \frac 12(dN - N*dN) = (dN)'\,.
\end{equation}

Note that $f: M \to
\R^4$ is minimal if and only if 
\[
(dR)'=0\,, \quad \text{ or, equivalently, } \quad (dN)'=0\,.
\]

In other words, if $f$ is minimal
then 
\[
*dR = -RdR = dR R
\]
and $*dN = -NdN = dN N$ for the left and right normal of $f$. Thus,
both $N$ and $R$ are quaternionic holomorphic sections
\cite{klassiker} with respect to the induced quaternionic holomorphic
structures on the trivial $\H$ bundle $\trivial{} = M \times \H$. Note
also that a map $R: M\to S^2$ is harmonic if and only if 
\begin{equation}
\label{eq: N harmonic}
d(dR)'=0 \quad \text{or, equivalently, } \quad d(dR)''=0\,.
\end{equation}
In particular, both the left and right
normal $N$ and $R$ of a minimal surface are conformal and harmonic.

Next, we observe that a conjugate surface $f^*$ of a minimal
surface $f$ has the same left and right normal as $f$ since
\[
*df^* = -*(*df) = df = *df R= - df^*R\,,
\]
and similarly $*df^* = Ndf^*$.  
Since $f$ is harmonic and $*df =-df^*$, the map  
\[
\Phi= f+\ii f^*: \tilde M \to \Cc^4
\]
is a  holomorphic   curve in $\Cc^4$, that is, $*d\Phi = \ii d\Phi$.   Here we use ``$\ii$'' to denote  the complex structure of the
com\-plexifica\-tion $\Cc^4=\R^4 + \ii \, \R^4$ to avoid confusion
with the imaginary quaternion $i$. 
If $f: M \to\R^4$ is a minimal conformal immersion, then the  holomorphic   curve   
\[
\Phi =\begin{pmatrix} \Phi_0, \Phi_1, \Phi_2, \Phi_3 
\end{pmatrix}
\colon\tilde M  \to\Cc^4
\]
is a null curve in $\Cc^4$, i.e., $d\Phi_0\otimes d\Phi_0 +
d\Phi_1\otimes d\Phi_1+d\Phi_2\otimes d\Phi_2+d\Phi_3\otimes d\Phi_3
=0$. In fact, every holomorphic null curve $\Phi: \tilde M \to \Cc^3$
gives rise to a conformal minimal immersion by setting $f = \Re(\Phi):
\tilde M \to \R^3$.  Note that the holomorphic null curve of the
associated family $f_{\cos \theta, \sin\theta}$ of a minimal surface
$f$ is given by $\Phi_{\cos \theta, \sin \theta} = e^{-\ii \theta}
\Phi$ where $\Phi$ is the holomorphic null curve of $f$.

   The \emph{Weierstrass data} of $f$ is given by the
meromorphic functions
\[
g_1 = \frac{d\Phi_3}{d\Phi_1 - \ii  d\Phi_2},  \quad g_2 =\frac{d\Phi_0}{ d\Phi_1 - \ii  d\Phi_2}
\]
and the holomorphic 1--form
\[
\omega= d\Phi_1 - \ii  d\Phi_2\,.
\]

Conversely, let $g_1, g_2: M \to \Cc \cup \{\infty\}$ be meromorphic
functions  and $\omega$ a holomorphic 1--form. Assume that  if $m$ is
the maximum 
order 
of poles of $g_1, g_2,$ and $g_1^2+g_2^2$  at $p\in M$  then
$\omega$ has a 
zero at $p$ of order at least $m$. Then $(g_1, g_2, \omega)$ gives
rise \cite{hoffman_osserman}  to a holomorphic null
curve $\Phi$, and thus a minimal surface
$f=\Re(\Phi): \tilde M \to \R^4$, via 
\[
\Phi = \int \left( g_2 \omega
,  \frac 12(1-g_1^2 -g_2^2)\omega, \frac
  {\ii }2(1+g_1^2+g_2^2)\omega , g_1\omega\right)\,.
\]
Our choice of Enneper--Weierstrass representation is so that it
specialises to the standard Enneper--Weierstrass representation in $\R^3$
whenever $f: M\to \H$ is a minimal surface with $\Re(f) =0$.
We allow $f$ to be branched which happens whenever the order of
$\omega$ at $p$ is bigger than the maximum order of poles of $g_1,
g_2,$ and $g_1^2+g_2^2$ at $p\in M$. Note that $f$ is in general only
defined on the universal cover $\tilde M$ of $M$. We say that a
minimal surface $f: \tilde M \to \R^4$ is single--valued on $M$ if
$f\circ \pi\invers: M \to \R^4$ is well--defined where $\pi: \tilde M
\to M$ is the canonical projection of the universal cover $\tilde M$
to $M$. In this case, we will identify $f$ and $f\circ \pi\invers$ and
write, in abuse of notation, from now on $f: M \to\R^4$.

The Gauss map $G: M \to \Gr_2(\R^4)$ of $f$ is a map into the
Grassmannian of oriented two planes in $\R^4$. In our case, it is given by $\varphi
dz=d\Phi$ where the two-plane $G(p)$ in
$\R^4$ is spanned by $\Re \varphi, \Im \varphi$ at $p$ since
$d\Phi = df +\ii df^* = df - \ii *df$.    
But $
\Gr_2(\R^4) = \CcP^1\times \CcP^1$ and  the Gauss map $G$ can be identified
\cite{osserman} with the two meromorphic functions $G_1, G_2:
M \to \Cc$:
\[
G_1 = \frac{\varphi_3-\ii\varphi_0}{\varphi_1-\ii\varphi_2} = g_1-\ii g_2,
\quad
G_2 = \frac{\varphi_3+\ii\varphi_0}{\varphi_1- \ii\varphi_2} = g_1+\ii g_2\,.
\]
Indeed, stereographic projections of $G_1$ and $G_2$ give the left and
right normals $N$ and $R$ by 
\[
N =  \frac 1{1+|G_1|^2}(2\, \Re G_1, 2\, \Im G_1, |G_1|^2-1), \quad
R=  \frac 1{1+|G_2|^2}(2\, \Re G_2, 2\, \Im G_2, |G_2|^2-1)\,.
\]
To  verify this, we  note that $\Phi$ can be expressed in terms of $G_1, G_2, \omega$ as
\[
d\Phi= \varphi dz= \frac\omega 2\big(\ii (G_1 -G_2), 1-G_1G_2, \ii (1+ G_1G_2),
G_1+G_2\big)\,.
\]
Since $\varphi = f_x - \ii f_y$ in the conformal coordinate $z=x+iy$, 
we obtain $f_x = \frac 12(\varphi +\bar\varphi), f_y =  \frac
\ii 2(\varphi-\bar\varphi)$, and it is a straight forward computation to
verify that $N = \frac 1{1+|G_1|^2}(2\, \Re G_1, 2\, \Im G_1,
|G_1|^2-1)$ and
$R= \frac 1{1+|G_2|^2}(2\, \Re G_2, 2\, \Im G_2, |G_2|^2-1)$ satisfy
$Nf_x = -f_x R = f_y$ which are the defining equations
(\ref{eq:left and right normal})  of the left and right
normal of $f$.

Finally we recall the following result due to  Chern and Osserman:
\begin{theorem}[\cite{chern_osserman}, \cite{moriya_minimal}]
\label{thm:FTC}
Let $f: M \to\R^4$  be a complete (branched) minimal immersion with holomorphic
null curve $\Phi: M \to \Cc^4$. 

Then $f$
has finite total curvature if and only if $M$ is
conformally equivalent to a compact Riemann surface $\bar M$ punctured
at finitely many points $p_1, \ldots, p_r$ such that $d\Phi$ extends
meromorphically into punctures $p_i$.  
\end{theorem}

\subsection{Minimal surfaces in $\R^3$}

We identify the Euclidean 3--space with the imaginary quaternions
$\R^3 =\Im\H$. If 
 $f: M \to \R^3$ is conformal then the left and right normal coincide
 and are given by the Gauss map $N: M\to S^2$ of $f$. In this case, the function $H$
given in (\ref{eq:H}) is  real--valued, and indeed, $H$ is
the mean curvature function of $f$.   For a minimal
 immersion in $\R^3$,
the Gauss map is thus both harmonic and conformal, that is, 
\[
*dN = -NdN\,, \quad \text{ and } \quad d(dN)'=0\,,
\]
where as before $(dN)' = \frac 12(dN - N*dN)$ is the $(1,0)$--part of
$dN$ with respect to the complex structure $N$.  

The holomorphic null curve 
\[
\Phi = f + \ii f^*: \tilde M \to \Cc^3
\]
gives the Weierstrass data $(g, \omega)$ of $f$ as 
\[
\omega = d\Phi_1-\ii  d\Phi_2
\quad \text{
and
}
\quad
g= \frac{d\Phi_3}{d\Phi_1-\ii  d\Phi_2}\,,
\]
where $\Phi_l$ are the coordinates of $\Phi=(\Phi_1, \Phi_2, \Phi_3)$. 

Conversely, a meromorphic $g$ and a holomorphic 1--form $\omega$,  such
that  if $g$ has a pole of order $m$ at $p$ then $\omega$ has a zero of
order at least $2m$,   give a
minimal surface $f: \tilde M \to \R^3$ as $f = \Re(\Phi)$ where $\Phi$ is
given by the Enneper--Weierstrass representation,   \cite{enneper},
\cite{weierstrass}, 
\[
\Phi = \int \left( \frac 12(1-g^2)\omega, \frac {\ii }2(1+g^2)\omega, g \omega\right)\,.
\]
 For convenience, we also will occasionally use in the case of a surface in $\R^3$
 the Weierstrass data $(g, dh)$  given in terms of the height differential
 $dh = g \omega$. Written in terms of $(g, dh)$ the Enneper--Weierstrass
 representation becomes
\begin{equation}
\label{eq:Weierstrass representation}
\Phi = \int \left( \frac 12(\frac 1g-g), \frac {\ii }2(\frac 1g+g),
  1\right)dh\,.
\end{equation}
The Gaussian curvature of $f$
is given in terms of the Weierstrass data \cite{karcher} as
\[
K = -\left(\frac{2}{|g|+\frac 1{|g|}}\right)^4 \left|\frac{d\log g}{dh}\right|^2\,.
\]

and the Gauss map is given by $g$ via stereographic projection
\begin{equation}
\label{eq:Gauss with g}
N =  \frac1{|g|^2+1}(2 \, \Re g, 2 \, \Im g, |g|^2-1)\,.
\end{equation}

Note that if we consider a minimal surface $f: M\to\R^3$ as
a surface in 4-space, that is $f: M \to\H$ with $\Re(f) =0$, then its
Weierstrass data in $\R^4$ is under our choices given by $(g_1=g,
g_2=0, \omega)$ and its holomorphic null curve $\int \left(0, \frac
  12(1-g^2)\omega, \frac {\ii }2(1+g^2)\omega, g \omega\right)$ in
$\Cc^4$ is given by the embedding of the curve $\Phi$ into
$\Cc^4$. Moreover, we see that in this case the Gauss map $G=(G_1,
G_2): M \to S^2\times S^2$ takes values in the diagonal of $S^2\times
S^2$  and is given by $G_1=G_2 =g$  which is by
\cite{hoffman_osserman}, \cite{taimanov_2dim_Dirac} the condition for the Enneper--Weierstrass
representation to take values in 3-space.

As before, complete minimal immersions of finite total curvature can
be characterised by the holomorphic null curve $\Phi$:
\begin{theorem}[\cite{osserman}]
\label{thm:FTC in 3space}
Let $f: M \to\R^3$  be a complete (branched) minimal immersion with holomorphic
null curve $\Phi: M \to \Cc^3$. 

Then $f$
has finite total curvature if and only if $M$ is
conformally equivalent to a compact Riemann surface $\bar M$ punctured
at finitely many points $p_1, \ldots, p_r$ such that $d\Phi$ extends
meromorphically into punctures $p_i$.  
\end{theorem}

We will now give a description of embedded finite total curvature ends
in terms of the holomorphic null curve. Although the result seems to be
known for vertical ends, we include the  argument for completeness.
\begin{theorem}
\label{thm:FTCend}
Let $f: M \to\R^3$ be a minimal surface with complete end at $p$.  Let
$z$ be a conformal coordinate of $M$ at the end $p$ which is defined
on a punctured disc $D_* = D\setminus\{0\}$ and is centered at $p$.

Then the following statements are equivalent:

\begin{enumerate}
\item $f$ has an embedded finite total curvature end at $p$.
\item $d\Phi$ has order $-2$  at $z=0$ and  $\res_{z=0}d \Phi $ is real. 
\end{enumerate}
If $\res_{z=0}d \Phi =0$ then the end is \emph{planar}, otherwise, it
is \emph{catenoidal}. \\

Here, $\Phi =(\Phi_1, \Phi_2, \Phi_3)$ is the holomorphic null curve
of $f$ and $\ord_{z=0}d\Phi$ is the minimum of $\ord_{z=0} d\Phi_i$
for $i=1,2,3$.

\end{theorem} 
\begin{proof} 
 
We first assume that the end is vertical.  By \cite{hoffman_karcher} an embedded complete finite total
curvature vertical end has logarithmic growth $\alpha\in\R$ satisfying
$\res_{z=0} d\Phi = -(0, 0, 2\pi \alpha)$. The end is planar if
$\alpha=0$ and catenoidal otherwise.

 Moreover,
if $f$ has a catenoidal end then the
Gauss map $g$ of $f$ has a simple pole or zero, and the height differential $dh$ has a simple
pole. If $f$ has a planar end then $g$ has a pole or zero of order $m>1$
and $dh$ has a zero of order $m-2$.  In both cases,
(\ref{eq:Weierstrass representation}) shows that $\ord_{z=0} d\Phi
=-2$.

If the end is not vertical, we can apply a rotation $\mathcal R\in
\SO(3, \R)$ on $f$ to
obtain a min\-i\-mal surface $\tilde f = \mathcal R f$ with a vertical end. Then
the holomorphic null curve of $\tilde f$ is $\tilde \Phi =\mathcal R \Phi$ and
thus,   $\ord_{z=0}
d\Phi =\ord_{z=0}
d\tilde\Phi =-2$.  
Moreover,  the residues at $z=0$ vanish
at a planar end, whereas $\res_{z=0}d\Phi = \mathcal R\invers \res_{z=0}d\tilde
\Phi $ is real   for a
catenoidal end.
 
Conversely, we will show that if $\ord_{z=0} d\Phi=-2$ and
$\res_{z=0}d\Phi$ is real, then we can assume $f$ has an embedded
vertical end at $p$ and its Gauss map $g$ has a pole or zero of order
$m\ge 1$ and $\ord_{z=0} dh = m-2$ holds for the height differential.
From this we conclude that the end has finite total curvature.

If $\res_{z=0}d\Phi =(0,0,0)$ then we can rotate the end so that we
have a vertical end without changing the order and  the residues of
$d\Phi$ at the end.  In particular, we can assume that the Gauss map
$g$ has a pole or zero at $p$ of order $m\ge 1$.  Since
the residue of $d\Phi$ at $z=0$ vanishes, we see that $dh$ cannot have a simple pole.
But then $\ord_{z=0} d\Phi=-2$ implies that $dh$ is holomorphic with
$\ord_{z=0} dh=m-2$, $m\ge 2$, so that with the Enneper--Weierstrass
representation (\ref{eq:Weierstrass
  representation})
\begin{equation}
 \label{eq:asymptotic planar}
d\Phi = \begin{pmatrix} \frac{c_{-2}}{z^2} + b_0 + \ldots\\[4pt]
i\frac{c_{-2}} {z^2} + c_0 + \ldots\\[4pt]
a_{m-2} z^{m-2} + a_{m-1} z^{m-1} +\ldots
\end{pmatrix}
dz\,, \qquad c_{-2}, a_{m-2}\not=0\,.
\end{equation}
After possible translation, we thus have 
\[
\Phi(z) = \begin{pmatrix}
-\frac{c_{-2}}z + \ldots\\[4pt]
-i\frac{c_{-2}}z + \ldots\\[4pt] \frac{a_{m-2} }{m-1}z^{m-1} + \ldots\end{pmatrix}\,.
\]
Consider the set $C_r=\{z\in D_* \mid ||f(z)||=r\}$ where $f=\Re(\Phi)$.
Since the end is at $z=0$, we have that $z\in C_r$ tends to $0$  for
$r\to \infty$. In particular, the asymptotic behaviour of $\psi_r =
\frac{f}{r}: C_r \to S^2$ is
governed by 
\[
\psi_r \sim \frac{|z|}{|c_{-2}|}\begin{pmatrix}
  \Re(-\frac{c_{-2}}z)\\ \Im(\frac{c_{-2}}z)\\ 0
\end{pmatrix}\,.
\]

From \cite{jorge_meeks} we know that $\psi_r$ converges to a
horizontal circle with multiplicity when $r\to \infty$ and that the end is
embedded if the multiplicity is one. From the
above asymptotic behaviour we see that the multiplicity of
$\lim_{r\to\infty} \psi_r$ is indeed one. Hence, the end is embedded
but then (\ref{eq:asymptotic planar}) shows that $p$ is a planar end.

If the order of $d\Phi=-2$ and $\res_{z=0} d\Phi\not=0$ is real then we can rotate the end so that
$\res_{z=0} d\Phi=-(0,0,2\pi\alpha)$ for some $\alpha\in\R_*=\R\setminus\{0\}$.
Then the order of $dh$  is either $-1$ or $-2$.

We first show that the end is vertical by contradiction:  if g has
neither a zero nor a pole at the end, then $\ord_{z=0}gdh= \ord_{z=0}
\frac{dh}g= \ord_{z=0} dh$. Using $\ord_{z=0} d\Phi=-2$ and the
Enneper--Weierstrass representation  (\ref{eq:Weierstrass representation}), the order of $dh$ at
the end is $-2$.  We write
\[
dh=(\frac{a_{-2}}{z^2}-\frac\alpha z+ a_0+a_1z+\ldots) dz\,, \quad
g = b_0+b_1z+\ldots 
\]
where $a_{-2}, b_0\not=0$. 
Since 
\[
gdh=(\frac{a_{-2}b_0}{z^2}+\frac{-\alpha b_0+a_{-2}b_1}z+...) dz\,,
\quad 
\frac{dh}g=(\frac{a_{-2}}{b_0z^2}-\frac{\alpha b_0+ a_{-2}b_1}{b_0^2z}+...) dz
\]
and $\res_{z=0} d\Phi_1=\res_{z=0} d\Phi_2=0$, we have
\[
(-\alpha b_0+a_{-2}b_1) + \frac{\alpha b_0+ a_{-2}b_1}{b_0^2}=0\,, \quad 
(-\alpha b_0+a_{-2}b_1) - \frac{\alpha b_0+ a_{-2}b_1}{b_0^2}=0\,.
\]
Hence
$\alpha b_0=0$ which contradicts $\alpha \neq 0$ and $b_0\neq 0$. 
Therefore,  $g$ has a  zero or a pole at the end, but then   the end has
vertical normal, the zero or pole of $g$ is simple and $dh$ has a
simple pole at $z=0$, that is,  the holomorphic null
curve $\Phi$ can be written as 
\[
d\Phi = \begin{pmatrix} \frac{c_{-2}}{z^2} + b_0 + \ldots\\[4pt]
i\frac{c_{-2}} {z^2} + c_0 + \ldots\\[4pt]
\frac{a_{-1}} z + a_{0} + \ldots
\end{pmatrix}
dz\,, \qquad c_{-2}, a_{-1}\not=0\,.
\]
Since $\res_{z=0} d\Phi = (0, 0, -2\pi \alpha)$   we have $a_{-1} =
-2\pi \alpha\in\R$ and, 
after possibly translation, 
\[
f = \Re \Phi =  \begin{pmatrix}
-\Re \big(\frac{c_{-2}}z+ \ldots\big)\\[4pt]
\Im \big(\frac{c_{-2}}z + \ldots\big)\\[4pt]
a_{-1} \log|z| + \Re(a_0 z) +
\ldots)
\end{pmatrix}\,.
\]
Let as before $C_r=\{z\in D_* \mid ||f(z)||=r\}$   then
the asymptotic behaviour of $\psi_r =
\frac{f}{r}: C_r \to S^2$ is
determined by
\[
\psi_r \sim \frac{1}{\sqrt{\left(\frac{|c_{-2}|}{|z|}\right)^2
  +(a_{-1}\log|z|)^2}}\begin{pmatrix}
 - \Re(\frac{c_{-2}}z)\\[4pt] \Im(\frac{c_{-2}}z)\\[4pt] a_{-1}\log|z|
\end{pmatrix}\,.
\]
 
Since the multiplicity does not depend on $|z|$ we see that in the
limit $r\to \infty$ the multiplicity is
again 1, and the end is embedded by \cite{jorge_meeks}.
   Following the arguments in the proof of 
\cite[Prop. 2.1]{hoffman_karcher}, $f$ is a graph over (the exterior
of a bounded domain of) the $(f_1,
f_2)$ plane with   asymptotic
behaviour 
\[
f_3(f_1, f_2) = \alpha\log R + \beta + R^{-2}(\gamma_1 f_1 +
\gamma_2 f_2) + O(R^{-2})
\]
for $R = \sqrt{f_1^2   + f_2^2}$, and  the end has finite
total curvature, \cite{schoen} .
 
\end{proof}

The \emph{L\'opez-Ros deformation} of a minimal surface   $f: M \to
\R^3$ with Weierstrass data $(g, \omega)$ is \cite{lopez_ros} the
minimal surface $f_r: \tilde M \to \R^3$ given by the new Weierstrass data
$(r g, \frac\omega r)$ with $r\in\R_* $.  Obviously
this can be extended to a deformation $f_\sigma$ with complex parameter
$\sigma\in\Cc_*=\Cc\setminus\{0\}$ by using the Weierstrass data
$(\sigma g,
\frac\omega \sigma)$.

Since any matrix in $\Oo(3,\Cc)=\{ \mathcal{A} \in \rm{GL}(3, \Cc)
\mid \mathcal{A}^t = \mathcal{A}\invers\}$ preserves the standard
symmetric bilinear form on $\Cc^3$, we obtain new minimal surfaces via
the \emph{Goursat transformation}, \cite{goursat}: if $f: M\to\R^3$ is
minimal with holomorphic null curve $\Phi = f+\ii f^*$ then
$\mathcal{A}\Phi$ is again a holomorphic null curve and
$\Re(\mathcal{A} \Phi)$ is a minimal surface in $\R^3$.

As pointed out  by  P\'erez and Ros  \cite{perez_ros}, the L\'opez-Ros deformation
  is a special case of the Goursat transformation:

\begin{theorem}
\label{thm:Goursat}
The L\'opez-Ros deformation $f_\sigma$ with complex 
parameter $\sigma=e^{s+\ii t}\in\Cc_*$ is given  by
\[
f_\sigma = \begin{pmatrix}
\cos t \, (f_1\cosh s- f_2^*\sinh s) - \sin  t \, (f_2 \cosh s + f_1^* \sinh s)  \\
\sin t \, (f_1\cosh s - f_2^*\sinh s) + \cos t (f_2\cosh s + f_1^* \sinh s) \\
f_3
\end{pmatrix}\,.
\]
\end{theorem}
\begin{proof}
Let $\Phi = f + \ii f^*$ be the
holomorphic null curve of $f$.  Putting
\[
\tilde f = \begin{pmatrix}   \cos t & -\sin t  & 0 \\  \sin t&
  \cos t & 0\\ 0 &0&1
\end{pmatrix} \begin{pmatrix}
f_1\cosh s - f_2^*\sinh s \\ f_2\cosh s+ f_1^*\sinh s \\ f_3
\end{pmatrix} \,,
\]
a straightforward computation  shows that $\tilde f = \Re(
\mathcal{L}_\sigma\Phi)$ where the holomorphic map $\mathcal{L}_\sigma\Phi$,  
 \[
\mathcal{L}_\sigma =  \begin{pmatrix}  
 \frac 12(\sigma+ \frac 1{\sigma}) & \frac i 2(\sigma - \frac
  1{\sigma}) & 0\\[.1cm]
 -\frac i2(\sigma- \frac 1{\sigma}) & \frac 12 (\sigma+ \frac
 1{\sigma}) & 0\\
0&0&1
\end{pmatrix} \in \Oo(3, \Cc)\,,
\]
is a null curve in $\Cc^3$.  Thus, $\tilde f$ is a Goursat
transformation of $f$. Indeed, if $(g, \omega)$ denotes the
Weierstrass data of $f$ then the Weierstrass data of $\tilde f$
computes to
\[
\tilde \omega = d\tilde \Phi_1 -\ii d\tilde\Phi_2 = \frac{\omega}{\sigma} 
\]
and
\[
\tilde g = \frac{d\tilde\Phi_3}{d\tilde \Phi_1 -\ii d\tilde \Phi_2}
=\sigma
{g}\,.
\]
This shows that $\tilde f=f_\sigma$ is the L\'opez--Ros deformation
of $f$ with
parameter $\sigma\in\Cc_*$.
\end{proof}

\subsection{Willmore surfaces}

Using the one--point compactification of $\R^4$ we  consider a
conformal immersion $f: M \to \R^4$ as a conformal immersion into the
4--sphere. We identify the 4-sphere $S^4=\HP^1$ with the quaternionic
projective line
where the oriented M\"obius transformations are given by
$\Gl(2,\H)$. In particular, a map $f: M \to\HP^1$ can be identified with a line
subbundle $L\subset \trivial{2} = M \times \H^2$ of the trivial $\H^2$ bundle over
$M$ whose fibers  at $p\in M$ are given by 
\[
L_p= f(p)\,.
\]

For an immersion $f: M \to \R^4$ the line bundle $L$ is given by
\[
L = \psi \H\,, \quad \text{ where } \quad \psi =\begin{pmatrix} f\\ 1
\end{pmatrix}\,,
\]
when choosing the point at infinity as $\infty = \begin{pmatrix} 1 \\ 0
\end{pmatrix}\H\in\HP^1$.
Oriented M\"obius transformations on $\R^4=\H$ are given by
\[
v \mapsto
(av+b)(cv+d)\invers \quad\text{with} \quad\begin{pmatrix} a & b \\ c&d
\end{pmatrix} \in\Gl(2,\H).
\]
In particular, the group of oriented M\"obius transformations acts on the line bundle $L$ given
by an immersion $f: M \to \R^4$ via $L \mapsto B L, B\in\Gl(2,\H)$.
Every pair of unit quaternions $m, n\in S^3$ gives an element $\mathcal{R}_{m,n} \in
\SO(4,\R)$ by
\[
v\in\H \mapsto \mathcal{R}_{m,n}v= m v n\invers \in\H\,,
\]
and conversely, every element of the special orthogonal group arises
this way, \cite{cayley}.  The corresponding action on the line bundle $L$ of an
immersion $f: M \to \R^4$ is given by
\[
L \mapsto \begin{pmatrix} m&0\\ 0& n
\end{pmatrix} L, \quad m, n\in S^3\,.
\]

\begin{definition}[{\cite[p. 27]{coimbra}}]
  The \emph{conformal Gauss map} of a conformal immersion $f: M \to S^4$ is
  the unique complex structure $S$ on $\trivial 2$ such that $S$ and
  $dS$ stabilise the line bundle $L$ of $f$ and its Hopf field $A$ is
  a 1--form with values in $L$. 

  Here, the \emph{Hopf field} $A$  of $S$ is the 1--form given by
\[ 
A =  \frac 14 (*dS +SdS) = \frac 12(*dS)'
\]
where $(dS)' =\frac 12(dS - S*dS)$ is the $(1,0)$--part of the
derivative of $S$ with respect to the complex structure $S$.
\end{definition}

In affine coordinates, the conformal Gauss map of a conformal immersion $f: M \to\R^4$ 
is given 
by the complex structure,  see \cite[p. 42]{coimbra}, 
\begin{equation}
\label{eq:conformal Gauss map}
S = G
\begin{pmatrix} N &0 \\ -H& -R
\end{pmatrix} G\invers, \quad G = \begin{pmatrix} 1 & f\\ 0 &1
\end{pmatrix}
\end{equation}
on the trivial bundle $\trivial 2$ where $N, R$ are the
left and right normal of $f$ and $H = -R\bar\Hh$ with mean curvature
vector $\Hh$.  Thus,  we see that
 $S\psi = -\psi R$ and 
\[
(dS) \psi =  \psi(-dR + H  df)\,.
\]
Therefore, $S$ and $dS$ indeed stabilise the line bundle $L=\psi\H$ of $f$. In
affine coordinates, the Hopf field computes \cite[Prop 12, p. 42]{coimbra} to
\begin{equation}
\label{eq: Hopf}
2*A = G\begin{pmatrix} 0& 0 \\ \omega & \tau
\end{pmatrix}
 G\invers
\end{equation}
with $2\omega = dH + H*dfH + R*dH - H*dN$ satisfying 
\[
*\omega= -\omega
N + (dR)'' H
\]
 and 
\[
\tau= (dR)'' =\frac12(dR + R*dR)\,.
\]
Thus, $A$ is indeed a 1--form with values in $L$. Note that the Hopf
field $A$ is holomorphic \cite[p. 68]{coimbra} with $\im A \subset
L$. In particular, if $A
\not\equiv 0$, the Hopf field $A$ gives the line bundle $L$ by
holomorphically extending $\im A$ into the isolated zeros of
$A$. Therefore, when fixing the point at $\infty$, the immersion $f: M
\to\R^4$ is uniquely determined by $A$.

\begin{cor} 
 A conformal immersion $f: M \to\R^4$ is minimal if and only if the
  Hopf field $A$ of the conformal Gauss map $S$ of $f$ satisfies
\begin{equation}
\label{eq:A of minimal}
2*A = G \begin{pmatrix} 0 &0 \\ 0 & dR
\end{pmatrix}
G\invers
\end{equation}
with $G = \begin{pmatrix} 1&f\\ 0&1
\end{pmatrix}
$.  
\end{cor}

In particular, if $f: M\to\R^4$ is minimal then $*dR = - RdR$ and thus
$df\wedge dR =0$ by type arguments. Therefore, the Hopf field is
harmonic, that is, $d*A =0$.

\begin{theorem}[\cite{Ejiri}, \cite{rigoli}]
  Let $f: M \to S^4$ be a conformal immersion with conformal Gauss
  map $S$ and Hopf field $A$. Then $f$ is Willmore if and only if $S$
  is harmonic, that is, if and only if
\[
d*A =0\,.
\]
\end{theorem}

From this characterisation of Willmore  surfaces and the fact that  
the conformal Gauss map of a minimal surface is harmonic we see:
\begin{cor}
Every minimal immersion in $\R^4$ is a Willmore surface in $\R^4$.
\end{cor}

Consider the right normal $R$ of a minimal surface $f: M \to\R^4$
which satisfies $*d R = -RdR = dRR$. Thus, if the right normal is not
constant, we can consider $R: M \to S^2$ as a (branched) conformal
immersion whose right normal is $-R$. Since $*d(-R) = (-R)d(-R)$, we
see from (\ref{eq: Hopf}) that its Hopf field vanishes on the line
bundle of $R$, that is, $R$ is the twistor projection of a holomorphic
curve in $\CP^3$:

\begin{theorem}[{\cite[Thm 4, p. 47]{coimbra}}]
  \label{thm: twistor} Let $f: M \to\R^4$ be a conformal immersion
  with conformal Gauss map $S$ and Hopf field $A$. Then $f$ is the
  twistor projection of a holomorphic curve $F: M \to\CP^3$ if and
  only if $A|_L =0$ where $L$ is the line bundle of $f$.

  In this case, $f$ is Willmore and the twistor lift of $F$ is the
  holomorphic curve $F: M \to \CP^3$ which is given by the line
  subbundle $E\subset L$ by
\[
F(p) = E_p\,,
\]
where $E$ is the $+i$ eigenspace of the conformal Gauss map $S|_L$ restricted to $L$.

\end{theorem}

 The surfaces which are both minimal and twistor projections are
 indeed given by holomorphic maps into $\C^2$:

\begin{cor}
Let $f: M \to \R^4$ be a conformal immersion with right normal $R$.

Then $R$ is constant if and only if $f$ is minimal and the twistor
projection of a holomorphic curve in $\CP^3$. 

In this case, we can identify $\R^4 =\H$ with $\C^2$ via the complex
structure which is given by right multiplication by the constant
$-R$. Then $f: M \to \C^2$ is  holomorphic.
\end{cor}
\begin{proof}
If $dR =0$ then we see by (\ref{eq:H R}) that $H=0$ and $f$ is
minimal.  But then (\ref{eq:A of minimal})  shows that the Hopf field
$A$ vanishes identically, that is, $f$ is the twistor projection of a
holomorphic curve. 

Conversely, if $f$ is minimal and $A|_L =0$ we see again by (\ref{eq:A
  of minimal}) that $dR =0$. 

If $R$ is constant, then the right multiplication by $-R$ gives a
complex structure on $\R^4=\H$.  Then $*df = -df R$ shows that $f: M
\to \C^2$ is holomorphic when identifying  $\C^2 = (\H, -R)$.
\end{proof}

Note that if $f$ is a minimal surface with constant right normal $R$
then $df^* = - *df = dfR$ shows that $f^* = fR
+ c$ with $c\in\H$. For general minimal surfaces $f$ the map
$fR -f^*$ is a Willmore surface in $\R^4$:

\begin{theorem}[\cite{moriya}, \cite{dajczer}]
\label{thm:twistor minimal}
If $f: M \to\R^4$ is a minimal surface with conjugate surface
$f^*$ and (non--constant) right normal $R$ then 
\[
f^\flat= fR  -f^*\,,
\]
is a twistor projection of a holomorphic curve in $\CP^3$. We call $f^\flat$
an \emph{associated Willmore surface} of $f$. 

Conversely, let $f^\flat: M \to \R^4$ be a twistor projection of a
holomorphic curve in $\CP^3$ on the simply connected $M$ and let
$R^\flat$ be its non--constant right normal.  Away from the isolated
zeros of $dR^\flat$, the immersion $f^\flat$ is, up to translation,
given by either $f^\flat=cR^\flat, c\in\H,$ or by a minimal surface
$f: M \to \R^4$ via
\[
f^\flat=-fR^\flat -f^*
\]
where  $f^*$ is a conjugate surface of
$f$.
\end{theorem}

\begin{proof}
If $f$ is minimal then its right normal $R$ satisfies $*dR
= -RdR = dRR$ and
\[
d(fR -f^*) = fdR
\] 
shows that the right normal of $f^\flat=fR-f^*$ is $R^\flat= -R$. But then
\[
dR^\flat + R^\flat*dR^\flat = -dR + R*dR=0
\]
 and thus
$f^\flat$ is a twistor projection of a holomorphic curve in $\CP^3$ by
Theorem~\ref{thm: twistor} and (\ref{eq: Hopf}).

Conversely, let $f^\flat: M \to\R^4$ be the twistor projection of a
holomorphic curve in $\CP^3$, and let $R^\flat$ be its right normal. From
Theorem \ref{thm: twistor} and (\ref{eq: Hopf}) we see that
\[
*dR^\flat = R^\flat dR^\flat = -dR^\flat R^\flat\,.
\]
Since $R^\flat$ is conformal and not constant, the zeros of $dR^\flat$ are isolated.
Since $*df^\flat = -df^\flat R^\flat$ we conclude that there is a function $f: M
\to\R^4$ such that $df^\flat= -f dR^\flat$ away from the zeros of $dR^\flat$. If $f$ is not constant then
$0=df\wedge dR^\flat$ shows that $*df = df R^\flat$ and $f$ is a
(branched)  
conformal immersion with right normal $R= -R^\flat$. But then $f$ is minimal since
$*dR = -R dR$. Now, for a conjugate surface $f^*$ of $f$ the map
\[
\tilde f = -fR^\flat - f^*
\]
satisfies $d\tilde f = dfR - fdR^\flat + *df = -fdR^\flat = df^\flat$ so that $f^\flat=-fR^\flat-f^*$
up to a constant. 
\end{proof}


\section{Harmonic maps and their associated families of flat connections}

It is well--known \cite{uhlenbeck} that a harmonic map gives rise to a
family of flat connections. There are various transformations on
harmonic maps whose 
new harmonic maps   are build from
parallel sections of the associated family of flat connections: e.g.,
the associated family, the simple factor dressing
\cite{terng_uhlenbeck} and Darboux transforms
\cite{cmc}, \cite{simple_factor_dressing}.  In this paper, we investigate the links
between these transformations, when applied to the left and right normals
and to the conformal Gauss map of a minimal surface $f: M \to\R^4$. 

\subsection{The harmonic right normal and its associated family}

We  equip $\H$ with the complex structure $I$ which is given
by the right multiplication by the unit quaternion $i$. This way, we
 identify $\C^2=(\H, I)$. It is worthwhile to point out that this
complex structure $I$ differs from the complex structure $\ii$ we
used before. We will use the symbol $\C =\Span_\R\{1, I\}$ to indicate
that we use the complex structure $I$ whereas $\Cc= \Span_\R\{1,
\ii\}$.

With this at hand, the $\C_*$--family of  flat connections of a
harmonic map $R: M \to S^2$ on the trivial $\C^2$ bundle $\trivialC 2=
\trivial {}$ over $M$ can be written as
\begin{equation}
\label{eq:associated connections}
d_\lambda = d + (\lambda-1)Q\oz + (\lambda\invers-1)Q\zo\,,
\end{equation}
where $d$ is the trivial connection on $\trivial{}$,
$\lambda\in\C_*=\C\setminus\{0\}$, and $Q=-\frac 12(*dR)'' = \frac 14(RdR -*dR)$ is the
Hopf field of $R$. Moreover,
\[
Q\oz = \frac 12(Q -I*Q), \quad Q\zo = \frac 12(Q+I*Q)
\]
are the $(1,0)$ and $(0,1)$--parts of $Q$ with respect to the complex
structure $I$.  The flatness of $d_\lambda$ is obtained from the
harmonicity  (\ref{eq: N harmonic}) of $R$, that is, from 
$d*Q=\frac 12d(dR)'' =0$. Note that our choice of associated family differs from the
associated family
\begin{equation}
\label{eq:a family}
\hat d_\lambda  = d + (\lambda-1)A_R\oz + (\lambda\invers-1)A_R\zo
\end{equation}
used in \cite{cmc, simple_factor_dressing} where $A_R=\frac 12(*dR)' = \frac 14(*dR
+RdR)$. However, since
\[
Q = \frac 14(RdR -*dR) = \frac 14\Big(*d(-R) + (-R)d(-R)\Big)
\]
we see that our family of flat connections $d_\lambda$ is the
associated family in \cite{cmc, simple_factor_dressing} of the harmonic map
$-R$. Indeed, both families are gauge equivalent \cite{cmc}.  We
choose the $d_\lambda$ family since  it is closely related to the
associated family of flat connections of the conformal Gauss map.

 Since $R$ is conformal, i.e., 
 $*dR=-RdR$, we have 
\[
2*Q = dR\,,
\]
and,  for $\mu\in\C_*$ and
$\beta\in\Gamma(\trivial{})$, 
\begin{equation}
\label{eq:dbeta}
 d_\mu \beta= d\beta+\frac12dR(-R\beta(a-1)+\beta b)\,,
\end{equation}
where $a=\frac{\mu+\mu\invers}2, b=i\frac{\mu\invers-\mu}2$.
We first note that for $\mu=1$ the connection $d_{\mu}=d$ is the
trivial connection on $\trivial{}$, and all parallel sections are
constants. Thus, we will from now on assume that $\mu\not=1$. For a
$d_\mu$--parallel section $\beta\in\Gamma(\trivial{})$ put
\begin{equation}
\label{eq:m}
2 m = R\beta(a-1) - \beta b\,.
\end{equation}
Then $m$ is constant since  $a^2+b^2=0$ and thus
\[
2 dm = dR \beta(a-1) + R d\beta (a-1) - d\beta b= 
dR \beta(a-1)+  \frac 12 dR\beta((a-1)^2+ b^2)=0\,.
\]
In particular, if $\beta$ is $d_\mu$--parallel then
\begin{equation}
\label{eq:dbetam}
d\beta = dR m
\end{equation}
with $m\in\H$ constant, and (\ref{eq:m})  shows
\begin{equation}
\label{eq:beta1}
\beta = R m  + m\frac{b}{a-1}\
\end{equation}
with $m\in\H$. Note that $\beta$ is a global parallel section of the
trivial connection $d_\mu$ which is nowhere vanishing if $m\not=0$:
If
$R(p)= m \frac{b}{1-a} m\invers$ then
\[
-1 = R^2(p) = m \frac{b^2}{(1-a)^2} m\invers = m\frac{1+a}{1-a}m\invers\,,
\]
where we used $a^2+b^2=1$. Therefore $-1 = \frac{1+a}{1-a}$  which gives
a contradiction.

Conversely, every $\beta$ given by the   equation (\ref{eq:beta1}) is
$d_\mu$--parallel. Since
\begin{equation}
\label{eq:b/ a-1 in mu}
\frac{b}{a-1} = \frac{i(1+\mu)}{1-\mu}
\end{equation}
we can summarise:

\begin{lemma}
\label{lem:parallel sections r}
  Let $f: M \to \R^4$ be minimal and $d_\lambda$ the associated
  family of the right normal $R$ of $f$. For $\mu\in
  \C\setminus\{0,1\}$ every (non--trivial)
  $d_\mu$--parallel section $\beta\in\Gamma(\trivial{})$ is given by  
\begin{equation}
\label{eq:beta}
\beta = Rm + m \frac{i(1+\mu)}{1-\mu}, \quad m\in\H_* =\H\setminus\{0\}\,.
\end{equation}
\end{lemma}

In a similar way, the associated family of the left normal of a
minimal surface can be discussed.

\subsection{The conformal Gauss map and its associated family}

Again, we identify $\C^4=(\H^2, I)$ where $I$ is given by right
multiplication by the unit quaternion $i$. If the conformal Gauss map
of a conformal immersion $f: M \to S^4$ is harmonic, that is  $d*A=0$,
by the same arguments as in the case of harmonic maps into the
2--sphere, the $\C_*$--family of connections
 \begin{equation}
\label{eq:associated family of S}
d^S_\lambda = d +(\lambda-1)A\oz
+ (\lambda\invers-1)A\zo\,, \quad \lambda\in\C_*\,,
\end{equation}
is flat \cite{willmore_harmonic} on the trivial $\C^4$ bundle over $M$
where as before 
\[
A\oz =\frac 12(A-I*A) \quad \text{ and } \quad A\zo = \frac 12(A+I*A)
\]
denote the $(1,0)$ and $(0,1)$ parts of $A$ with respect to $I$.

We consider the case when $S$ is the conformal Gauss map of a minimal
immersion in $\R^4$.  We  fix $\mu\in\C_*$ and compute all
parallel sections of $d^S_\mu$. If $\mu=1$ then $d^S_\mu=d$ is
trivial, and every constant section is parallel. Assume from now on
that $\mu\not=1$, and let
\[
e=\begin{pmatrix} 1\\0
\end{pmatrix}, \quad
\psi = \begin{pmatrix} f\\ 1
\end{pmatrix}
\]
and $L=\psi \H$ the line bundle of $f$. We denote by
$ \ttrivial 2$ the pull back of the trivial $\H^2$ bundle
under the canonical projection $\pi: \tilde M \to M$ of the universal
cover $\tilde M$ to $M$. Since $e\H
\oplus L =\trivial 2$ every $d^S_\mu$--parallel section
$\varphi\in\Gamma(\ttrivial 2)$ can be written as
\[
\varphi = e\alpha + \psi \beta
\] 
where
\begin{equation}
\label{eq:dalpha}
d\alpha = -df \beta
\end{equation}
and
\begin{equation*}
d\beta = \frac 12 dR(R\beta(a-1) - \beta b)
\end{equation*}
with $a= \frac{\mu+\mu\invers}2, b=i\frac{\mu\invers-\mu}2$. Here we
used that $*dR = -RdR = dR R$ by the minimality of $f$ and thus from
(\ref{eq:A of minimal}) we see
\[
A\oz\psi = -\frac 14\psi dR(R+i), \quad A\zo\psi = -\frac 14\psi dR(R-i)\,.
\]
In particular, $\beta$ is a $d_\mu$--parallel section by (\ref{eq:dbeta}), and by Lemma \ref{lem:parallel sections r}  it is given by
\[
\beta = R m  + m \frac{i(1+\mu)}{1-\mu}\
\] 
with $m\in\H$. If $\beta=0$ is trivial, that is, $m=0$, then (\ref{eq:dalpha}) shows
that $\alpha$ is constant. For $m\not=0$, we see that
\begin{equation}
\label{eq:alpha}
\alpha = -f^*m  - fm \frac{i(1+\mu)}{1-\mu}\,,
\end{equation}
is the general solution of (\ref{eq:dalpha}) where $f^*$ is a  conjugate minimal surface of $f$, that is, 
\[
df^* = -*df = dfR\,.
\]
  We summarise:

\begin{prop}
\label{prop: parallel sections}
Let $f: M \to \R^4$ be a minimal surface with conjugate surface $f^*$
and $d^S_\lambda$ the associated family of flat connections of the
conformal Gauss map $S$ of $f$. For $\mu\in\C\setminus\{0,1\}$ every
$d^S_\mu$--parallel section is either a constant section $\varphi =e n, n\in\H$, or is
given by $\varphi = e\alpha + \psi \beta \in\Gamma(\ttrivial 2)$ with  
\begin{equation}
\label{eq:parallel sections}
\alpha =-f^*m  - fm \frac{i(1+\mu)}{1-\mu}, \quad \beta= Rm + m
\frac{i(1+\mu)}{1-\mu}\,, \quad m\in\H_* \,.  
\end{equation}   

\end{prop}
 

\section{The associated family of a minimal surface}

Motivated by our  observation (\ref{eq:parallel sections})
that parallel sections of the associated family of flat connections of
a minimal surface $f: M \to\R^4$ are given by a quaternionic linear
combination of $f$ and its conjugate surface $f^*$, we now discuss a
generalisation of the associated family of a minimal surface. This
associated family is indeed given by the associated family of the
harmonic conformal Gauss map.

\subsection{The generalised associated family}
Given a minimal surface $f: M \to\R^4$ with left and right normal $N$
and $R$ respectively, we have seen that its conjugate surface $f^*:
\tilde M \to \R^4$ has the same left and right normals $N$ and $R$
respectively since $df^* =-*df$.

In particular, the associated family  $f_{\cos\theta, \sin \theta} = f\cos \theta + f^*\sin \theta$  can easily be extended to a  family
\begin{equation*}
f_{p, q}  = fp + f^*q: \tilde M \to\R^4\,
\end{equation*}
with $p, q \in\H, (p,q)\not=(0,0)$.
Since $p + Rq$ is a quaternionic holomorphic section it has only
isolated zeros, and  
\begin{equation}
\label{eq:dfmn}
df_{p,q} = df(p + Rq)\,,
\end{equation}
shows that $f_{p,q}$
is a branched conformal immersion. The 
 right normal $R$ of $f$ gives, away from the isolated zeros of $p+Rq$, via  
\begin{equation}
\label{eq:right normal associated}
R_{p,q} = (p+Rq)\invers R(p+Rq)
\end{equation}
the right normal   of $f_{p,q}$ whereas the left normal $N$ of $f$ is the left normal 
\[
N_{p,q} = N
\]
of $f_{p,q}$ . Thus, by (\ref{eq:H})
\[
df_{p,q} H_{p,q} = (dN_{p,q})' =dN'=0
\]
and $f_{p,q}$ is a (branched) minimal immersion. Similarly, we have a family of (branched) minimal immersions $f^{p,q} =pf + qf^*$, $(p,q)\not=(0,0)$, with right normal $R^{p,q} = R$ and left normal $N^{p,q} = (p-qN)N(p-qN)\invers$. \\

\begin{definition}
  Let $f: M \to \R^4$ be a minimal surface.  The family of (branched) minimal immersions
\begin{equation}
\label{eq: fmn}
f_{p,q}  = fp+ f^*q \colon \tilde M \to \R^4\,, \quad p,q \in\H, (p,q)\not=(0,0)\,,
\end{equation}
where $f^*: \tilde M \to \R^4$ is a conjugate surface of $f$, is called the
\emph{right associated family} of $f$. The family of (branched) minimal immersion 
\[
f^{p,q} = p f + q f^*\colon \tilde M \to \R^4\,, \quad p,q \in\H, \quad (p,q)\not=(0,0)\,,
\]
is called the \emph{left associated family} of $f$.
\end{definition}
Note that for $p,q\in\R, (p,q)\not=(0,0),$ we obtain the usual associated family of a
minimal surface up to scaling. Moreover, $f_{pn, qn} = f_{p,q} n$ is given by a
scaling of $f_{p,q}$ and an isometry on $\R^4$ for $n\in\H_*$.

\begin{theorem}
  The right (left) associated family is a $S^3$--family of minimal
  surfaces in $\R^4$ which contains the classical associated family
  $f_{\cos\theta, \sin\theta}, \theta\in\R,$ of minimal surfaces. 

  The right (left) associated family preserves the conformal class,
  and a surface $f_{p,q}$ (or $f^{p,q}$) is isometric to $f$ if and
  only if it is an element of the classical
  associated family, up to an isometry of $\R^4$. 
\end{theorem}
\begin{proof}   It only remains to show that $f_{p,q}$ is isometric
  to $f$ if and only if $f_{p, q} = f_{n\cos\theta, n\sin\theta}$ for some
  $\theta\in\R$ and $n\in S^3$. Assume first that $f_{p,q}$ is
  isometric to $f$.  By multiplying with a unit quaternion from the
  right we may assume that $p\in\R$. Then $|df_{p,q}| = |df|$ implies
  by (\ref{eq:dfmn})
\begin{equation*}
p^2 + |q|^2 + 2 p\Re(Rq) = |p+Rq|^2 = 1
\end{equation*}
and thus $\Re(Rq)$ is constant. Since the stereographic projection of
$R$ is a meromorphic function, the right normal $R$ can only take
values in a plane if $R$ is constant. In other words, if $R$ is not
constant then $\Re(Rq) = \Re(R\, \Im q) = -< R, \Im q>$ is constant only
if $\Im q=0$.  But then the above equation gives $p^2+q^2=1$ with
$q\in\R$ and $(p,q) = (\cos\theta,\sin\theta)$ for some
$\theta\in\R$. If $R$ is constant then $df^* = - *df = df R$ gives
$f_{p,q} = f(p+Rq) = f_{p+Rq, 0}$ with constant $p+Rq\in S^3$.

Conversely, $f_{p,q} = f_{n\cos\theta, n\sin\theta}, n\in S^3$, then $|p+Rq|=1$ and thus $|df_{p,q}| = |df|$ by (\ref{eq:dfmn}).

A similar argument shows the statement for the left associated family.
\end{proof}

\begin{rem} For  any  immersion $f: M \to\R^3 =\Im\H$  in Euclidean 3--space, the left and
right normal coincide. A surface in the right associated
family $ f_{p,q} $ of a minimal surface $f: M \to \R^3$ has left normal $N_{p,q} = N$ and right
normal $R_{p,q} = (p+Nq)\invers N(p+Nq)$. In particular,  
we have $N_{p,q}
\not= R_{p,q}$ in general  and thus, elements of the right associated families of
a minimal surface $f: M \to \R^3$  are
not necessarily minimal in 3-space but are
minimal surfaces in $\R^4$.  
\end{rem}

\subsection{The associated family of the harmonic conformal Gauss map}

We now give a derivation of the associated family of minimal surfaces
in terms of the associated family of harmonic maps. Recall that in the
case of a constant mean curvature surface $f:M \to\R^3$ the Gauss map
$N: M \to S^2$ of $f$ is by the Ruh--Vilms theorem \cite{ruh_vilms}
harmonic and its associated family $\hat d_\lambda$ of flat
connections (\ref{eq:a family}) gives rise to the associated family of
constant mean curvature surfaces: for $\mu\in \C_*$ the map $N$ is
harmonic with respect to the quaternionic connection $\hat d_\mu$ for
all $\mu\in S^1$, see e.g. \cite{simple_factor_dressing} for
details. Thus, for any $\hat d_\mu$--parallel section
$\varphi\in\Gamma(\trivial{})$ the map $N_\mu = \varphi\invers
N\varphi$ is harmonic with respect to $d$. Note that $\varphi$ is
unique up to a right multiplication by a constant quaternion. This
family $N_\lambda$, $\lambda\in S^1$, is called the \emph{associated
  family} of $N$.  The Sym--Bobenko formula \cite{sym_bob} relies on
the link (\ref{eq:H}) between the differentials of $f$ and $N$ to
obtain the constant mean curvature surface $f$ by differentiating a
family $\varphi_\lambda$ of parallel sections of $d_\lambda$ with
respect to the parameter $\lambda$.  This way, one can obtain from the
associated family $N_\lambda$ of $N$ a family of constant mean
curvature surfaces, the \emph{associated family} of $f$.

In the case of a minimal surface $f: M\to \R^4$ with right normal $R$,
we have seen in Lemma \ref{lem:parallel sections r} that every
parallel section $\beta\in\Gamma(\trivial{})$ of the associated family
of flat connections of $R$ is, using (\ref{eq:b/ a-1 in mu}), given by
\[
\beta = Rm  + m\frac{b}{a-1}, \quad m\in\H_*\,,
 \]
where $a=\frac{\mu + \mu\invers}2, b= i\frac{\mu\invers-\mu}2$, and $\mu\in\C\setminus\{0,1\}$. 

If $\mu\in S^1, \mu\not=1, $ then $a =\Re \mu, b=\Im \mu\in\R$,
and $\beta =(R+ \frac{b}{a-1})m$ with $m\in\H_*$. Thus, the associated
family of harmonic maps is given by
\[
R_\mu = \beta\invers R \beta = m\invers R m\,.
\]
In particular, the harmonic map $R_\mu$ does not depend on the
parameter $\mu$. In the case of a minimal surface, the equation
(\ref{eq:H}) does not allow to reconstruct the minimal surface from
the harmonic Gauss map. In particular, there is no Sym--Bobenko
formula to obtain a minimal surface with right normal $R_\mu$ via
differentiation of parallel sections with respect to the parameter
$\mu$.

To obtain a non--trivial family of minimal surfaces with right normal
$R_\mu$ for $\mu\in\C_*$, we consider instead the harmonic conformal Gauss map
$S$ of $f$:

\begin{theorem}
\label{thm: associated}
Let $f: M \to \R^4$ be a minimal surface with conformal Gauss map $S$.
Then, up to M\"obius transformation, the line bundle of a minimal
surface $f_{\cos\theta, \sin\theta}, \theta\in\R,$ in the associated
family of $f$ is given by
\[
L_\phi = \phi\invers L
\]
where $\phi = (\varphi_1, \varphi_2)$ is an invertible endomorphism
with $d_\mu$--parallel columns $\varphi_1, \varphi_2$, and $\mu =
\cos(2\theta) - i\sin(2\theta)\in S^1$. Moreover,
\[
S_\phi = \phi\invers S \phi
\]
is the conformal Gauss map of $f_{\cos\theta, \sin\theta}$.
\end{theorem} 
\begin{proof}
  For $\mu\in S^1$ and invertible  $\phi$ and $\tilde \phi$ with
  $d_\mu$--parallel columns,  we
  have $\tilde \phi = \phi B$ with $B$ constant so that
  $L_{\tilde\phi} = B\invers L_\phi$ is given by a M\"obius
  transformation. Thus, we can assume without loss of generality that
  for $\mu\not=1$
\[
\phi = (e, \varphi)
\]
where $\varphi =e\alpha + \psi\beta$ is a   $d_\mu$--parallel
section with nowhere vanishing $\beta$. Then
\[
\phi\invers \psi =\phi\invers (-e\alpha + \varphi)\beta\invers = \begin{pmatrix}
-\alpha \\ 1
\end{pmatrix}
\beta\invers\,,
\]
so that $L_\phi$ is the line bundle of the (branched) conformal
immersion $f_\phi = -\alpha: M \to\R^4$. By Proposition \ref{prop:
  parallel sections} and (\ref{eq:b/ a-1 in mu}) we have
\[
\alpha = -(f^* + f\frac{b}{a-1})m
\]
 since  $\frac{b}{a-1} \in
\R$ for $\mu\in S^1$. 
Writing $\mu=\cos\vartheta + i\sin\vartheta$ we
see
\[
\frac{b}{a-1} = \frac{\sin\vartheta}{\cos\vartheta-1} = - 
\cot\frac\vartheta 2\,.
\] 
We
consider $\alpha$ up to M\"obius transformations, and may thus assume
without loss of generality that $m= -\sin\frac\vartheta 2$ so
that
\[
f_\phi = f_{\cos\theta, \sin\theta}
\] 
for $\theta = -\frac\vartheta 2$.
For $\mu=1$ we can choose $\phi$ as the identity matrix, and obtain
$f_\phi = f_{1,0} = f$. 

By definition $S_\phi$ leaves $L_\phi$ invariant.   The final statement follows by a
similar argument as in \cite{simple_factor_dressing}. In fact, this is a special
case of a corresponding statement for (constrained) Willmore surfaces $f: M \to
\R^4$, see
\cite{tokyo}, \cite{burstall_quintino}.  
\end{proof}

  We also get an interpretation of the right associated family:
  If we consider $\phi=(e,\varphi)$ with $d^S_\mu\varphi=0$ for some
  $\mu\in\C\setminus\{0,1\}$, then the line bundle $L_\phi =
  \phi\invers L$ gives by the same argument as above a member of the
  generalised associated family, up to M\"obius transformation,
\[
f_\phi = \alpha(1-a) = f^*m(a-1) + fmb = f_{mb, m(a-1)}\,.
\]
Note however, that $\phi$ is not a $d^S_\mu$--parallel
endomorphism since $d^S_\mu$ is a complex but not a quaternionic
connection.


\section{Simple factor dressing}

As we have seen, the harmonic left and right normals and the harmonic
conformal Gauss map of a minimal surface give rise to families of flat
connections.  Conversely, if a family of flat connections is of an
appropriate form, it can be used to construct a harmonic map. In
particular, if $d_\lambda$ is the associated family of flat
connections of a harmonic map, one can gauge  $d_\lambda$ with a 
$\lambda$--dependent dressing matrix $r_\lambda$ to obtain a new
family of flat connections $\breve d_\lambda =
r_\lambda\cdot d_\lambda$. If $r_\lambda$
satisfies appropriate reality and holomorphicity conditions, then
$\breve d_\lambda$ is the associated family
of a new harmonic map, a so--called \emph{dressing} of the original
harmonic map,   see
\cite{uhlenbeck}, \cite{terng_uhlenbeck}.

\subsection{Simple factor dressing of the left and right normals}

We first consider the case when $R: M \to S^2$ is the harmonic right normal of a
minimal surface in $\R^4$ and choose the simplest possible dressing
matrix: If $d_\lambda$ is the associated family  (\ref{eq:associated
  connections}) of the harmonic map $R$ then the simple
factor dressing matrix $r_\lambda$ is obtained by choosing
$\mu\in\C\setminus\{0,1\}$ and a  parallel section $\beta$ of the flat
connection $d_\mu$ of the associated family   as
\[
r_\lambda v= \begin{cases}  v\frac{1-\bar\mu\invers}{1-\mu}
  \frac{\lambda-\mu}{\lambda-\bar\mu\invers}\,, & v \in \beta\C\\
v\,, & v \in (\beta\C)^\perp
\end{cases}\,.
\]
In \cite{simple_factor_dressing} it is shown that  $\breve d_\lambda = r_\lambda
\cdot d_\lambda$ is the associated family of a  harmonic map
$\breve R: M \to S^2$, 
the \emph{simple factor dressing} of $R$, and  
\begin{equation*}
 \breve R = \breve
T\invers R \breve T
\end{equation*}
where $ \breve T=\frac 12(-R\beta(a-1)\beta\invers+ \beta
b\beta\invers) $ is given explicitly in terms of $\beta$ and $\mu$,
where $a=\frac{\mu+\mu\invers}2, b = i\frac{\mu\invers-\mu}2$. In our
case, any $d_\mu$--parallel section $\beta$ is given by
Lemma \ref{lem:parallel sections r} and (\ref{eq:b/ a-1 in mu}) as
 \[
\beta = (R + \rho)m,  \quad m\in\H_*\,, \quad \rho = m\frac{b}{a-1}m\invers\,.
\]
Since $\beta$ is nowhere vanishing, 
we see that  
\[
\breve T =  \frac 12(-R \beta(a-1) + \beta b) \beta\invers
= -(R+\rho)\invers\,,
\]
and the simple factor dressing of $R$ is 
\begin{equation}
\label{eq:right normal dressing}
\breve R = (R+\rho) R  (R+\rho)\invers\,, \qquad \rho = m
\frac{i(1+\mu)}{1-\mu} m\invers\,,  \quad m\in\H_*\,.
\end{equation}  
In particular, a simple factor dressing is determined by the choice
$\mu\in\C\setminus\{0,1\}$, giving the pole of the simple factor, and $m\in\H_*$,
giving the parallel bundle of $d_\mu$.
A straight forward computation shows that 
\[
d\breve R = (1+\rho^2)(R+\rho)\invers dR (R+\rho)\invers\,.
\]
 Since
\[
\rho^2 =(R+\rho - R)^2 = (R+\rho)^2 - R(R+\rho) -(R+\rho)R -1
\]
we see that
\[
(R+\rho)\invers (1+\rho^2) (R+\rho)\invers = 1-  (R+\rho)\invers R -
R(R+\rho)\invers\,.
\]
The right hand side of this   equation 
commutes with $R$ so that
also 
\begin{equation}
\label{eq: R commutes with rho expression}
[(R+\rho)\invers (1+\rho^2)(R+\rho)\invers, R]= 0\,.
\end{equation}
Therefore, 
\[
(d\breve R)' = \frac 12(d\breve R - \breve R *d\breve R) =
(1+\rho^2)(R+\rho)\invers \frac 12(dR + R*dR) (R+\rho)\invers =0
\]
which shows that  $\breve R$ is conformal and harmonic. Indeed, the simple factor
dressing of the right normal of a minimal surface in $\R^4$ is the
right normal of a minimal surface:

\begin{theorem}
Let $f: M \to \R^4$ be a minimal surface with right normal $R$.

Then every simple factor dressing of $R$ is the right normal $R_{p,q}$
of a minimal surface $f_{p,q}$ in the right associated family of $f$.
\end{theorem}
\begin{proof}
Let $\mu\in\C\setminus\{0,1\}$ be the pole of the simple factor dressing and put, as usual,
$a= \frac{\mu+\mu\invers}2, b=\frac
{i(\mu\invers-\mu)}2$.  Write $\hat c = m c
m\invers$ for $c\in\C$ where $m\in\H_*$  determines the $d_\mu$--parallel bundle of the simple factor dressing.

Consider the minimal surface 
\[
h = f\hat b + f^*(\hat a-1)
\]
in the associated family of $f$.
      From (\ref{eq:right normal
  associated}) we see that the right normal of $h= f_{\hat b,
  \hat a-1}$ is given by
\begin{equation}
\label{eq: right normal of associated}
R_h = (\hat a-1)\invers(\rho + R)\invers R (\rho+R)(\hat a-1)\,,
\end{equation}
where we used that $\hat a\not=1$ since $\mu\not=1$. Now, $\hat a^2 +
\hat b^2=1$ gives $1 + \rho^2 = 1 + \frac{\hat b^2}{(\hat a-1)^2} =
-\frac 2{\hat a-1}$.  In particular, we have seen already (\ref{eq: R
  commutes with rho expression}) that $(R+\rho)\invers \frac 2{\hat
  a-1} (R+\rho)\invers$ commutes with $R$ so that also
\[
[R, (R+\rho)(\hat a-1)(R+\rho)] =0\,.
\]
Thus, the right normal
\begin{equation}
\label{eq:right normal associated simplified}
R_h =  (\rho +R) R (\rho+R)\invers
\end{equation}
of the associated minimal surface $h$ is (\ref{eq:right normal dressing}) the dressing of $R$.
\end{proof}

\quad

\begin{rem}
\label{rem:minimal surface with sfd not unique}
  Note that the simple factor dressing of the harmonic right normal
  does not single out a canonical minimal surface with this right
  normal: the element $f_{\hat
  b, \hat a-1}$ is one example of such a minimal
  surface but so is    $pf_{\hat b, \hat a-1} + qf^*_{\hat
  b, \hat a-1}$ for any 
  $p, q \in\H_*\,.$
\end{rem}

An analogue theorem holds for the left normal $N$ of a minimal
surface $f: M \to \R^4$:

\begin{theorem}
 The simple factor dressing of $N$ is the left
normal $N^{p,q}$ of an element $f^{p,q}$ in the left associated family of $f$.
\end{theorem}

As noted before, if $f: M \to\R^3$ is a minimal surface in $\R^3$ then the left and
right associated families give in general minimal surfaces in $\R^4$. 
 In particular:
\begin{cor}
  Let $f: M \to\R^3$ be a minimal surface with Gauss map $N$ and
  assume that $f$ is not a plane.  Let $f_{\hat b, \hat a-1}$ be the
  minimal surface in the associated family of $f$ whose right normal
  $R_{\hat b, \hat a-1}$ is the simple factor dressing of $N$ given by
  $\mu\in\C\setminus\{0,1\}$, $m\in\H_*$, where $\hat a = m
  \frac{\mu+\mu\invers}2m\invers, \hat b = m i\frac{\mu\invers-\mu}2m\invers$.\\

Then $f_{\hat
  b, \hat a-1}$    is a minimal surface in $\R^3$ if and only if
$\mu\in S^1, \mu \not=1$.\\

\end{cor}
\begin{proof}
If $f_{\hat b, \hat a-1}$ is in 3--space then $N_{\hat
  b, \hat a-1} = R_{\hat b, \hat a-1}$, that is by (\ref{eq:right
  normal associated simplified}) 
\[
N = (\rho + N)N(\rho+N)\invers
\]
with $\rho=\frac{\hat b}{\hat a-1}$. But then $[\rho, N]=0$ and, since
$\rho$ is constant, we see that $\rho = m i\frac{1+\mu}{1-\mu}
m\invers\in\R$.  In particular, $i\frac{1+\mu}{1-\mu}\in\R$
which is equivalent to $\mu\in S^1, \mu\not=1$.

Conversely,  for $\mu\in S^1$ we know $\hat a, \hat b\in\R$  so 
that $f_{\hat b, \hat a-1} = f \hat b + f^*(\hat a-1)$ takes values in $\R^3$. 
\end{proof}

We will use Remark \ref{rem:minimal surface with sfd not unique} to
obtain a minimal surface in 3-space with a given simple factor
dressing as its Gauss map. This operation turns out to be the surface
obtained by applying a corresponding simple factor dressing on the
conformal Gauss map.

\subsection{Simple factor dressing of a minimal surface}

The conformal Gauss map of a Willmore surface is harmonic and one can
define a dressing on it \cite{aurea}, \cite{burstall_quintino}. Since
the conformal Gauss map determines a conformal immersion (if the Hopf
field is not zero), this induces a transformation, a dressing, on
Willmore surfaces. (Actually, Burstall and Quintino define more
generally a dressing on constrained Willmore surfaces).

We will again only discuss the special case of simple factor dressing
by choosing the simplest possible dressing matrix.  As before, the
simple factor dressing of the conformal Gauss map $S$ of a Willmore
surface $f: M\to S^4$ is given explicitly by parallel sections of a
connection $d^S_\mu$ of the associated family of flat connections
\cite{willmore_harmonic}: for $\mu\in\C\setminus\{0,1\}$ let $W_\mu$
be a $d^S_\mu$--stable, complex rank 2 bundle over $\tilde M$ with
$W_\mu \oplus W_\mu j =\ttrivial{2}$. For two $d^S_\mu$--parallel
sections $\varphi_1, \varphi_2\in \Gamma(W_\mu)$ with
$\phi=(\varphi_1, \varphi_2)$ regular, define a conformal immersion
$\hat f: \tilde M \to S^4$ by
\[
\hat L = (S + \phi \frac b{a-1} \phi\invers)L\,,
\]
where $\frac b{a-1}$ is the left multiplication by the quaternion
$\frac b{a-1}$ on $\trivial 2$ and as usual $a=\frac{\mu+\mu\invers}2,
b = i\frac{\mu\invers-\mu}2$.  Then the conformal Gauss map $\hat S$
of $\hat f$ is the \emph{simple factor dressing} of $S$.  In
particular, $\hat S$ is harmonic, and $\hat f$ is a Willmore surface.
It is known that $\hat L$, and thus $\hat f$, is independent of the choice of
basis for $W_\mu$, \cite{willmore_harmonic}. We call the Willmore
surface $\hat f$ a \emph{simple factor dressing} of $f$.  Note that
$\frac b{a-1}\in\R$ for $\mu \in S^1$ so that $\hat f = f$ in this
case.

Note that this gives a conformal theory. However, in the case of a
minimal surface $f: M \to\R^4$ we are only interested in the Euclidean
theory, that is, simple factor dressings which are again surfaces in
the same 4--space. Thus, we will restrict to simple factor dressings
 such that $S +
\phi\frac b{a-1} \phi\invers$ stabilises the point at infinity $\infty
= e\H$.  Because $Se = eN$ by (\ref{eq:conformal Gauss map}) where $N$
is the left normal of $f$, we have to restrict to $\phi$ with $\phi
\frac b{a-1} \phi\invers \infty = \infty$. Since $\frac b{a-1}\in\R$
if and only if $\mu\in S^1$ we can assume that $\frac b{a-1}$ is not
real as otherwise $\hat f =f$.

We recall that any $d^S_\mu$--parallel section $\varphi\not= en,
n\in\H$, is given by Proposition \ref{prop: parallel sections} and
(\ref{eq:b/ a-1 in mu}) as $ \varphi = e\alpha + \psi
\beta 
$ where $ \alpha = -(f^* +f\rho)m, \beta =(R+\rho) m $ with $\rho =
m\frac{i(1+\mu)}{1-\mu} m\invers$ and $ m\in\H_*$. In particular, for
regular $\phi=(\varphi_1, \varphi_2)$ with two such
$d^S_\mu$--parallel sections $\varphi_1, \varphi_2$ we write
\[
\phi =\begin{pmatrix} \alpha_1 & \alpha_2\\ 
\beta_1 & \beta_2
\end{pmatrix}, \qquad \phi\invers= \begin{pmatrix} \zeta_1 & \xi_1 \\
  \zeta_2 & \xi_2
\end{pmatrix}
\]
in the basis $(e, \psi)$ where 
\begin{equation}
\label{eq:alpha_i}
\alpha_1 =f_{- m_1 \frac b {a-1},
  -m_1}, \quad \alpha_2 =f_{-m_2 \frac b{a-1}, -m_2}\,, \quad m_1, m_2\in\H_*\,,
\end{equation}
are in the right
associated family of $f$, and 
\begin{equation}
\label{eq:beta_i}
\beta_1 = Rm_1 + m_1\frac b{a-1}, \quad
\beta_2 = Rm_2 + m_2\frac b{a-1} 
\end{equation}
are nowhere
vanishing parallel sections with respect to the connection $d_\mu$ of the
associated family of connections of the right normal $R$. Assume that
$\phi\frac{b}{a-1} \phi\invers$ fixes the point at infinity. Then
there exists $\eta: M \to\H_*$ with
\[
\frac b{a-1}\begin{pmatrix} \zeta_1 \\  \zeta_2
\end{pmatrix}
=\begin{pmatrix} \zeta_1 \\   \zeta_2  
\end{pmatrix}\eta \,.
\]
Since $\beta_1, \beta_2$ are nowhere vanishing so is $\zeta_1$ and the
equation above shows that   $\zeta_2\zeta\invers_1$ commutes with the
complex  
 number $\frac b{a-1}\not\in\R$. Therefore,  we have $\zeta_2 = q\zeta_1$ with $q: M
\to \C_*$ and, because $\phi\phi\invers= \Id$ and $\zeta_1$ is nowhere
vanishing,  we conclude $ \beta_1 + \beta_2 q =0$.
 
Both $\beta_1$ and $ \beta_2$ are
$d_\mu$--parallel, and thus, $q\in \C_*$ is constant. 

If $R$ is constant, that is, if $f$ is the twistor projection of a
holomorphic curve in $\CP^3$, then $f^* = fR + c$ with some $c\in\H$,
so that with (\ref{eq:alpha_i}) and (\ref{eq:beta_i}) we obtain
$\varphi_1 + \varphi_2 q= en$ with $n= -c(m_2 q+m_1)$ and $q\in\C_*$.
In other words, $en$ is a $d^S_\mu$--parallel section of $ W_\mu$ and
we can replace $\phi$ in the definition of $\hat S$ by $\tilde\phi =
(en, \varphi_2)$.

If $R$ is not constant, we use again the explicit forms
(\ref{eq:beta_i}) to obtain from $\beta_1 + \beta_2q=0$, $q\in\C_*$,
that $R(m_1 + m_2q)$ is constant, which implies that $m_1 =-
m_2q$. But then (\ref{eq:alpha_i}) and (\ref{eq:beta_i}) show that
$\varphi_1 = -\varphi_2 q$ which contradicts the assumption that $\phi$
is regular.

Thus, from now on we can restrict to regular endomorphism $\phi$ of
the form
\[
\phi= \begin{pmatrix} n & \alpha\\ 0 & \beta
\end{pmatrix}
\]
in the basis $(e,\psi)$ where $n\in\H_*$ and $\varphi=
e\alpha+\psi\beta$, $\beta$ nowhere vanishing, is a parallel section
of $d^S_\mu$ to obtain all simple factor dressings in 4--space.

\begin{lemma}
\label{lem:simple factor with parameter}
Let  $f: M \to \R^4$ be  a minimal
surface and 
   $\hat f: \tilde M \to \R^4$ a simple factor dressing of $f$ in
   4--space. 

Then $\hat f$ is a minimal surface, and  $\hat f$
 is
preserved when changing the parameters $(\mu, m, n)$ to   $(\mu, mz, nw)$
 or $(\bar\mu\invers, mj, nj)$ for $z, w\in \C_*$. For $\mu\in S^1$
 the simple factor dressing  $\hat f= f$ is trivial. 
 
\end{lemma}
\begin{proof}

  Since $\psi\beta\in\Gamma(L)$ is nowhere vanishing,  the line bundle $\hat L$
  is spanned by   
\[
  \hat \varphi = (S + \phi\frac{b}{a-1}\phi\invers)(\varphi-e \alpha)\,.
\]
Since $Se = eN$ and $S\varphi = e N\alpha - \psi R\beta$ by
(\ref{eq:conformal Gauss map}) we conclude
\[
 \hat \varphi = \begin{pmatrix}N\alpha\\ -R\beta
 \end{pmatrix}
 + \begin{pmatrix}-N\alpha\\ 0
 \end{pmatrix}
+ \begin{pmatrix}\alpha\frac b{a-1} \\ \beta\frac b{a-1} 
 \end{pmatrix}+ \begin{pmatrix}-n\frac b{a-1} n\invers \alpha \\  0
 \end{pmatrix} = \begin{pmatrix} \alpha\frac b{a-1}  -n\frac
   b{a-1}n\invers \alpha  \\ -R\beta + \beta\frac b{a-1} 
 \end{pmatrix}  
\]
in the basis $(e,\psi)$. Recalling (\ref{eq:m}), that is,   $-R\beta
+ \beta\frac{b}{a-1} = -2m(a-1)\invers$, we see
that 
\[
\hat f = f - (\alpha\frac b{a-1} -n\frac b{a-1}n\invers\alpha) \frac{a-1}2m\invers\,.
\]
We substitute $\alpha = - fm\frac{b}{a-1} -f^*m$ and use $\hat a^2 + \hat b^2=1$
for $\hat a = mam\invers$, $b = m b m\invers$, and $\tilde \rho = n
\frac{b}{a-1}n\invers,$ to obtain  
\[
\hat f = \frac 12\left(-f(\hat a-1) + f^*\hat b -
  \tilde\rho\left(f\hat b+f^*(\hat a-1)\right)\right)\,.
\]
Thus,
\[
\hat f = -\frac{\tilde \rho}2 h+\frac 12h^*
=h^{-\frac{\tilde\rho}2,\frac 12}
\]
 is an 
element of the left associated family of the minimal surface 
\[
h = f
\hat b + f^*(\hat a-1) = f_{\hat b, (\hat a-1)}
\]
in the right associated family of $f$.  In particular, $\hat f$ is
minimal. From the explicit formula above we see that $\hat f$ is preserved
when changing $(\mu, m, n)$ to $(\mu, mz, nw)$ with $z, w\in\C_*$.
Since $\frac{\bar\mu\invers +
  \bar\mu}2=\bar a, i \frac{\bar\mu - \bar\mu\invers}2 = \bar b$ and
$(mj) \bar z (mj)\invers = mzm\invers$ for all $z\in \C, m\in\H_*$, we
obtain the same simple factor dressing for  the parameters $(\mu, m,
n)$ and $(\bar\mu\invers, mj, nj)$. The final statement follows from
the fact that $\frac b{a-1} \in\R$ for $\mu\in S^1$.

\end{proof}

\begin{rem}
Note that the last statements in the above corollary are special cases
of more general facts for simple factor dressings of Willmore
surfaces: since   the
simple factor dressing is independent of the choice of basis of
$W_\mu$ and the family of flat connections satisfies a reality
condition  \cite{simple_factor_dressing}, the surface is preserved under the given
changes of parameter. The last statement holds for general simple
factor dressings with $\mu\in S^1$.
\end{rem}

In particular, we emphasise again that in contrast to the simple factor dressing of the right
and left normal, the simple factor dressing of the conformal Gauss map associates
a unique minimal surface:  

\begin{definition}
  The \emph{simple factor dressing} of a minimal surface $f: M \to
  \R^4$ with \emph{parameters} $(\mu, m, n)$ is the minimal surface
  $\hat f: \tilde M \to\R^4$
  given by
\begin{equation}
\label{eq:simple conformal}
\hat f = -f\frac{m(a -1)m\invers}2 + f^*\frac{m bm\invers}2 - n\frac{b}{a-1}n\invers\left(f\frac{m b m\invers}2+f^*\frac{m(a-1)m\invers}2\right)
\end{equation}
where $m,n \in S^3$, $\mu\in\C\setminus\{0,1\}$ and $a= \frac{\mu+\mu\invers}2, b=
i\frac{\mu\invers-\mu}2$.

If $m=n=1$ then we refer to $f^\mu = \hat f$ as the \emph{simple factor
dressing} of $f$ with \emph{parameter} $\mu$. \\
\end{definition}

The simple factor dressing with parameter $\mu$ of the rigid motion
$\tilde f = n\invers f m$ of $f$ is given by
\[
\tilde f^\mu = - n\invers fm \frac{a-1}2 + n\invers f^* m \frac b2 -
\frac{b}{a-1}( n\invers fm \frac b2+
n \invers f^*m\frac{a-1}2  ) = n\invers  \hat f m\,,
\]
where $\hat f$ is the simple factor dressing (\ref{eq:simple
  conformal}) of $f$ with parameters $(\mu, m, n)$.  Thus, all simple
factor dressings are build from rigid motions of the simple factor
dressings with parameter $\mu$:

\begin{prop}
\label{prop: SFD given by parameter mu}
Let $\hat f$ be a simple factor dressing of a minimal surface $f: M
\to \R^4$ with parameters $(\mu, m, n)$. Then 
\[
\hat f = \mathcal{R}_{n, m}((\mathcal{R}_{n,m}\invers(f))^\mu)
\]
where $(\mathcal{R}_{n,m}\invers(f))^\mu$ is the simple factor dressing of the
rotated surface $\mathcal{R}_{n,m}\invers(f) = n\invers f m$ with parameter $\mu$.
\end{prop}

Since the associated families of the left and right normals and the
conformal Gauss  maps are related, we also have a correspondence
between the resulting simple factor dressings: 
\begin{cor}
\label{cor:complete}
The simple factor dressing of a minimal immersion $f: M \to \R^4$ with 
parameters $(\mu, m,n)$ is a minimal immersion $\hat f:
\tilde M \to \R^4$. 

The right and left normal of $\hat f$ are given by simple factor
dressings of the right and left normal of $f$ respectively. Moreover,
$\hat f$ is complete if and only if $f$ is complete.
\end{cor}
\begin{proof}
The differential of the
  simple factor dressing $\hat f$ with parameters $(\mu, m, n)$ is
  given by
\[
d\hat f = - (N+ n \frac{b}{a-1} n\invers) \frac{df}2 (R  + m \frac{b}{a-1} m\invers) m(a-1)m\invers
\]
where we used that $df^* =- *df$ and $*df = Ndf = -df R$.
In
particular, the right normal 
\[
\hat R =m(a-1)\invers m\invers(R+m\frac
b{a-1} m\invers)\invers R (R  + m \frac{b}{a-1} m\invers)
m(a-1)m\invers  
\]
and left normal
\[
\hat N =  (N+ n \frac{b}{a-1} n\invers) N  (N+ n \frac{b}{a-1} n\invers) \invers
\] 
of $\hat f$ are by (\ref{eq: right normal of associated}) and
(\ref{eq:right normal associated simplified}) the simple factor
dressings of the right and left normal of $f$ which are given by the
pole $\mu$ and the parallel sections $Rm + m\frac {b}{a-1}$
and $Nn + n\frac b{a-1}$ respectively.

Finally, $\hat f$ is branched at $p$ if and only if $(N(p)+ n \frac{b}{a-1} n\invers) =0$ or  $(R(p)  + m \frac{b}{a-1} m\invers)=0$.
We already have seen that $\beta= R  m + m \frac{b}{a-1}$ is nowhere
vanishing if $m\not=0$. A similar argument, as given before Lemma 
\ref{lem:parallel sections r},  gives the
corresponding statement for the expression in $N$,   so
that $f$ and  $\hat f$ have the same conformal class, that is, 
\begin{equation}
\label{eq: dhat f}
|d\hat f| = r |df|
\end{equation}
with $r: M \to (0,\infty)$. In  
particular, the simple factor dressing $\hat f$ of a minimal immersion
has no branch points, and $\hat f$ is complete if and only if $f$ is complete.
\end{proof}

From the explicit form (\ref{eq:simple conformal})  of the simple
factor dressing of a minimal surface we immediately see that the
simple factor dressing commutes with the conjugation:

\begin{cor}
Let $f: M \to\R^4$ be a minimal surface and $f^*$ a conjugate surface
of $f$. Then a conjugate surface of the  simple factor dressing of $f$
is given by a simple factor dressing of the conjugate surface $f^*$.

Moreover, the choice of a different conjugate surface results in a
translation of the simple factor dressing   in 4--space.  
\end{cor}


\section{Simple factor dressing and the L\'opez-Ros deformation}

Given a minimal surface $f: M \to\R^4$ in 4--space with Weierstrass
data $(g_1, g_2, \omega)$ denote, in analogy to the case of a minimal
surface in $\R^3$, by $f^\sigma$ the \emph{L\'opez-Ros deformation} of
$f$ with complex parameter $\sigma\in\Cc$, that is, the minimal
surface given by the Weierstrass data $(\sigma g_1, \sigma g_2,
\frac{\omega}\sigma)$. Similarly, the \emph{Goursat transformation} is
defined by $\Re(\mathcal{A}(f+\ii f^*))$ where $\mathcal{A}\in \Oo(4,\Cc)$ and $f^*$ is a
conjugate surface of $f$. In this section, we will show that the
L\'opez--Ros deformation is a special case of the simple factor
dressing. Indeed, all simple factor dressings are (special) Goursat
transformations.

\subsection{The L\'opez--Ros deformation in $\R^4$}

Since by Proposition \ref{prop: SFD given by parameter mu} any simple
factor dressing is given in terms of the simple factor dressing with
parameter $\mu$, we will first show that these simple factor dressings
are Goursat transformations.

\begin{theorem}
\label{thm: explit SFD with paramter mu}
Let $f: M \to\R^4$  be a minimal surface in $\R^4$. 
Then the simple factor dressing with
parameter $\mu$ of   $f=f_0 + f_1 i + f_2 j + f_3 k$ is given by 
\begin{equation}
\label{eq: sfd explicit}
f^\mu = 
\begin{pmatrix} f_0 \\ f_1 \\ 
\cos  t\, (f_2\cosh s  - f_3^*\sinh s)  - \sin t \, (f_3\cosh s + f_2^*\sinh
s)\ \\
\sin t \, (f_2\cosh s  - f_3^*\sinh s) + \cos t \, (f_3\cosh s + f_2^*\sinh
s)
\end{pmatrix}
\end{equation}
where $s =- \ln|\mu|,$ $t=\arg\frac{\bar\mu-1}{\bar\mu(1-\mu)}$.  In
particular, $f^\mu$ is a Goursat transform of $f$  whose holomorphic
null curve is  
\begin{equation}
\label{eq: null curve of sfd}
\Phi^\mu = \mathcal{L}^\mu\Phi
\end{equation}
where $\Phi$ is the holomorphic null curve of $f$ and 
\[
\mathcal{L}^\mu = \begin{pmatrix} 1&0&0&0 \\
0&1&0&0 \\
  0 & 0 &\cosh w  & \ii\sinh w\\
0&0 & -\ii\sinh w & \cosh w  
\end{pmatrix} \in \Oo(4, \Cc)\,.
\]
with $w = s+ \ii t$.  
\end{theorem}
\begin{proof}
Let  $\mu\in\C\setminus\{0,1\}$ and  put,  as
 usual,  $a= \frac{\mu + \mu\invers}2, b = i
\frac{\mu\invers-\mu}2$. 
The simple factor
dressing of $f$ with parameter $\mu$ is given by
\[
f^\mu = - f\frac{a-1}2 +  f^*\frac b2 -
\frac{b}{a-1}(f\frac b2 +
 f^*\frac{a-1}2) = T_1(f)+ T_2(f^*)\,,
\]
where  
\[
T_1(v) = -v \frac{a-1}2 - \frac b{a-1} v
\frac b2 \quad \text{ and } \quad T_2(v) = v \frac b2 -  \frac b{a-1}
v \frac{a-1}2,  \quad v\in\H\,.
\]
 
Next, we observe  for $
v\in \C =\Span_\R\{1, i\}$ that  
\[
T_1(v) = -\frac 12(v(a-1) + \frac{b}{a-1} v b )= v \quad \text{
  and }  \quad 
T_2(v) = \frac 12(v b - \frac{b}{a-1}v(a-1)) =0 \,,
\]
where we  used that $a^2 + b^2=1$.  To compute $T_1(v), T_2(v)$ for
$v\in \C j =\Span\{j, k\}$ we recall  that
\[
\frac{\bar\mu-1}{\bar\mu(1-\mu)} =e^{s+it}
\]
by definition of $s$ and $t$.
 Thus, with $a-1 = \frac{(1-\mu)^2}{2\mu}, b = i
\frac{1-\mu^2}{2\mu}$ we have
\[
\frac{|a-1|^2 + |b|^2}{2|a-1|} =\cosh s,  \quad
 \frac{\Im(b(\bar a-1))}{|a-1|} = \sinh s, \quad \text{ and } \quad
\frac{a-1}{|a-1|} =- e^{-it}\,.
\]
Therefore, since $w v = v\bar w$ for every $w\in \C$ and $v\in \C j$,
we see
\[
T_1(v) = -\frac 12(v(a-1) + \frac{b}{a-1} v  b
)= ve^{-it}\cosh s 
\]
and
\[ 
T_2(v) = 
\frac 12(v b - \frac{b}{a-1}v(a-1)) =
-v ie^{-it}\sinh s\,.
\]

Decomposing $f = (f_0 + f_1 i) + (f_2 j + f_3 k)$ and $f^* = (f^*_0 +
f^*_1 i) + (f^*_2 j + f^*_3 k)$, the simple factor dressing of $f$
with parameter $\mu$ is then given by
\[
 f^\mu =T_1(f) + T_2(f^*) =  (f_0 + f_1 i) + (f_2  j + f_3k) e^{-it}  \cosh s-  (f_2^* j +
f_3^*k) i e^{-it} \sinh s
\]
which gives (\ref{eq: sfd explicit}).
The final statement follows by a straight forward computation of the
holomorphic null curve.
\end{proof}

By Proposition \ref{prop: SFD given
  by parameter mu}  we immediately see that the general simple factor dressing
is a Goursat transformation, too.

\begin{theorem}
\label{thm:SFD is Goursat}
  The simple factor dressing  of a minimal surface $f: M \to \R^4$  is a
  Goursat transformation of $f$. 
\end{theorem}
\begin{proof}
 
Let $\mu \in\C\setminus\{0,1\}$, then by Proposition \ref{prop: SFD
  given by parameter mu} the simple factor dressing $\hat
f$ with parameters $(\mu, m, n)$ is given by
\[
\hat f = \mathcal{R}_{n, m}((\mathcal{R}_{n,m}\invers(f))^\mu)
\]
where $\mathcal{R}_{n,m}\in \SO(4, \R)$ is the map $v\mapsto nvm\invers$ and
$(\mathcal{R}_{n,m}\invers(f))^\mu$ is the simple factor dressing of
$\tilde f = n\invers f m$ with
parameter $\mu$. If $\Phi$ denotes the  holomorphic null curve of $f$
then the null curve of $\mathcal{R}_{n,m}\invers(f)$ is
 $\mathcal{R}_{n,m}\invers\Phi$ since $\mathcal{R}_{n,m}$ is real.
But then the holomorphic null curve of the simple factor dressing of
$\mathcal{R}_{n,m}\invers(f)$ with parameter $\mu$ is  $\mathcal{L}^\mu\mathcal{R}_{n,m}\invers
\Phi$ by (\ref{eq: null curve of sfd}). Thus,
the holomorphic null curve of the simple factor dressing $\hat f$ with
parameters $(\mu, m, n)$ is given by 
\[
\hat \Phi = \mathcal{R}_{n,m} \mathcal{L}^\mu \mathcal{R}_{n,m}\invers \Phi\,.
\]
But $ \mathcal{R}_{n,m} \mathcal{L}^\mu \mathcal{R}_{n,m}\invers \in \Oo(4,\Cc)$ so
that $\hat f$ is a Goursat transformation of $f$.
\end{proof}

Note that the simple factor dressing is a special case of the Goursat
transformation: its matrix $\mathcal{A}\in \Oo(4, \Cc)$ has $\det
\mathcal{A} =1$ and special behaviour of the eigenspaces.

As before  the L\'opez--Ros deformation
can be given in terms of the surface and its conjugate which
immediately shows that it is a special case of the simple factor
dressing:

\begin{theorem}
\label{thm:lopezros is SFD}
Let $f: M \to \R^4$ be a minimal surface in $\R^4$ with conjugate
surface $f^*$  and let $\sigma=e^{s+\ii t}
\in\Cc_*$. Then the 
 L\'opez-Ros deformation $f^\sigma$  of $f$ is given by  
\begin{equation}
\label{eq:LopezRos}
f^\sigma = \begin{pmatrix} 
f_0 \\
\cos t \, (f_1\cosh s  - f_2^*\sinh s) - \sin t \, (f_2\cosh s + f_1^*\sinh
s)\\
\sin t  \, (f_1\cosh s  - f_2^*\sinh s) + \cos t \, (f_2\cosh s + f_1^*\sinh
s)\\
f_3
\end{pmatrix}
\end{equation}

where $f_l$ and $ f_l^*$ are the coordinates of $f = f_0 + f_1 i + f_2 j + f_3 k$ and  
$f^*= f_0^* + f_1^* i + f_2^* j + f_3^* k$ respectively.

 In particular,
the L\'opez-Ros
deformation $f^\sigma$ of $f$ with parameter $\sigma = e^{s+ \ii t}\in
\Cc_*, |\sigma|\not=1,$ is the simple factor dressing $\hat f$ of $f$
with parameters $(\mu, m, m)$ where $\mu =\frac{1-e^{-(s+i t)}}{1-
  e^{s-it}} \in\C\setminus\{0,1\}$ and $m=\frac{1-i-j-k}2\in S^3$.
\end{theorem}

From this we see again that the L\'opez--Ros deformation is a trivial rotation in the
$ij$--plane if $\sigma\in S^1\subset \Cc$.  Moreover, if $\sigma
\in\R$ then $\mu= -\frac 1\sigma$.
   
\begin{proof}
Let $f^\sigma$ be the Lopez--Ros deformation of $f$ with parameter
$\sigma =e^{s+\ii
    t}\in\Cc_*, |\sigma|\not=1$.
The first equation (\ref{eq:LopezRos}) is an analogue computation as
in the proof of Theorem \ref{thm:Goursat}. 

By assumption $|\sigma| \not=1$ so
  that  $\mu =
  \frac{1-e^{-(s+it)}}{1-e^{s-it}} \in\C\setminus\{0,1\}$ is well
  defined. Put,  as
 usual,  $a= \frac{\mu + \mu\invers}2, b = i
\frac{\mu\invers-\mu}2$. 

 By
Proposition~\ref{prop: SFD given by parameter mu}, the simple factor
dressing with parameters $(\mu, m, m)$ is  $\hat f=
\mathcal{R}_{m,m}((\mathcal{R}_{m,m}\invers(f))^\mu)$ where $(\mathcal{R}_{m,m}\invers(f))^\mu$ is
the simple factor dressing of $\tilde f = \mathcal{R}_{m,m}\invers(f) = m\invers
f m$ with parameter $\mu$.  Decomposing $f = (f_0 + f_1 i) + (f_2 j +
f_3 k)$ we have $ \tilde f= (f_0 + f_3 i)+ (f_1 j + f_2 k) $ for
$m=\frac{1-i-j-k}2$, and the simple factor dressing of $\tilde f$ with
parameter $\mu$ is given by Theorem \ref{thm: explit SFD with paramter
  mu} as
\[
\tilde f^\mu = (f_0 + f_3 i) + (f_1 j + f_2k) e^{-it}  \cosh s-  (f_1^* j +
f_2^*k) i e^{-it} \sinh s\,.
\]
Therefore, the simple factor dressing of
$f$ with parameters $(\mu, m, m)$ 
\[
\hat f = \mathcal{R}_{m, m} \tilde f^\mu = f_0 + f_3 k + (f_1i  + f_2j)  e^{-kt}
\cosh s-  (f_1^*i  +
f_2^*j) k e^{-ik} \sinh s =f^\sigma
\]
is indeed by (\ref{eq:LopezRos})  the L\'opez--Ros deformation of $f$.

\end{proof}

\begin{rem} In particular, with Proposition \ref{prop: SFD given by
    parameter mu} we see that all simple factor dressings of a minimal
  surface are given, up to rotations, by the L\'opez--Ros deformation
  applied to a rigid motion of $f$. 
\end{rem}

If  $f: M \to \R^4$ is a periodic minimal surface then the
periods of the simple factor dressing $f^\mu$ with parameter $\mu$ are
immediately given by  the explicit formulation (\ref{eq: sfd explicit}):

\begin{cor}
\label{cor:periods}
If $f: M \to \R^4$ is a periodic minimal surface with translational
periods $\gamma^*f = f + \tau$ for $\gamma\in\pi_1(M)$,
$\tau=(\tau_0, \tau_1, \tau_2, \tau_3)$, then the simple factor
dressing $f^\mu$ of $f$ with parameter $\mu$ is periodic with
$\gamma^* f^\mu = f^\mu + \tau^\mu$ where
\[
\tau^\mu = \begin{pmatrix} \tau_0 \\ \tau_1 \\ 
\cos t \, (\tau_2 \cosh s - \tau_3^* \sinh s) - \sin t \, (\tau_3 \cosh s +
\tau_2^* \sinh s) \\
\sin t\, ( \tau_2\cosh s - \tau_3^* \sinh s) + \cos t\, (\tau_3 \cosh s +
\tau_2^* \sinh s)
\end{pmatrix}\,.
\]
Here $\tau^*=(\tau_0^*, \tau_1^*, \tau_2^*, \tau_3^*)$ denote the periods of a conjugate surface $f^*$ of $f$,
that is, $\gamma^* f^* = f^* + \tau^*$,  and
$s =- \ln|\mu|, t =\arg\frac{\bar \mu-1}{\bar \mu(1-\mu)}$. \\

In particular, $ f^\mu$ is closed along $\gamma\in\pi_1(M)$, that is,
$\gamma^*f^\mu =f^\mu$, if and only if
\[
\tau_0= \tau_1 = 0 \quad \text{ and } \quad \begin{pmatrix} \tau_2\\ \tau_3
\end{pmatrix}
 = \begin{pmatrix} \tau_3^* \\ -\tau_2^*
\end{pmatrix} \tanh s\,.
\]
\end{cor}
From this, we can immediately compute the periods of all simple factor
dressings by Proposition \ref{prop: SFD given by parameter mu}.

In particular, assume that $f: M \to \R^4$ is single--valued on $M$,
and  that there exist $ m, n\in \H_*$ such that all periods of the
conjugate surface $f^*$ can be rotated simultaneously into the $1,i$--plane,
that is,
\[
\mathcal{R}_{n, m}\invers \tau^*_\gamma \in \Span\{1, i\}
\]
for all $\gamma\in\pi_1(M)$ where $\gamma^*f^* = f^* + \tau^*_\gamma$.
Then all minimal surfaces in the complex 1--parameter family given by
the simple factor dressings with parameters $(\mu, m, n)$,
$\mu\in\C\setminus\{0,1\}$, are single--valued on $M$.
\\

Finally, since a simple factor dressing of a finite total
curvature minimal surface is given by a Goursat transformation, it has again 
finite total curvature: 

\begin{theorem}
If $f: M \to \R^4$ has finite total curvature and if the simple factor
dressing $\hat f: M \to \R^4$ of $f$ with parameters $(\mu, m, n)$ is
single--valued on $M$ then $\hat f$ has finite total curvature.
\end{theorem}
\begin{proof}
  Since $f$ has finite total curvature, we can assume by Theorem
  \ref{thm:FTC} that $M =\bar M \setminus\{p_1, \ldots, p_r\}$ where
  $\bar M$ is a Riemann surface punctured at finitely many
  $p_i$. Moreover, if $\Phi$ denotes the holomorphic null curve of $f$
  then we can assume that $d\Phi$ extends meromorphically into the
  $p_i$. Since the simple factor dressing is a Goursat transformation,
  the holomorphic null curve $\hat \Phi$ of $\hat f$ is given by $\hat
  \Phi = \mathcal{A} \Phi$ with $\mathcal{A}\in \Oo(4,\Cc)$. Thus, $d\hat\Phi$ extends
  meromorphically into the punctures $p_i$.
\end{proof}

\subsection{Simple factor dressing in $\R^3$}

Given a minimal surface in $\R^3$ we now discuss when the simple
factor dressing $\hat f$ is a minimal surface in 3--space. Considering
a surface in $\R^3=\Im\H$ as a surface in $\H$ with vanishing real
part, we immediately see with Theorem \ref{thm: explit SFD with
  paramter mu}  that   a simple
factor dressing of $f$ with parameter $\mu\in\C\setminus\{0,1\}$ gives a
minimal surface 
\begin{equation}
\label{eq:sfd in r3 quaternion}
f^\mu = \begin{pmatrix} f_1 \\ 
\cos t \, (f_2\cosh s  - f_3^*\sinh s) - \sin t\,  (f_3\cosh s + f_2^*\sinh
s) \\
\sin t \, (f_2\cosh s  - f_3^*\sinh s) + \cos t\,  (f_3\cosh s + f_2^*\sinh
s)
\end{pmatrix}
\end{equation}
in $\R^3$ where $s=-\ln|\mu|,
t=\arg\frac{\bar\mu-1}{\bar\mu(1-\mu)}$. Moreover, $f_j$ and $f_j^*$
are the coordinates of $f=if_1 + jf_2+kf_3$ and $f^* = if_1^* + jf_2^*
+ kf_3^*$ respectively.  Since any simple factor dressing $\hat f$ of
$f$ with parameters $(\mu, m, n)$ is given by a simple factor dressing
with parameter $\mu$ and an operation of $\mathcal{R}_{n,m}\in \SO(4, \R)$, we see
from (\ref{eq: sfd explicit})
that $\hat f$ is in 3--space if $\mathcal{R}_{n,m}$ stabilises $\C =\Span_\R\{1
, i\}$. In particular:

\begin{theorem} 
Let $f: M \to \R^3$ be minimal. 
The simple factor dressing  $\hat f$ with parameters $(\mu, m,
n)$ with $m=n\lambda$, $\lambda\in\C_*$,  is a minimal surface $\hat
f:M \to\R^3$  in 3--space.
\end{theorem}

As before, we also obtain the periods of the simple factor dressing:
 
\begin{cor}
\label{cor: periods in 3 space}
If $f: M \to \R^3$ is a periodic minimal surface with  $\gamma^*f = f +
\tau$ for $\gamma\in\pi_1(M)$, $\tau=(\tau_1, \tau_2, \tau_3)$,  then
the simple factor dressing $f^\mu$  of $f$ with parameter $\mu$ is
periodic with  $\gamma^* f^\mu = f^\mu + \tau^\mu$ where  
\[
\tau^\mu = \begin{pmatrix} \tau_1 \\  
\cos t \,(\tau_2 \cosh s - \tau_3^* \sinh s) - \sin t \,(\tau_3 \cosh s +
\tau_2^* \sinh s) \\
\sin t\, ( \tau_2\cosh s - \tau_3^* \sinh s) + \cos t\, (\tau_3 \cosh s +
\tau_2^* \sinh s)
\end{pmatrix}\,.
\]

Here $\tau^*=(\tau_1^*, \tau_2^*, \tau_3^*)$ denote the periods of a conjugate surface $f^*$ of $f$,
that is, $\gamma^* f^* = f^* + \tau^*$,  and
$s =- \ln|\mu|, t =\arg\frac{\bar \mu-1}{\bar \mu(1-\mu)}$. \\

In particular,  $ f^\mu$ is closed along $\gamma$   if and
only if 
$
\tau_1 =0$ and $\begin{pmatrix} \tau_2 \\ \tau_3
\end{pmatrix}
 = \begin{pmatrix} \tau_3^* \\ -\tau_2^*
\end{pmatrix} \tanh s\,.
$
\end{cor}

We can also investigate the behaviour of simple factor dressings in
$\R^3$ at ends:

\begin{theorem}
\label{thm:SFDends}
Let $f: M\to \R^3$ be a minimal surface on a
punctured disc $M=D\setminus\{p\}$ and $\hat f: \tilde M \to
\R^3$ its simple factor dressing with parameter $(\mu, m, m)$,
$m\in S^3$.

Then the following hold:
\begin{enumerate}
\item If $f$ has a planar end at $p$  then   $\hat f: M \to\R^3$
 is single--valued on $M$ and $\hat f$ has a planar end at $p$.
\item If $f$ has a catenoidal end at $p$ and $\hat f: M \to \R^3$ is
  single--valued on $M$ then $\hat f$ has a catenoidal end at $p$.
\end{enumerate}
\end{theorem}

\begin{proof}
Let $p$ be an complete, embedded, finite total curvature end of
$f$.  We can assume that the end  of $f$ at $p$ is vertical: if the end is not
vertical, let $n\in \H_*$ such that $\tilde f = \mathcal{R}_{n, n}\invers f$ has vertical
normal at $p$. Since $\mathcal{R}_{m, m} = \mathcal{R}_{n, n} \circ \mathcal{R}_{n\invers m,n\invers m}$ and
$\hat f = \mathcal{R}_{m,m}((\mathcal{R}_{m,m}\invers(f))^\mu)$ by Proposition \ref{prop: SFD
  given by parameter mu}, the simple factor dressing of $f$ is 
up to rotation given by the simple factor dressing of $\tilde f$ with
parameters $(\mu, n\invers m,  n\invers m)$.

 In a conformal coordinate $z$ on the punctured disk
$D\setminus\{0\}$,  we know from
\cite{hoffman_karcher}, see also  Theorem
\ref{thm:FTCend}, that the holomorphic null curve $\Phi$ of $f$ has  
$\ord_{z=0}d\Phi =-2$ and $\res_{z=0} d\Phi=-(0, 0, 2\pi \alpha)$
where $\alpha\in\R$ is the logarithmic growth.

By \cite{hoffman_karcher} the periods of the conjugate surface $f^*$
around the end are given by $\res_{z=0} d\Phi$.  Therefore, if $f$ has
a planar end then $f^*$ is single--valued on $M$, and if $f$ has a catenoidal
end then the periods of $f^*$ are given by $-2\pi \alpha k$. By
Proposition \ref{prop: SFD given by parameter mu} and Corollary
\ref{cor: periods in 3 space}, the simple factor dressing is
single--valued for all parameters if $p$ is a planar end. Otherwise,
it is single--valued for parameters $(\mu, m, m)$ such that $m\invers
k m =\pm i$, that is, $m=(1\mp j)\lambda$ with $\lambda\in \C$.

We know from Corollary \ref{cor:complete} that the simple factor
dressing preserves completeness, that is, the end of $\hat f$ at $p$
is complete.  Since the simple factor dressing $\hat f$ is a Goursat
transformation, the holomorphic null curve of $\hat f$ is given by
$\hat \Phi = \mathcal{A} \Phi$ with $\mathcal{A}\in \Oo(3, \Cc)$ and thus, $\ord_{z=0}
d\hat\Phi =-2$.  At a planar end we have $\res_{z=0} d\hat \Phi =
\res_{z=0} d\Phi = (0,0,0)$. Therefore, $\hat f$ has a planar end by
Theorem \ref{thm:FTCend}.

From Proposition \ref{prop: SFD
  given by parameter mu} and  (\ref{eq: null curve of sfd}) we know that
 $\mathcal{A} = \mathcal{R}_{m,m} \mathcal{L}^\mu
\mathcal{R}_{m,m}\invers$   where  
\[
\mathcal{L}^\mu= \begin{pmatrix} 1&0&0\\
 0 &\cosh w  & \ii\sinh w\\
0 & -\ii\sinh w & \cosh w  
\end{pmatrix}\,, \quad w = s+ \ii t\,.
\] 
At a catenoidal end, a single--valued simple factor dressing has
parameters $(\mu, m, m)$ with $m = (1\mp j)\lambda$,
$\lambda\in\C$. Thus, 
$\mathcal{R}_{m,m}\invers\res_{z=0}d\Phi$ is an eigenvector with eigenvalue 1 of
the matrix $\mathcal{L}^\mu$. But then  $\res_{z=0}d\hat \Phi = \res_{z=0} d\Phi$ is real,
and the end of the simple factor dressing is catenoidal by Theorem
\ref{thm:FTCend}.

\end{proof}

Again, we obtain from Theorem \ref{thm:lopezros is SFD} the link to
the L\'opez-Ros deformation:

\begin{theorem} 
  Let $f: M\to\R^3$ be a minimal surface in $\R^3$ with conjugate
  surface $f^*$. The L\'opez-Ros deformation with complex parameter
  $\sigma=e^{s+\ii t}\in\Cc_*, |\sigma| \not=1,$ is the simple factor
  dressing of $f$ with parameters $(\mu, m, m)$ where $\mu=
  \frac{1-e^{-(s+it)}}{1-e^{s-it}}\in\C$ and $m=\frac{1-i-j-k}2$.
\end{theorem}

We obtain as a
consequence of the last two theorems  the following well--known result
\cite{lopez_ros}:

 \begin{cor}
Let $f:  M\to \R^3$ be a minimal surface on a punctured disk $M =
D\setminus\{p\}  $ and $f^\sigma: \tilde M \to
\R^3$ a L\'opez-Ros deformation with parameter $\sigma$. 

Then the following hold:
\begin{enumerate}
\item If $f$ has a planar end at $p$  then   $f^\sigma: M \to \R^3$
 is single--valued on $M$ and $f^\sigma$ has a planar end at $p$.
\item If $f$ has a catenoidal end at $p$ and $ f^\sigma: M \to\R^3$  is single--valued 
 on $M$ then $f^\sigma$ has a catenoidal end at $p$.
\end{enumerate}
 \end{cor}
 Note that if $f$ is a minimal surface with vertical catenoidal end at
 $p$, then the proof of Theorem \ref{thm:SFDends} shows that the
 L\'opez--Ros deformation is single--valued since $2m=1-i-j-k
 =(1-j)(1-i)$. In particular, the L\'opez--Ros deformation of $f$ has
 a catenoidal end at $p$, too.



\section{Darboux transforms of  minimal surfaces}

We now connect the simple factor dressing of a minimal surface with
its Darboux transform. Previous results \cite{cmc}, \cite{bohle},
\cite{hsl}, \cite{burstall_quintino} seemed to indicate that the
Darboux transformation preserves a surface class which is given by a
harmonicity condition as long as the Darboux transform is given by a
parallel section of the associated family of the harmonic map. These
Darboux transforms are the so--called $\mu$--Darboux transform. We
will show that this does not hold for minimal surfaces: a
$\mu$--Darboux transform of a minimal surface is a (non--minimal)
Willmore surface in $\R^4$. However, the Darboux transforms are still
closely related to the simple factor dressing of the minimal surface,
and in particular, a $\mu$--Darboux transform is given by complex
holomorphic data.

Let us recall that two isothermic immersions $f,  f^\sharp: M \to \R^4$
form a \emph{classical Darboux pair} \cite{Darboux} if there exists a
sphere congruence enveloping both $f$ and $f^\sharp$.  In particular, a
minimal surface $f: M\to\R^3$ in $\R^3$ is isothermic, and a classical
Darboux transform $f^\sharp = f+ T$ of $f$ is given
\cite{darboux_isothermic} by a solution $T$ of the Riccati equation
\begin{equation}
\label{eq:Riccati}
dT =   -df  + T (dN) r  T\invers
\end{equation}
where $N$ is the Gauss map of $f$ and  $r\in\R_*$.   

By weakening the enveloping condition the notion of a classical
Darboux transformation has been extended in \cite{conformal_tori} to
any conformal immersion $f: M\to S^4$. In case of a conformal torus
$f: T^2\to S^4$, there exists at least a Riemann surface worth of
Darboux transforms $f^\sharp: T^2\to S^4$ of $f$. This way, one
obtains a geometric interpretation of the spectral curve $\Sigma$ of
the conformal torus $f$ as the normalisation of the set of closed
Darboux transforms of $f$.

In this paper however, we are only interested in local theory, so we
will assume from now on that $M$ is simply connected. Denoting as
before by 
\[
L =\begin{pmatrix}
  f\\ 1
\end{pmatrix}
\H
\] 
the line bundle $L\subset \trivial 2$ of a conformal immersion $f: M
\to \R^4$, the left normal $N$ of $f$ induces a quaternionic
holomorphic structure on the bundle $\trivial 2/L$ via
\[
D(e\alpha) = e \frac 12(d\alpha + N*d\alpha)\,.
\]
Here we identify $\trivial 2/L = e\H$ via $(\pi_L)|_{e\H}: e\H \to
\trivial 2/L$ where $\pi_L: \trivial 2 \to \trivial 2/L$ is the
canonical projection.

A section $e\alpha\in\Gamma(e\H)$ is called \emph{holomorphic} if
$D(e\alpha)=0$, or, equivalently, $*d\alpha = N d\alpha$. Since $*df=
Ndf$, for any holomorphic section $e\alpha$ there is $\beta: M \to\H$
with $d\alpha = -df \beta$. In particular, there exists a \emph{prolongation}
of the holomorphic section $e\alpha$, that is, a lift $\varphi =
e\alpha + \psi \beta\in\Gamma(\trivial 2)$  such that
$d\varphi \in\Omega^1(L)$ where $\psi =\begin{pmatrix}
  f\\ 1
\end{pmatrix}$.\\

A \emph{(generalised) Darboux transform} $f^\sharp: M \to S^4$ of the
conformal immersion $f: M \to \R^4$ is then given
\cite{conformal_tori} by the prolongation $\varphi\in\Gamma(\trivial
2)$ of a holomorphic section  of $e\H$: away from the zeros
of $\varphi$, the (singular) Darboux transform is given by the line
bundle
\[
L^\sharp = \varphi\H\,.
\] 

Conversely, we obtain Darboux transforms of a minimal immersion $f$ by
finding sections $\varphi\in\Gamma(\trivial 2)$ with
$d\varphi\in\Omega^1(L)$: writing $\varphi = e\alpha + \psi \beta$ we
see that $\alpha\not=0$ since otherwise $0=\pi_L d\varphi = edf \beta$
implies $\varphi =0$. But then  $\pi_L \varphi = e\alpha$ is a non--trivial holomorphic
section and $\varphi$ is its prolongation.

If $f: M \to \R^4$ is minimal we have the associated family
(\ref{eq:associated family of S}) of flat connections
\[
d^S_\lambda = d +(\lambda-1)A\oz + (\lambda\invers-1)A\zo\,, \quad \lambda\in\C_*
\]
of the harmonic conformal Gauss map $S$ of $f$, where $A$ is the Hopf
field of $S$.  Since $\im A\subset L$ we see that for fixed
$\mu\in\C_*$ every $d^S_\mu$--parallel section
$\varphi\in\Gamma(\trivial 2)$ has $d\varphi\in\Omega^1(L)$, and thus
$L^\sharp = \varphi\H$ is a Darboux transform of $f$, a so--called
\emph{$\mu$--Darboux transform} of $f$.

If $f$ has constant right normal $R$ then by (\ref{eq: Hopf}) the Hopf
field $A$ vanishes and $d_\mu = d$ for all
$\mu\in\C\setminus\{0,1\}$. That is, all $\mu$--Darboux transforms of
$f$ are in this case the constant sections $\Gamma(\trivial 2)$.
Therefore, from now on we will assume that $f$ is not the twistor
projection of a holomorphic curve in $\CP^3$.

With Proposition \ref{prop: parallel sections} at hand, we can again
discuss all $\mu$--Darboux transforms of $f$. If $\varphi =en,
n\in\H_*$ then the corresponding $\mu$--Darboux transform is the
constant point $\infty = e\H$. On the other hand, every non--constant
$d^S_\mu$--parallel section $\varphi\in\Gamma(\trivial 2),
\mu\in\C\setminus\{0,1\}$, is given by
\[
\varphi = e\alpha + \psi \beta
\]
with $e = \begin{pmatrix} 1 \\0
    \end{pmatrix}, 
      \psi =\begin{pmatrix} f\\ 1
\end{pmatrix}$ and
\[
\alpha = -fm\frac{b}{a-1} - f^* m, \quad \beta =Rm + m \frac{b}{a-1}\,.
\]
Here $f^*$ is a conjugate surface of $f$ and $a= \frac{\mu+\mu\invers}2,
b= i\frac{\mu\invers-\mu}2$ and $m\in\H_*$.
The $\mu$--Darboux transform is in this case given by 
\[
L^\sharp = (e\alpha +\psi\beta) \H\,,
\]
where    $\beta
$ is nowhere vanishing.   Therefore, the $\mu$--Darboux
transform is given by the affine coordinate $f^\sharp = f+ T$ with 
\begin{equation}
\label{eq:T Darboux}
T=
\alpha\beta\invers =  -(f^* + f \frac{\hat b}{\hat a-1})(R+\frac{\hat b}{\hat a-1})\invers
\end{equation}
and  $\hat a = m a m\invers$, $\hat b = m b m\invers$. We summarise:

\begin{theorem}
\label{thm:mu-DT}
Every non--constant $\mu$--Darboux transform of a minimal surface $f: M
\to \R^4$, $M$ simply connected,  is given by  
\begin{equation}
\label{eq: hat f}
f^\sharp = (fR-f^*)(R+\rho)\invers\,,
\end{equation}

where
  $f^*$ is a conjugate surface of $f$ 
and $\rho = m \frac{i(1+\mu)}{1-\mu} m\invers$ with $\mu\in\C\setminus
\{0, 1\}$\,.  
\end{theorem}
Note that a $\mu$--Darboux transform depends non--trivially on the
choice of the conjugate surface $f^*$.

Moreover, we see with $T= \alpha\beta\invers$ and $d\alpha=-df\beta$ that
\begin{equation}
\label{eq:gen Riccati}
dT = -df - \alpha\beta\invers d\beta \beta\invers\,.
\end{equation}
Since $d\beta = dR m$ by (\ref{eq:dbetam}) this is a generalised
Riccati equation (away from the zeros of $\alpha$)
\[
dT = -df - T dR m\alpha\invers T
\]
with $m\alpha\invers$  non--constant. In particular, if $f:
M \to\R^3$ is minimal in $\R^3$ then the Gauss map $N$ is the right
normal of $f$, and the above equation generalizes
(\ref{eq:Riccati}). Note however that non--constant $\mu$--Darboux
transforms of a minimal surface $f: M\to\R^3$ are never classical; for
$f^\sharp = f+T$ to be classical, $ m\alpha\invers\in\R_* $ has to hold
but  $\alpha$ is not constant. However, we will show
that the $\mu$--Darboux transformation still preserves geometric
information of the minimal surface: it is a Willmore surface which is
given  by a minimal surface in the
associated family of $f$:

\begin{theorem}
\label{thm:asociated darboux}
Every (non--constant) $\mu$--Darboux transform $f^\sharp: M\to\R^4$ of a
minimal surface $f: M \to\R^4$ is an associated Willmore surface of a
minimal surface in the right associated family of $f$.
\end{theorem}
\begin{proof}
 Consider as before the minimal surface
\[
h = f\hat b + f^*(\hat a-1)
\]
in the right associated family of $f$ where $\hat a =
m\frac{\mu+\mu\invers}2m \invers , \hat b = m i \frac{\mu\invers-\mu}2
m\invers$ and $m\in\H_*$. From (\ref{eq:right normal
  associated simplified}) we see that the right normal of the minimal surface $h$ is
\[
R_h =  (\rho +R) R (\rho+R)\invers\,.
\]
where $\rho = m
\frac{i(1+\mu)}{1-\mu} m\invers$.   Using the conjugate
surface
\[
h^* =  f^*\hat b- f(\hat a-1)
\]
of the minimal surface $h$,  the associated Willmore surface $g$ of
$-\frac h2$, see
Theorem \ref{thm:twistor minimal},  is given by
\begin{equation*}
g = -\frac 12(hR_h - h^*) =
 (fR - f^*)(\rho + R)\invers\,.
\end{equation*}
Here we used again that $\hat a^2+ \hat b^2=1$.  In other words, $g$
is by (\ref{eq: hat f}) the $\mu$--Darboux transform $f^\sharp$ of
$f$.  By Theorem \ref{thm:mu-DT} every $\mu$--Darboux transform arises
this way.
\end{proof}

By Theorem \ref{thm:twistor minimal} every twistor projection of a
holomorphic curve in $\CP^3$ which is not minimal in $\R^4$ is the
associated Willmore surface of a minimal surface. In particular, there
is an induced transformation on Willmore surfaces which are given by
complex holomorphic curves in $\CP^3$:

\begin{cor}
Let $g: M \to\R^4$ be the twistor projection of a holomorphic curve in
$\CP^3$ and $R$ its right normal. Then 
\[
g^\sharp = g(\rho-R)\invers
\]
for $\rho\in\H, (\rho-R)\not=0,$ has right normal $R^\sharp =
(\rho - R)R(\rho - R)\invers$. In particular, $g^\sharp$ is the twistor
projection of a holomorphic curve in $\CP^3$.
\end{cor}

By Theorem \ref{thm:asociated darboux} and Theorem  \ref{thm:twistor
  minimal} the right normal of a Darboux
transform is given by the negative of the right normal of a minimal
surface in the right associated family:

\begin{cor}
Let $f: M \to \R^4$ be a minimal surface with right normal $R$.

Then the  right normal $R^\sharp$ of a $\mu$--Darboux
  transform $f^\sharp$ of  $f$ is given by
  a simple factor dressing of $-R$, and vice versa.
\end{cor}

\section{Examples}

We conclude this paper by demonstrating some of our results for
well--known examples of minimal surfaces, including surfaces with one
planar end, the first Scherk surface and a punctured torus. In
particular, as we can control the periods and the end behaviour at
punctures of simple factor dressings by choosing appropriate
parameters, we obtain simple factor dressings which are minimal
surfaces with one planar end, doubly-periodic surfaces and minimal
puncture tori respectively.  Our first example is the catenoid for
which all computations can be done completely explicitly.

The images were implemented by using the software jReality and the
jTEM library of TU Berlin.

\subsection{The catenoid}

We consider the catenoid $f: \Cc\to \R^3$ in the conformal
parametrisation
\[
f(x,y) = i x + j \cosh x \, e^{-iy}\,,
\]
where $ z=x+ i y$ is the standard conformal coordinate on $\C$ with
$*dz= i dz$. 
The left and right normal of the catenoid are given by the  Gauss map
\[
R(x,y) = N(x,y) = \frac 1{\cosh x}( i \sinh x - j e^{-iy})\,.
\]
A conjugate surface is the helicoid 
\[
f^*(x,y) = iy + ji\sinh x \, e^{-iy} 
\]
and, identifying $z= x  + \ii y$,  we obtain the   holomorphic null curve
\[
\Phi(z) = f(z) + \ii f^*(z) = \begin{pmatrix} z \\ \cosh z \\- \ii \sinh z
\end{pmatrix}: \C \to \Cc^3
\,,
\]
and the Weierstrass data $g(z)=\frac{e^z-\ii}{e^z+\ii}$ and
$\omega=-\frac{\ii}2 e^{-z}(e^z+\ii)^2 dz$.  
The right  associated family $f_{p,q} = fp + f^*q$ is given (\ref{eq: fmn}) by
\[
f_{p,q}(x,y) = i(px + qy) + je^{-iy}(p\cosh x + i q\sinh x),
\quad p, q\in\H\,.
\]
\begin{figure}[H]
\includegraphics[height=4.5cm]{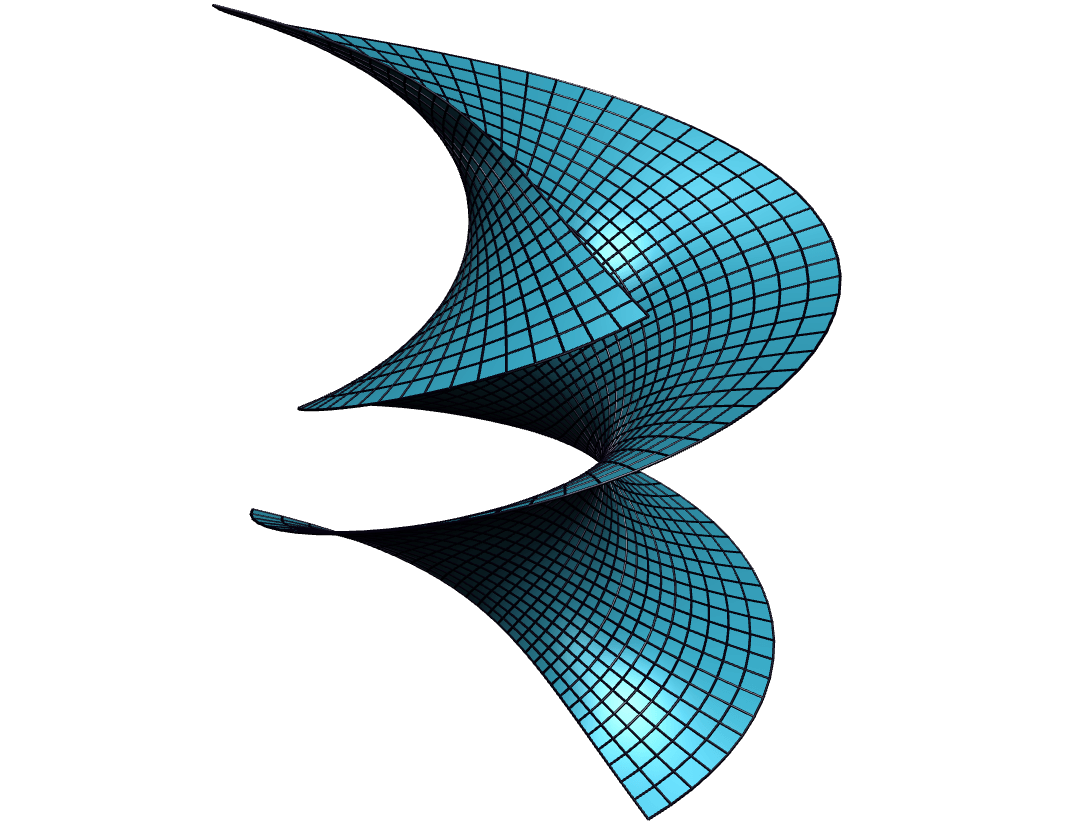} \quad
\includegraphics[height=4.5cm]{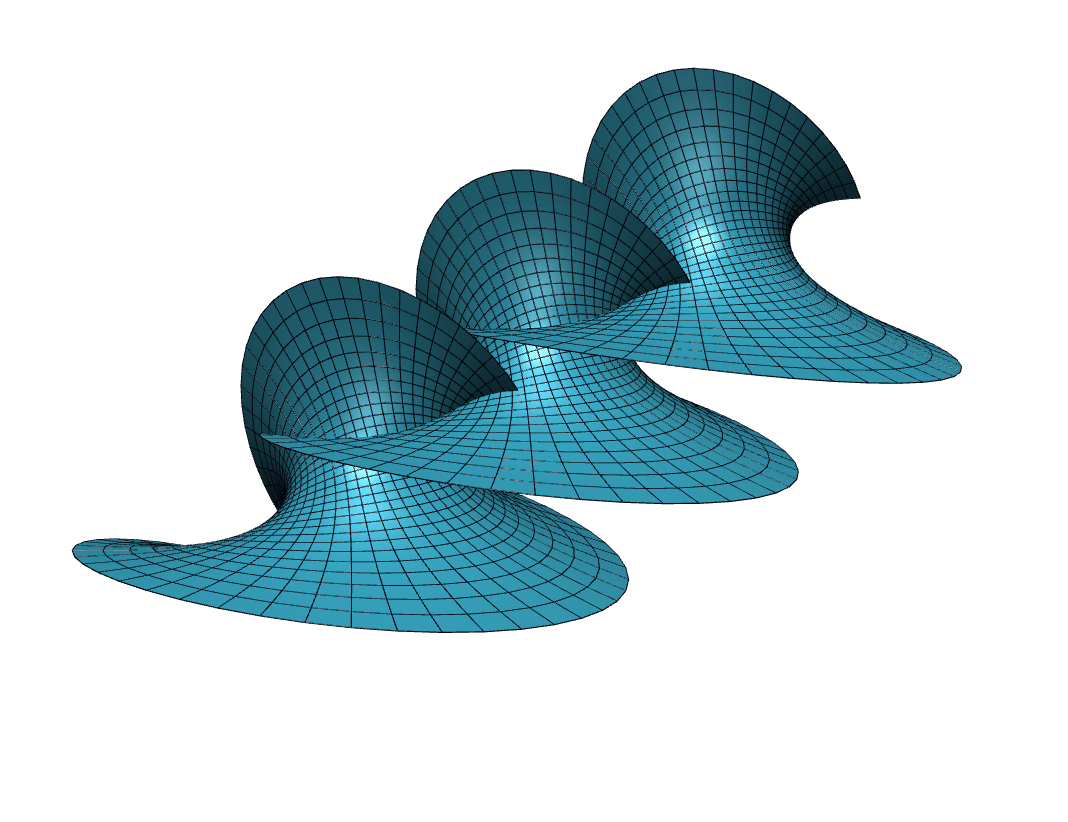}\\ \quad \\
\caption{Elements  $f_{\frac 2{\sqrt 6}, \frac 1{\sqrt 6} +\frac
    i{\sqrt 6}}$ and $f_{\frac 1{\sqrt 2}, \frac 1{2\sqrt 2}(1 +
    i-j-k)}$ of the right associated family of
  the catenoid, orthogonally projected into $\R^3$.}
\end{figure}
and the left associated family $f^{p,q} = p f + qf^*$  by 
\[
f^{p,q}(x,y) =(px + qy)i +(p\cosh x - q i\sinh x) je^{-iy},
\quad p, q\in\H\,.
\]

\begin{figure}[H]
\includegraphics[height=4.5cm]{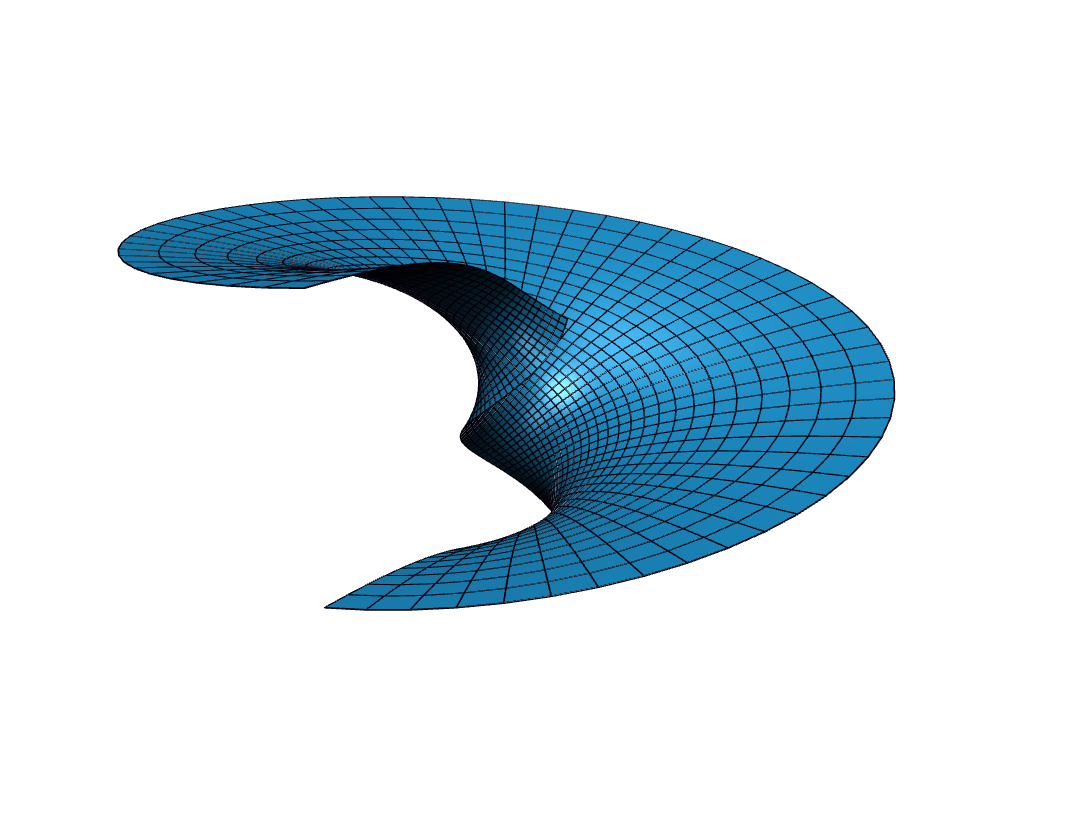}\quad
\includegraphics[height=4.5cm]{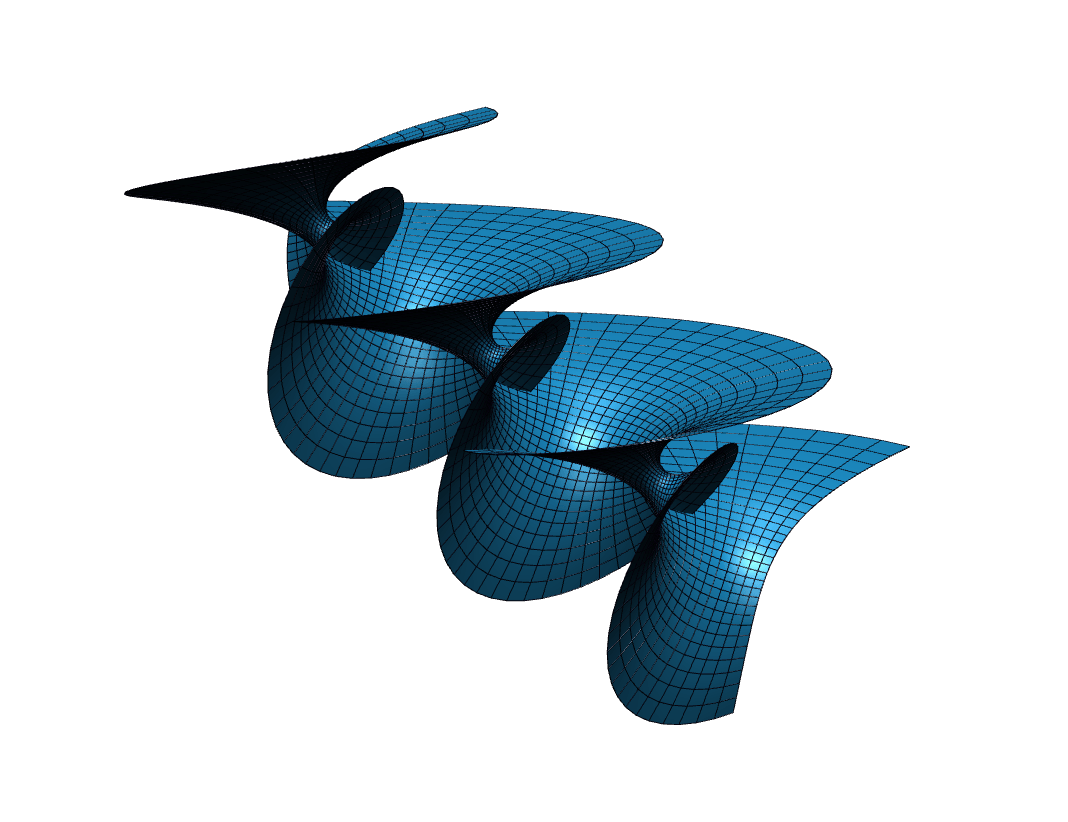} 
\caption{Elements  $f^{\frac 2{\sqrt 6}, \frac 1{\sqrt 6} +\frac
    i{\sqrt 6}}$ and $f^{\frac 1{\sqrt 2}, \frac 1{2\sqrt 2}(1 +
    i-j-k)}$ of the left associated family of
  the catenoid, orthogonally projected into $\R^3$.}
\end{figure}

The associated Willmore surface $f^\flat = fR -f^*$ of $f$ computes to
\begin{equation}
\label{eq: associated Willmore of catenoid}
f^\flat(x,y) = \frac 1{\cosh x}\big(\cosh x - x\sinh x - iy\cosh x + ji
xe^{-iy}\big)\,,
\end{equation}
and has right normal $R^\flat=-R$.
The associated Willmore surface
   $f^\flat$  is the twistor projection of the holomorphic curve 
\[
F^\flat(z)= \begin{pmatrix} ie^z(z-1)) \\ z+1 \\- ie^z \\1
\end{pmatrix}\,,
\]
where we used  that by  Theorem \ref{thm: twistor} and (\ref{eq:conformal Gauss map}) the line subbundle $E= F^\flat
  \C$ is given by
\[
E = \begin{pmatrix} f^\flat \\ 1
\end{pmatrix}
(1+R^\flat i)\C\,.
\]

\begin{figure}[H]
\includegraphics[height=4.2cm]{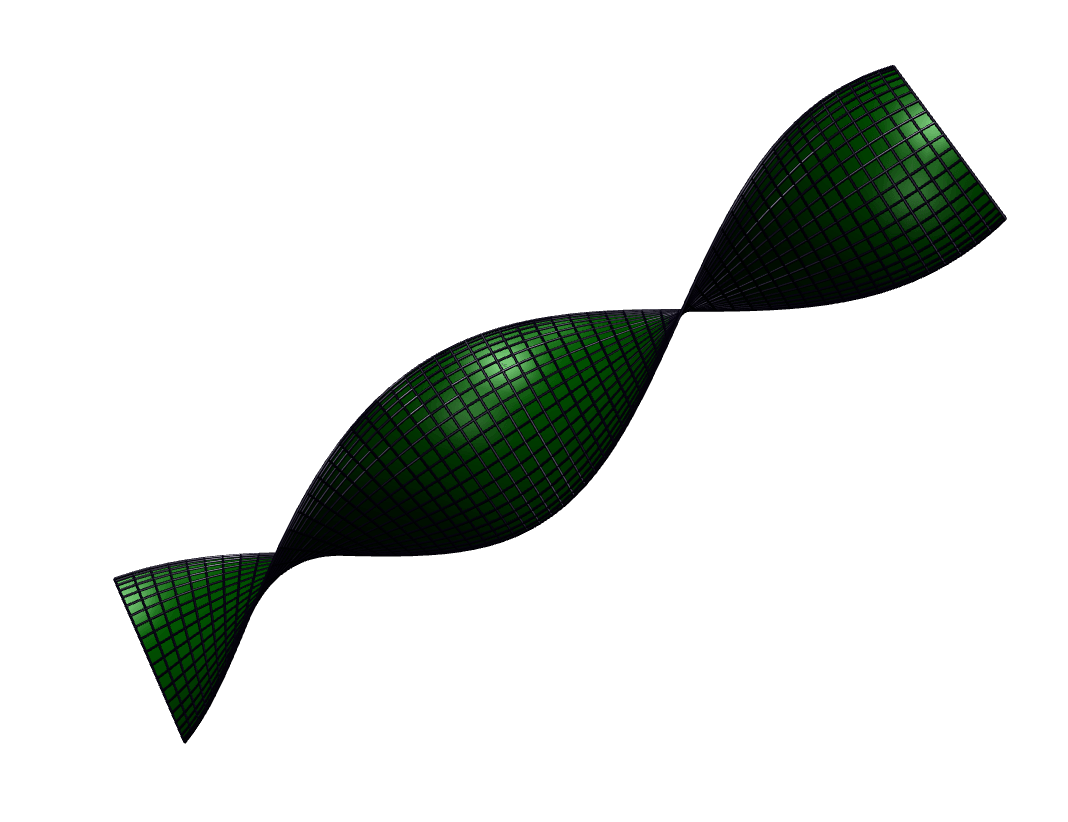}
\includegraphics[height=4.2cm]{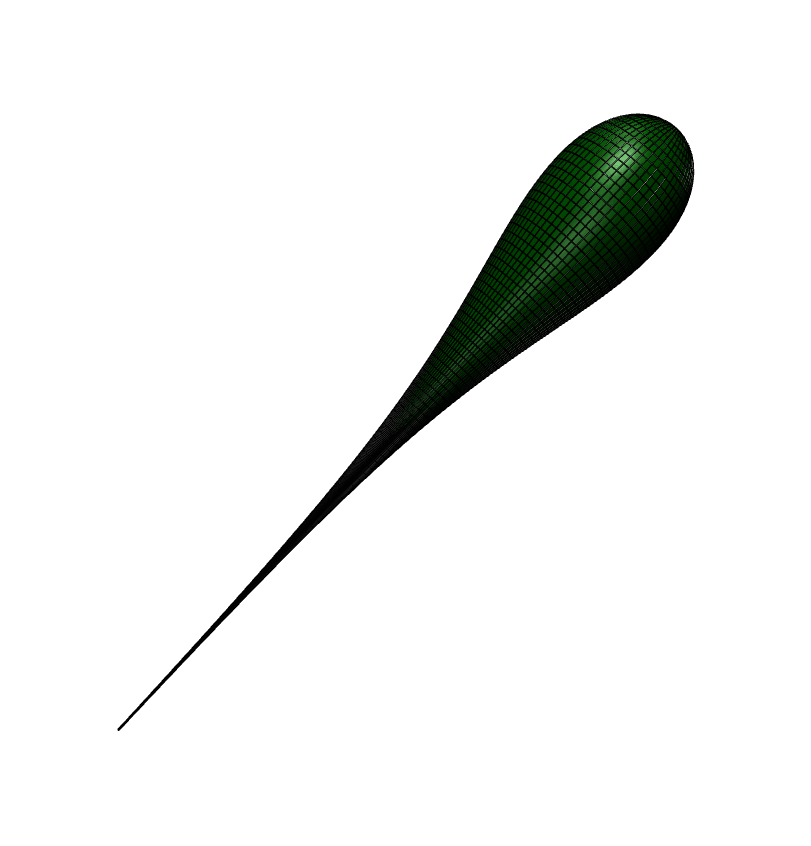}
\includegraphics[height=4.2cm]{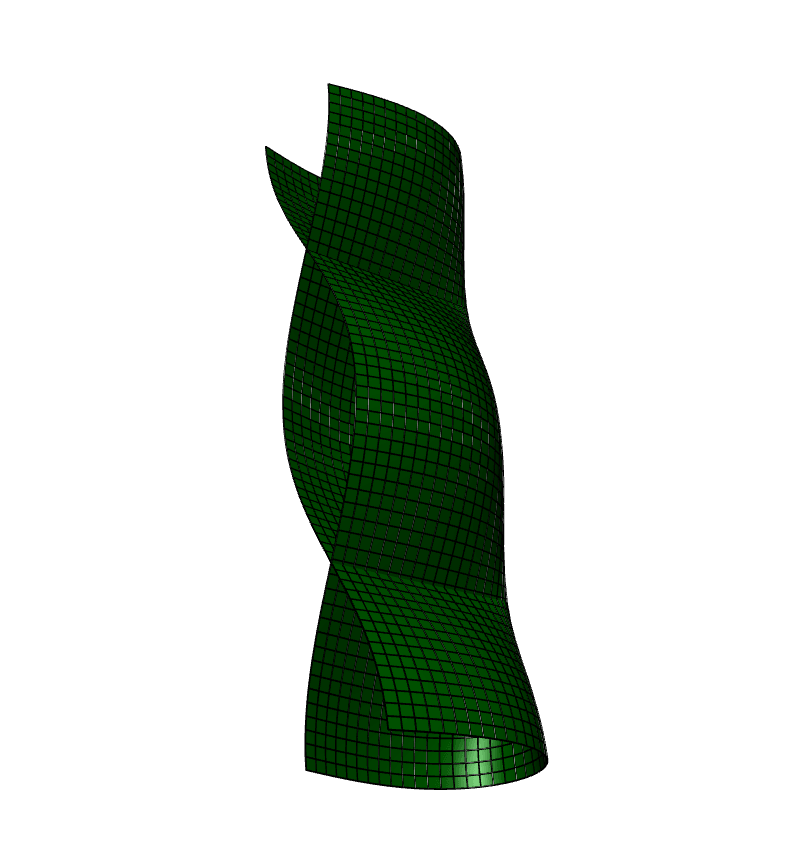}
\caption{The associated Willmore surface of the catenoid, various
  orthogonal projections  into $\R^3$.}
\end{figure}

The L\'opez-Ros deformation of the catenoid with parameter $\sigma =
e^{s+\ii t} \in\Cc_*$ is given, see Theorem \ref{thm:Goursat},  by
 \[
f_\sigma(x,y) = \begin{pmatrix}
\cos t (x\cosh s- \sinh x \sin y\sinh s) -\sin t (\cosh x\cos y \cosh s + y \sinh s) \\
\sin t (x\cosh s- \sinh x \sin y\sinh s) + \cos t (\cosh x\cos y \cosh s + y \sinh s) \\
\cosh x\sin y  
\end{pmatrix}\,.
\]

\begin{figure}[H]
\includegraphics[height=7cm]{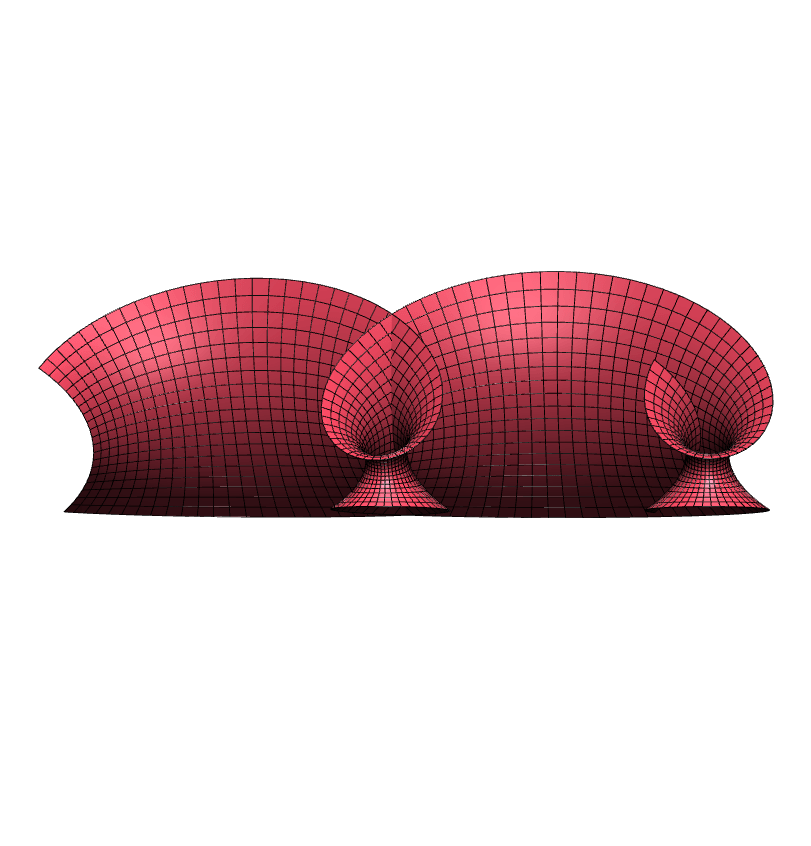}\quad
\includegraphics[height=7cm]{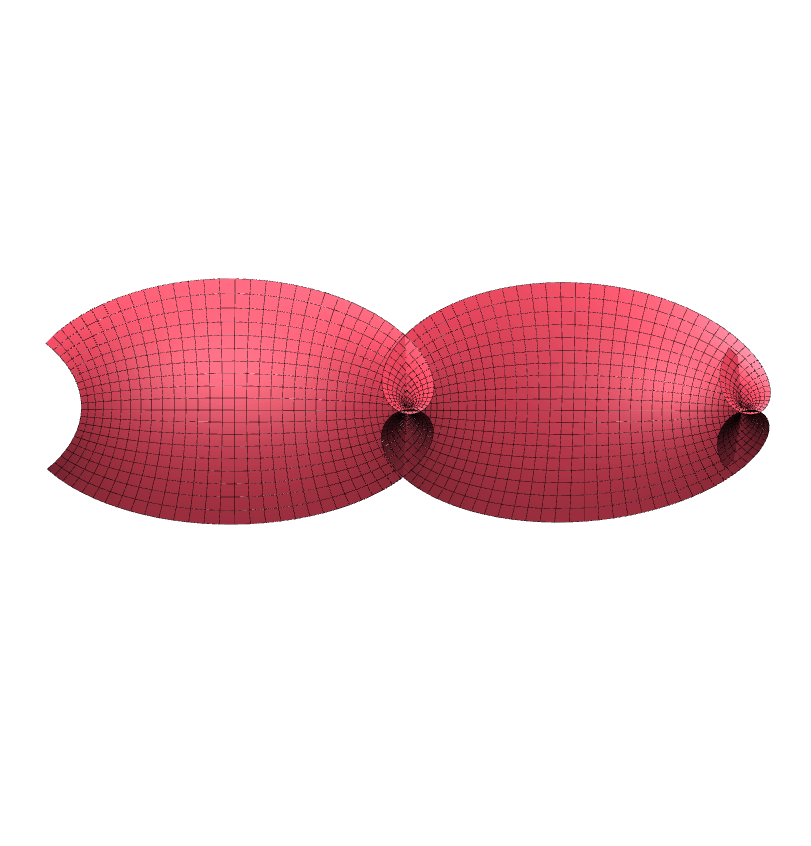}
\caption{The  L\'opez-Ros deformation of the catenoid with parameter
  $\sigma =2$ and $\sigma =4$.}
\label{fig:LRcatenoid}
\end{figure}

The periods of the simple factor dressing with parameter $\mu$  are given by
Corollary \ref{cor: periods in 3 space}:

\begin{lemma} The simple factor dressing $\hat f$ of the catenoid with
  parameters $(\mu,m, m)$ has translational periods 
\[
\hat f(x, y+2\pi) = \hat f(x,y) + \pi\left(i \hat b - \frac{\hat b}{\hat a-1} i (\hat a-1)\right)\,,
\]
where $\hat a = m \frac{\mu+\mu\invers}2m\invers, \hat b = m i
\frac{\mu\invers-\mu}2 m\invers$.

In particular, $\hat f(x, y+2\pi) = \hat f(x,y)$ if and only if
$m\in\C_*$ or $m\in \C_*j$ or $\mu\in S^1$. In this
case, $\hat f$ is a (reparametrised) catenoid.
\end{lemma}
\begin{proof} 
Since $\hat f = \mathcal{R}_{m,m}((\mathcal{R}_{m,m}\invers(f))^\mu)$ is the   simple
factor dressing with parameters $(\mu, m,  m)$ it is enough to
investigate the periods of the simple factor dressing  $\tilde f^\mu$ with parameter
$\mu$ of the  minimal
surface $\tilde f = \mathcal{R}_{m,m}\invers f$.  Since $\tilde f(x, y+2\pi) =
\tilde f(x,y)$
and $\tilde f^*(x, y+2\pi) = \tilde f^*(x,y) + 2\pi m\invers i m$ we see by  
  Corollary \ref{cor: periods in 3 space} that the simple factor
  dressing  $\tilde f^\mu$ has
  vanishing periods if and only if   $\mu\in S^1$ or $m\invers i m = \pm i$,
that is, $m\in\C_*$ or $m\in \C_*j$. In the former case, the simple
factor dressing of $f$ is trivial. In the latter case  we see with Lemma
\ref{lem:simple factor with parameter} that $\hat f$ is
the simple factor dressing of $f$ with parameter $\mu$ or $\bar
\mu\invers$. But the simple factor dressing  of $f$ with parameter
$\mu$ is by (\ref{eq:sfd in r3 quaternion}) given by  
 \[
f^\mu(x,y) = ix + j \cosh(x+s)e^{-i(y +t)}
\]
where $s=-\ln|\mu|, t = \arg \frac{\bar
  \mu -1}{\bar \mu(1-\mu)}$.  Thus, for $m\in\C_*$ we see that $\hat f =
f^\mu$ is a reparametrisation of the catenoid. Using   again Lemma 
\ref{lem:simple factor with parameter}  we obtain also the case $m\in
\C_*j$. 

In the  case of general parameters $(\mu, m, m)$, $m\in\H_*$, we obtain with
 (\ref{eq:simple conformal}) that  
\[
\hat f(x, y+2\pi) =  \hat f(x, y) + \pi\left(i \hat b - \frac{\hat b}{\hat a-1} i (\hat a-1)\right)\,.
\]
\end{proof}

\begin{figure}[H]

\includegraphics[height=5cm]{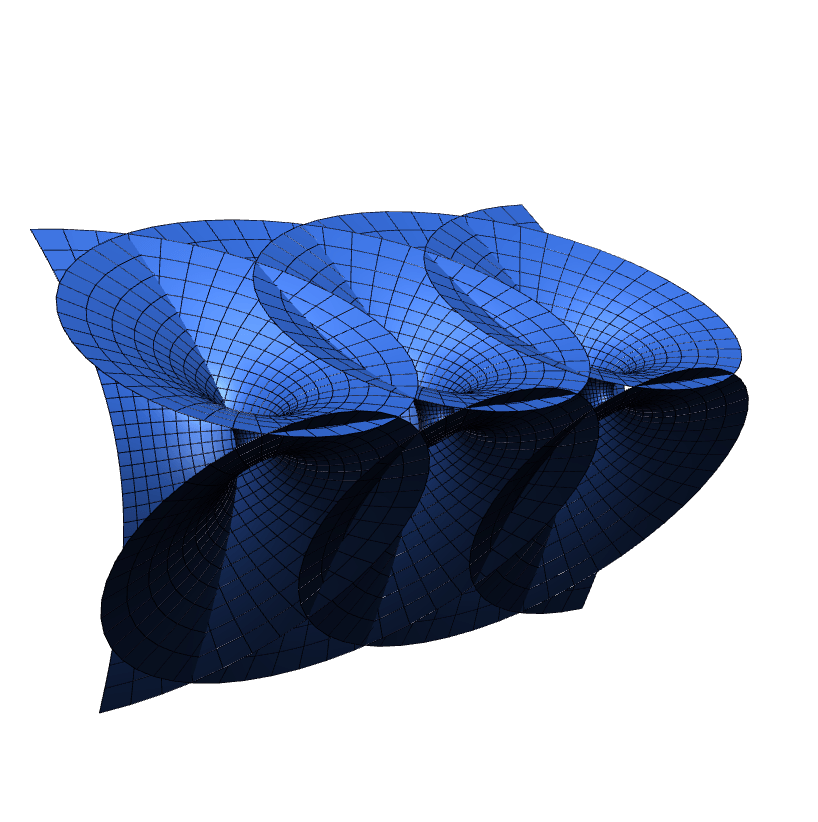}
\includegraphics[height=5cm]{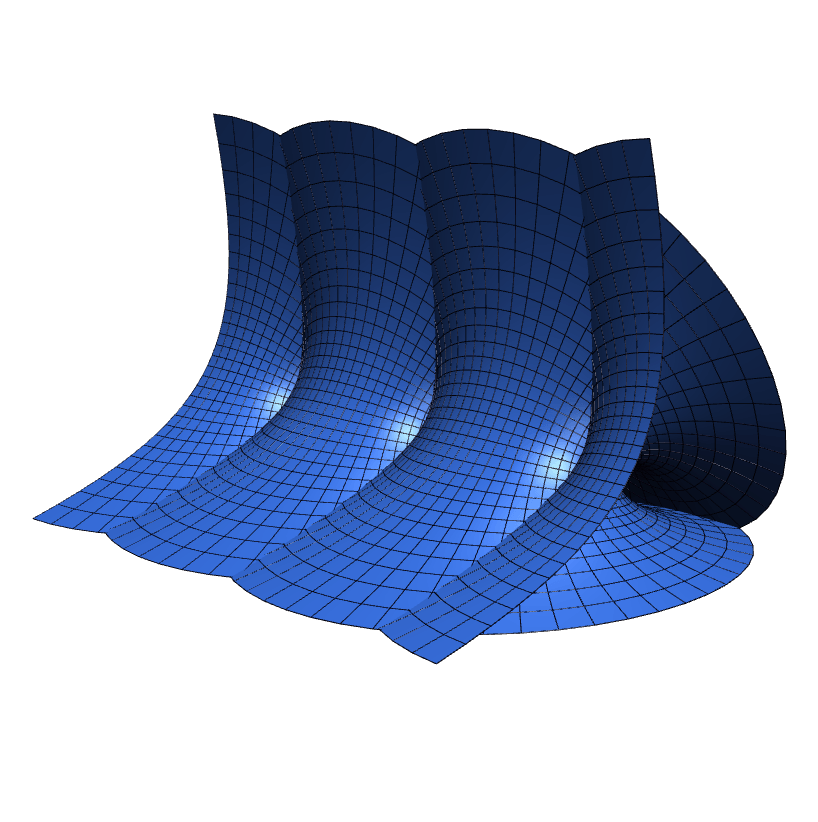}
\caption{Simple factor dressing of the catenoid with parameters
  $(-\frac i2, \frac 12(1+i-j-k))$.}
\label{fig: SFD catenoid}
\end{figure}

Thus, in general the simple factor dressing of a catenoid will have
translational periods. Although the resulting surfaces resemble
Catalan's surface (see Figure \ref{fig:LRcatenoid} and Figure \ref{fig: SFD catenoid})
 the simple factor dressing of a catenoid has by Corollary
 \ref{cor:complete}    no branch points.

If we allow the simple factor dressing to be a minimal surface in
$\R^4$ we obtain with Corollary \ref{cor:periods} further closed minimal surfaces: for example,
when choosing $m= \frac{1+k}2, n =\frac{ i-j}2$ then the simple factor dressing with
parameters $(\mu, m, n)$ gives a minimal immersion into $\R^4$ by
\[
\hat f(x,y) = ix  +  k \sin y \cosh x + (\sin y \sinh x \sinh s + j
\cos y \cosh x \cosh s)e^{-jt}\,,
\] 
where $s=-\ln|\mu|, t = \arg\frac{\bar \mu-1}{\bar\mu(1-\mu)}$. In
particular, we immediately see  that $\hat f(x, y +2\pi) = \hat f(x,y)$.

\begin{figure}[H]
\includegraphics[height=4.5cm]{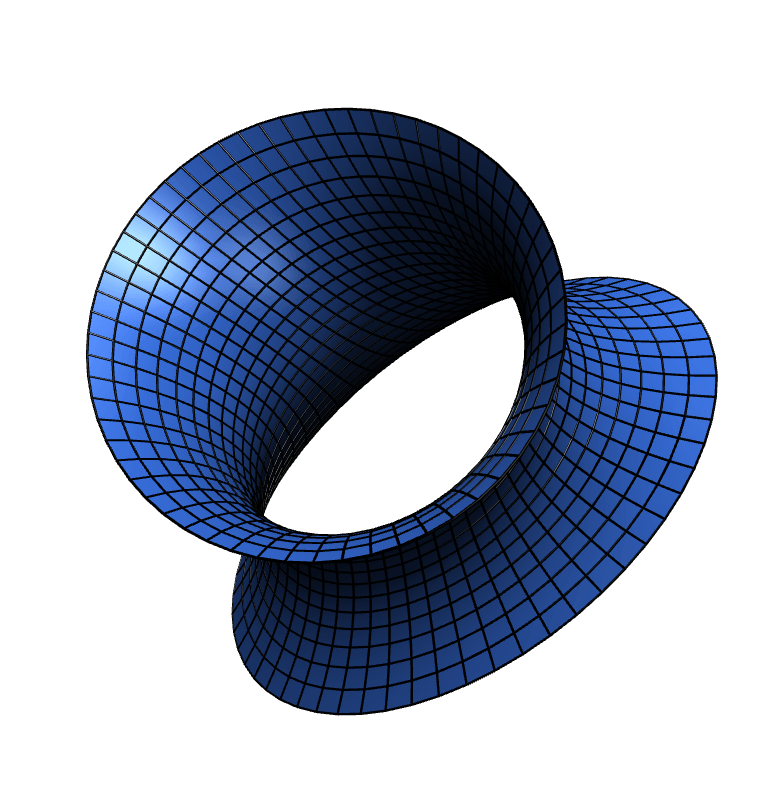}\quad
\includegraphics[height=4.5cm]{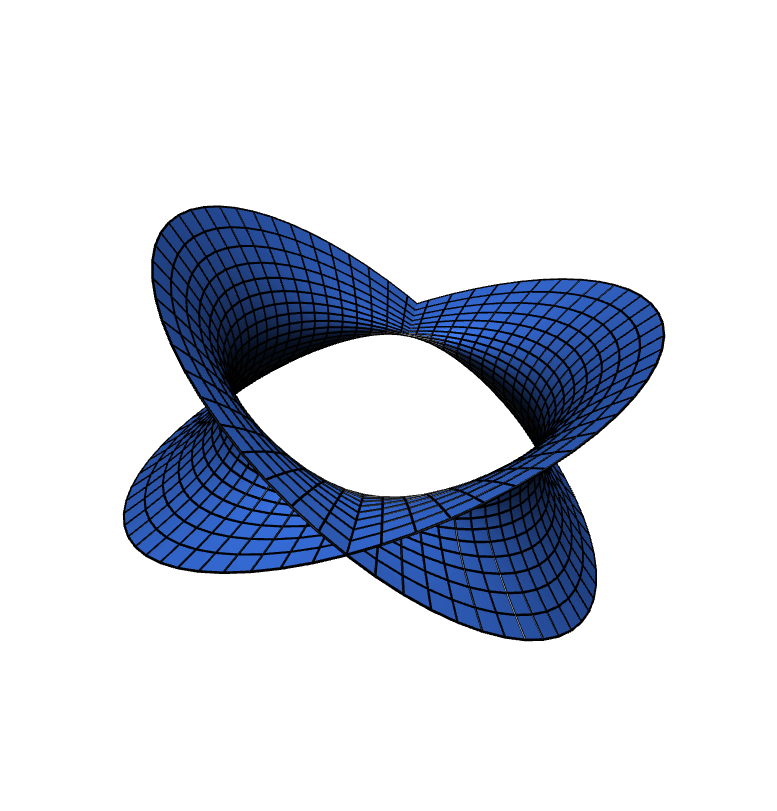}
\includegraphics[height=4.5cm]{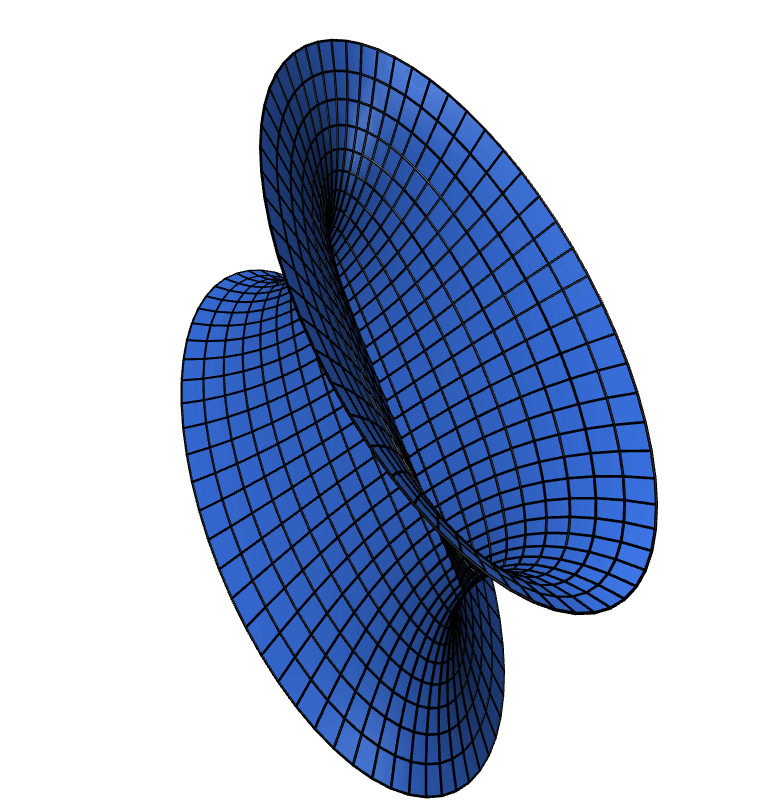}
\caption{Simple factor dressing of the catenoid with parameters
  $(-\frac i2, \frac{1+k}2, \frac{i-j}2)$, various orthogonal projections to $\R^3$.}
\end{figure}

Finally, we compute the $\mu$--Darboux transforms of the catenoid
which are given (\ref{eq: hat f}) by
\[
f^\sharp = (fR-f^*)(R+\rho)\invers
\]
where $\rho = m \frac{i(1+\mu)}{1-\mu}m\invers$ and $m\in\H_*$.  In
our case recall (\ref{eq: associated Willmore of catenoid}) that the associated Willmore surface $f^\flat=fR -f^*$
of the catenoid $f$ is
\[
f^\flat(x,y) =\frac 1{\cosh x}\big(\cosh x - x\sinh x - iy\cosh x + ji
xe^{-iy}\big)\,.
\]
For $\mu\in S^1$ we have $\rho =m\frac{i(1+\mu)}{1-\mu}m\invers \in\R$ and
$(R+\rho)\invers = \frac 1{1+\rho^2}(\rho-R)$. Therefore, we can
simplify in this case
\[
f^\sharp(x,y) = 
\frac 1{1+\rho^2}\left(\rho -(x\rho + y)\tanh x +  i(x- y\rho - \tanh
  x) + \frac{je^{-iy}}{\cosh x}(1+i(y + \rho x))\right)\,.
\] 

\begin{figure}[H]
\includegraphics[height=4.5cm]{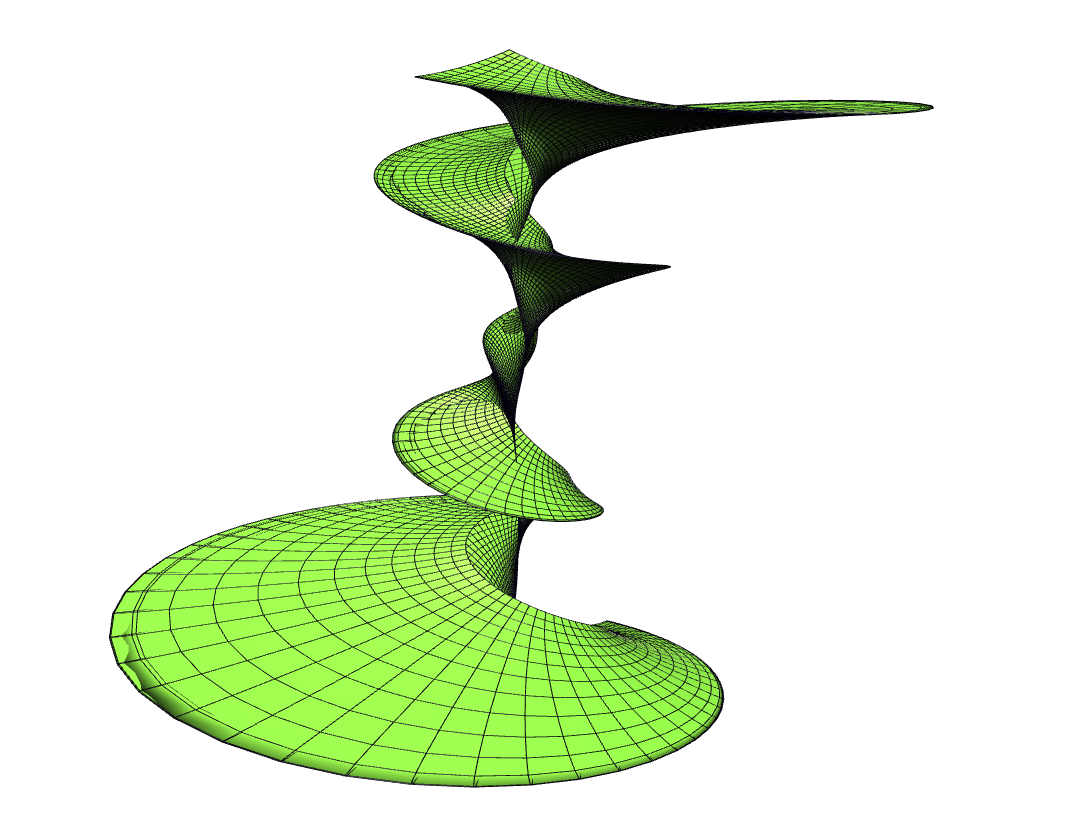}
\includegraphics[height=4.5cm]{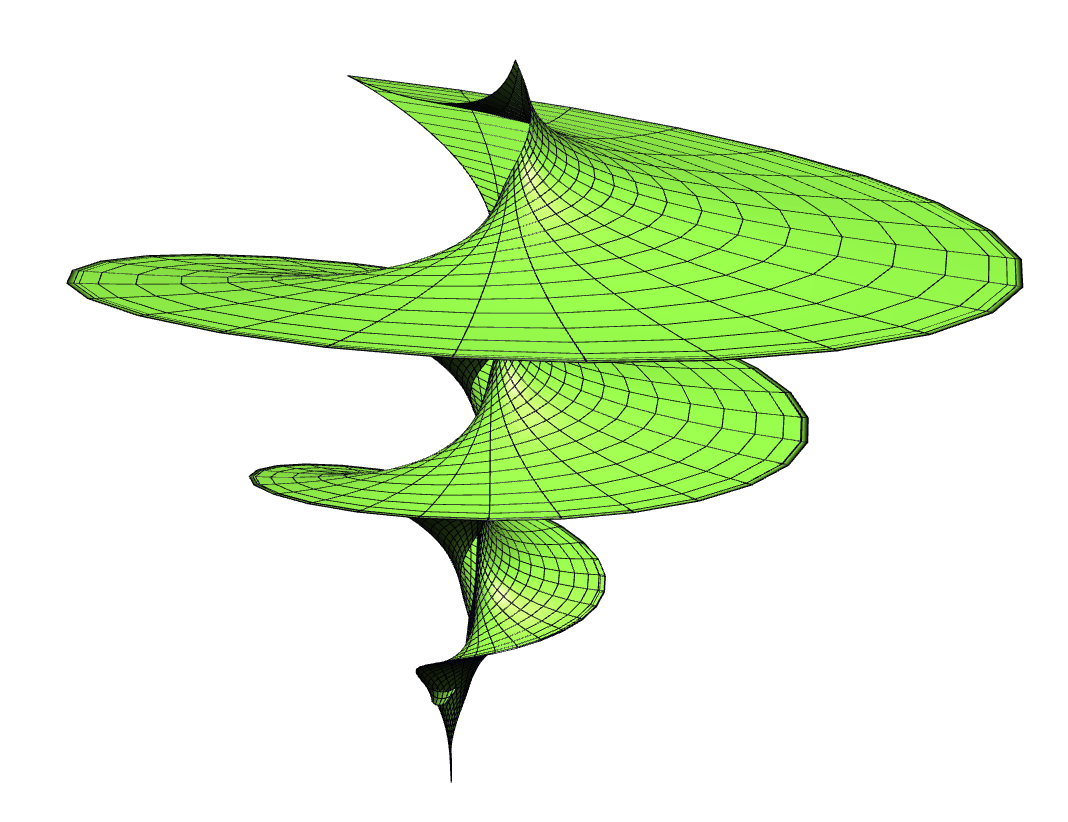}
\caption{$\mu$--Darboux transforms of the catenoid with
  $\mu=-\frac i2$, $m= \frac 12(1+i-j-k)$ and $\mu=i, m=1$,
  orthogonally projected into $\R^3$.}
\end{figure}


\subsection{Explicit examples with one planar end}

We will now consider examples, \cite{karcher},    of minimal surfaces with one planar
end given by the Weierstrass data $g(z) = z^{l+1}, dh =
z^{l-1} dz$, $l\in\N$, that is,
$
f(x,y) = \Re\Phi(z), \, z=x+iy\in\C_*\,,
$ 
where the holomorphic null curve $\Phi$ is given, see
(\ref{eq:Weierstrass representation}), as
\[
\Phi(z) =\left(\frac12(-\frac 1z - \frac{z^{2l+1}}{2l+1}), \frac
\ii 2(-\frac 1z +  \frac{z^{2l+1}}{2l+1}), \frac{z^l}l\right): \C_*
\to \Cc^3\,.
\]

Here we identify as before $z = x + \ii y$. Indeed, by Theorem
\ref{thm:FTCend}  the immersion  $f$ has a planar
end at the puncture $z=0$ since $\ord_{z=0} d\Phi = -2$ and the
residue of $d\Phi$ at $z=0$ vanishes.

\begin{figure}[H]
\includegraphics[height=4cm]{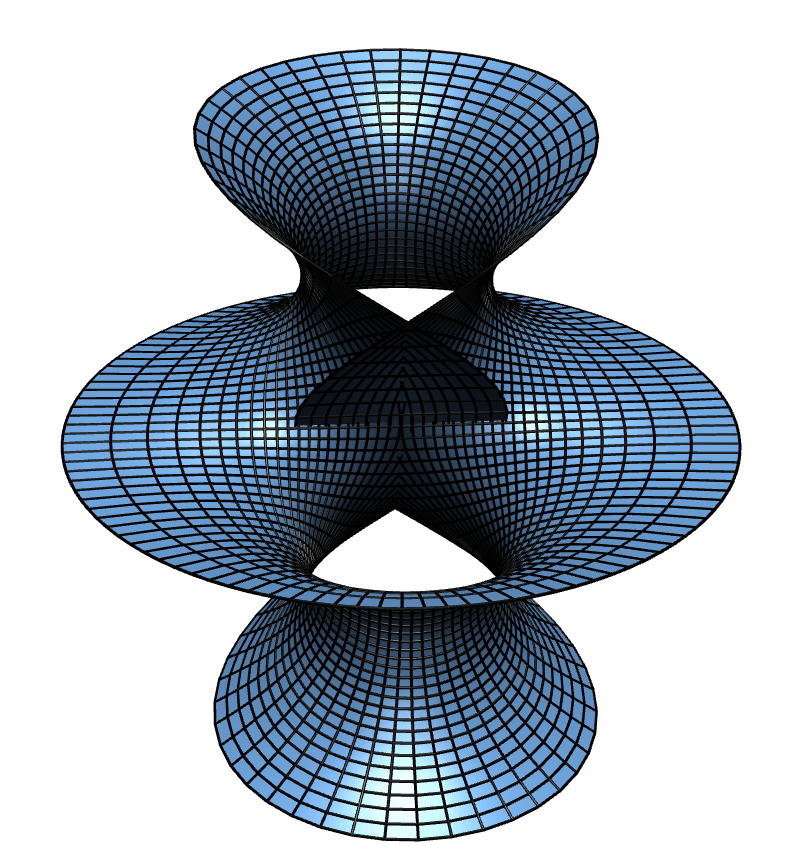} \qquad \qquad
\includegraphics[height=3.5cm]{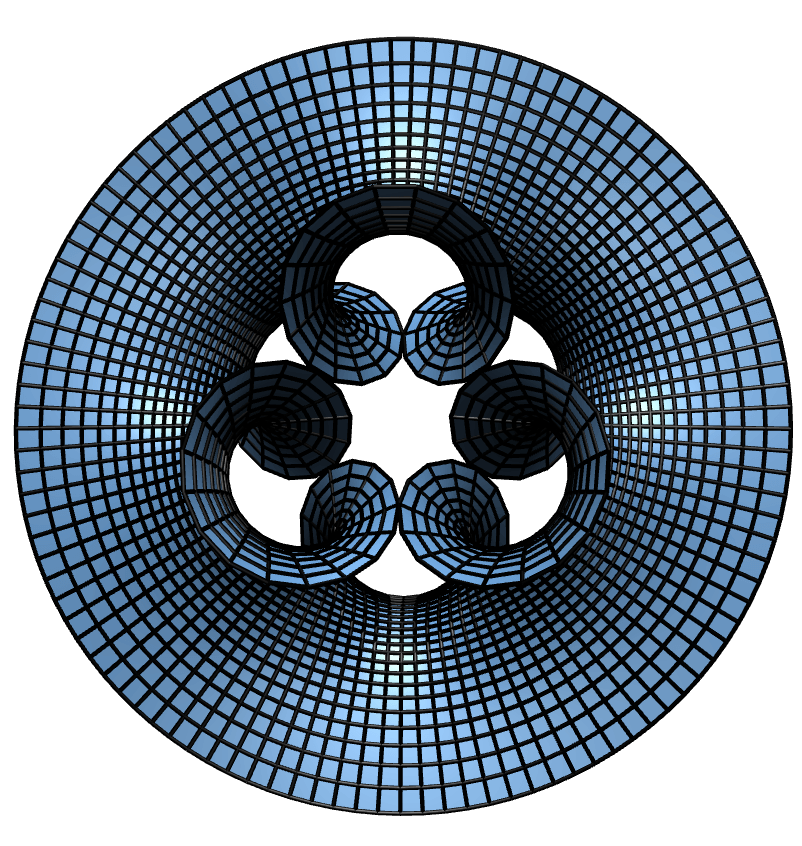} \\ \quad \\ 
\caption{Minimal surfaces with one planar end, $l=1$ and $l=3$.}
\end{figure}
Then  the  conjugate surface
\[
f^*(x,y) = \Im\Phi(z)
\]
is single--valued on $\C_*$, and so are the left and right
associated family,  

\begin{figure}[H]
\includegraphics[height=3.5cm]{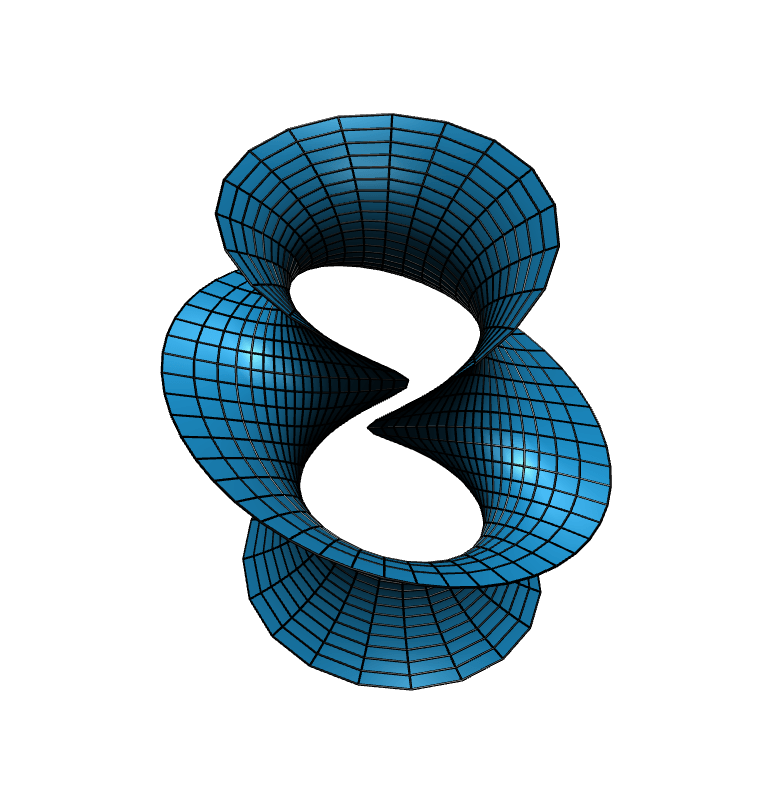}
\includegraphics[height=3.5cm]{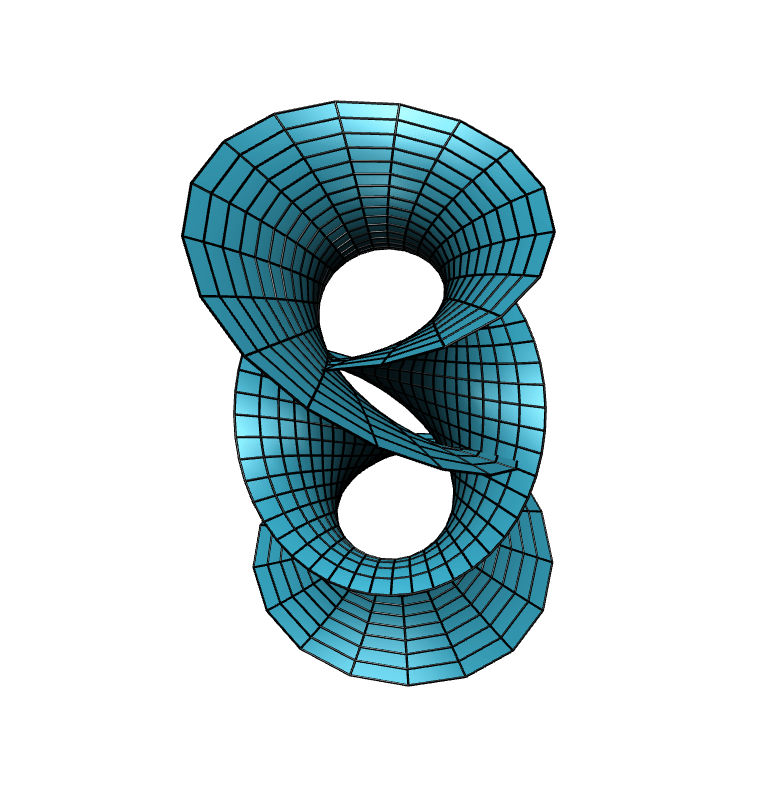} \\
\includegraphics[height=3.5cm]{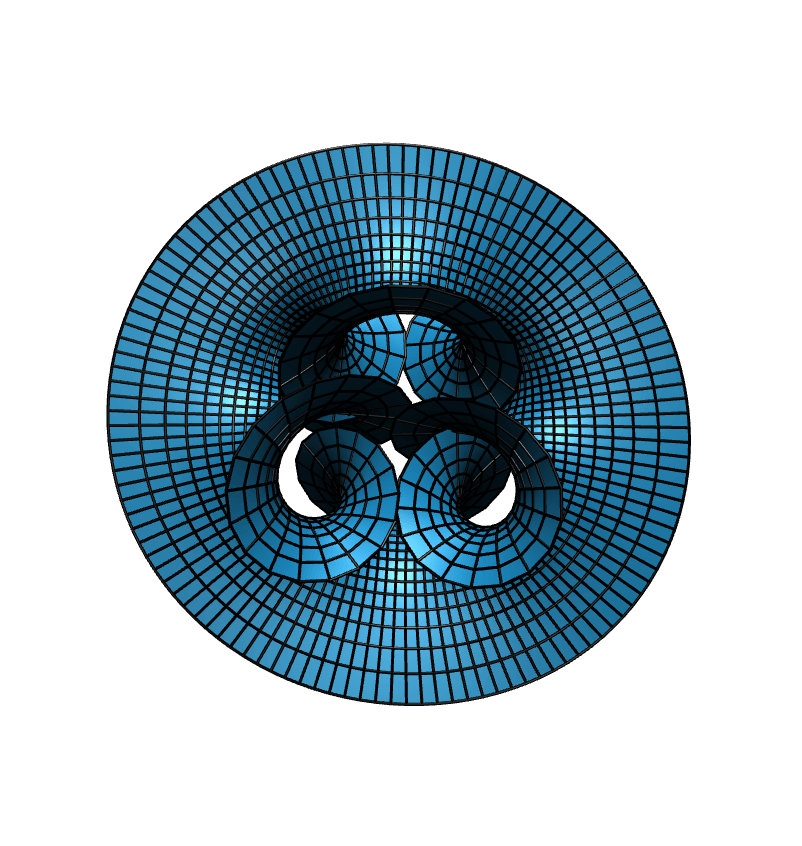}
\includegraphics[height=3.5cm]{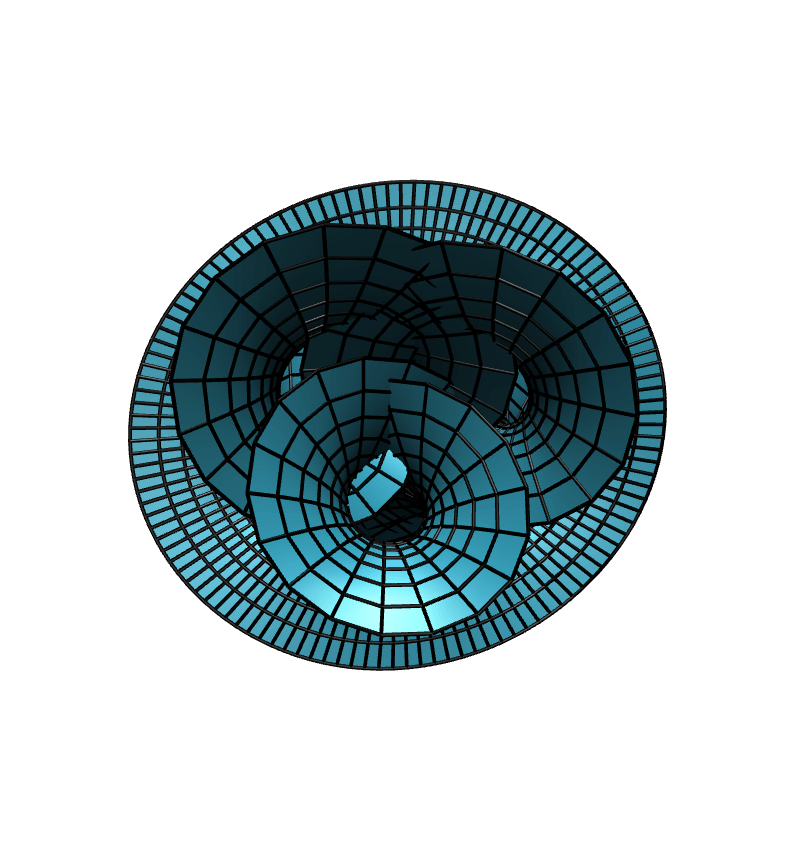}
\caption{Elements   $f^{\frac 1{\sqrt 7}, \frac 2{\sqrt {7}} +\frac
    j{\sqrt 7} - \frac k{\sqrt 7}}$ and $f_{\frac 1{\sqrt 7}, \frac 2{\sqrt {7}} +\frac
    j{\sqrt 7} - \frac k{\sqrt 7}}$  of the left and right associated family of
  a minimal surface with one planar end, $l=1$ and $l=3$, orthogonally projected into $\R^3$.}
\end{figure}

the associated Willmore surface,
\begin{figure}[H]
\includegraphics[height=5.5cm]{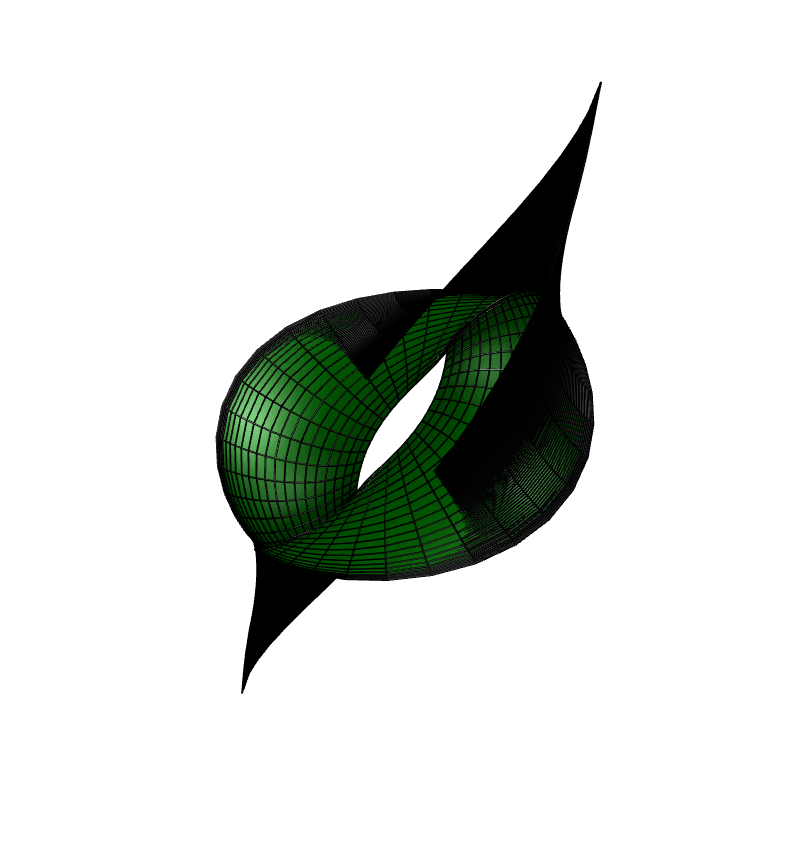}
\includegraphics[height=5.5cm]{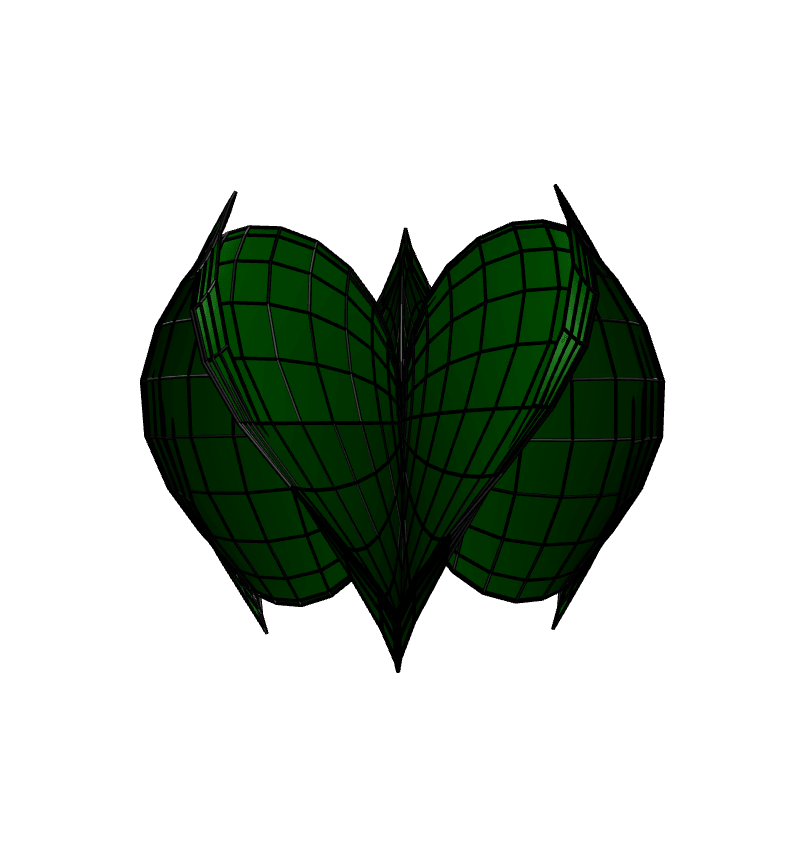}
\caption{The associated Willmore surface of a minimal surface with one
  planar end,  $l=1$ and $l=3$, orthogonally projected into $\R^3$}
\end{figure}
and the $\mu$--Darboux transforms of $f$. 
\begin{figure}[H]
\includegraphics[height=5.5cm]{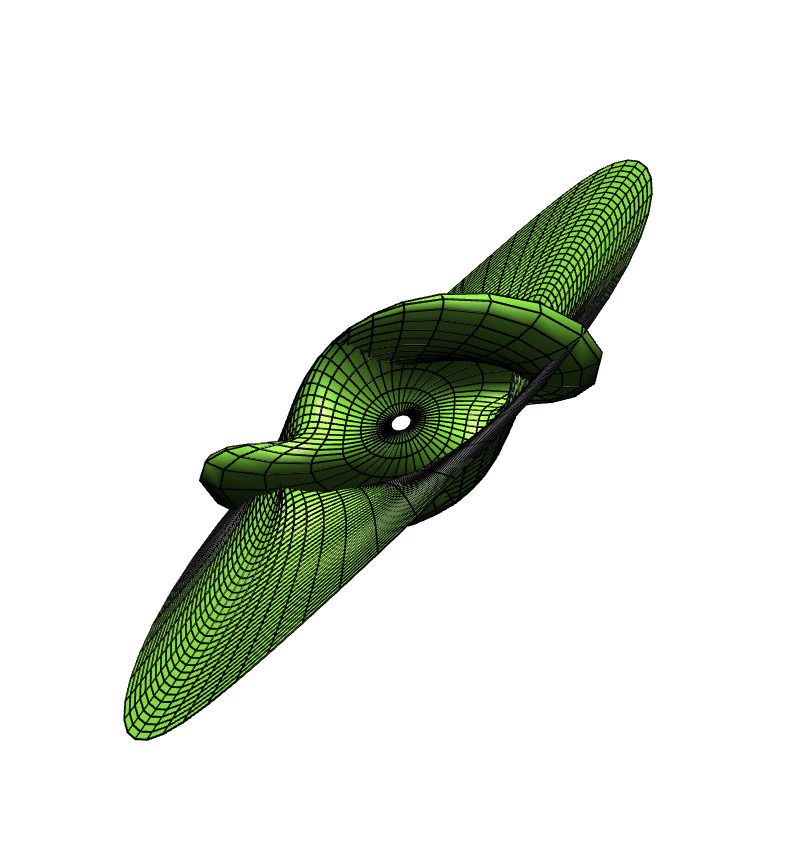}
\includegraphics[height=5.5cm]{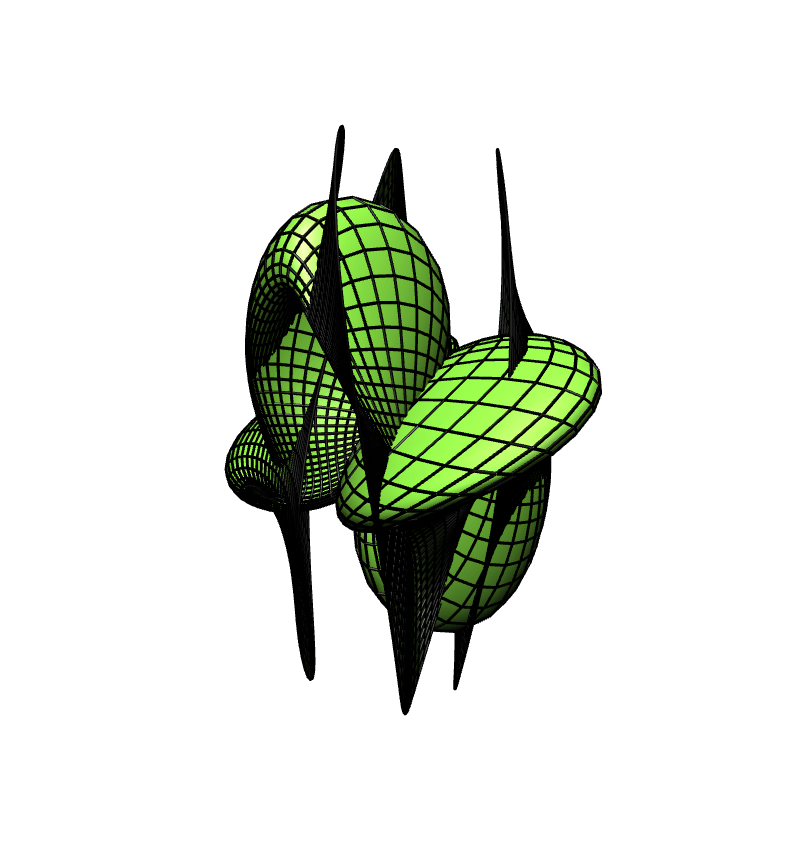}
\caption{A $\mu$--Darboux
  transform with $\mu=-\frac i2, m=1$, of a minimal surface with one
  planar end,  $l=1$ and $l=3$, orthogonally projected into $\R^3$.}
\end{figure}
We  discuss the  simple
factor dressings of $f$ with parameter $\mu$ in more detail for the
case $l=1$, that is, 
\[
f(x,y) = -\frac{ix}2\left(\frac 1{x^2+y^2}+ \frac{x^2-3y^2}3\right) -
\frac{jy}2\left(\frac 1{x^2+y^2}+ \frac{3x^2-y^2}3\right) + kx, \quad
(x,y)\not=(0,0), 
\]
with Gauss map, using (\ref{eq:Gauss with g}),
\[
N(x,y) =\frac1{1 +(x^2+y^2)^2}(2i(x^2-y^2)+ 4jxy +
  k((x^2+y^2)^2-1),  \quad
(x,y)\not=(0,0), 
\]
and conjugate minimal surface
\[
f^*(x,y) = \frac{iy}2\left(\frac 1{x^2+y^2}+ \frac{y^2-3x^2}3\right) -
\frac{jx}2\left(\frac 1{x^2+y^2}+ \frac{3y^2-x^2}3\right) + ky,  \quad
(x,y)\not=(0,0)\,.
\]
The simple factor dressing $f^\mu$ with parameter $\mu$ is given (\ref{eq:sfd in r3 quaternion}) by
\begin{eqnarray*}
f^\mu(x,y) &=&  -\frac{ix}2\left(\frac 1{x^2+y^2}+ \frac{x^2-3y^2}3\right)
\\ &&
+ \Big\{- jy(\frac{1}2\left(\frac 1{x^2+y^2}+
  \frac{3x^2-y^2}3\right)\cosh s  + \sinh s)  \\\
&&\quad  + kx(\cosh s 
-
\frac{1}2\left(\frac 1{x^2+y^2}+ \frac{3y^2-x^2}3\right)\sinh s)\Big\}e^{-it}\,, \quad
(x,y)\not=(0,0)\,,
\end{eqnarray*}
where $s=-\ln|\mu|, t= \arg\frac{\bar \mu-1}{\bar\mu(1-\mu)}$.
From Theorem \ref{thm:SFDends}  we know that $f^\mu$ is a minimal surface with a
planar end at the puncture $(x,y) =(0,0)$. 

We recall that $\mu = \frac{1-e^{-(s+it)}}{1-e^{s-it}}$ so that $\rho
  =i \frac{1+\mu}{1-\mu} = i\frac{1+e^{2s}-2e^{s+it}}{e^{2s} -1}$ .
  Since $N^\mu = (N+\rho)N(N+\rho)\invers$ and $\lim_{(x,y)\to
    (0,0)}N(x,y) =-k$ we therefore see that
\[
\lim_{(x,y)\to (0,0)} N^\mu(x,y) = (k - \rho)k(\rho -k)\invers=   \frac
1{\cosh(s)}(i \sinh (s) - k e^{-it}).
\]
In particular,  the end near the puncture  $(x,y)=(0,0)$ is asymptotic to the plane spanned by $(i+k\sinh s
e^{-it})$ and $je^{-it}$. 
\begin{figure}[H]
\includegraphics[height=6.5cm]{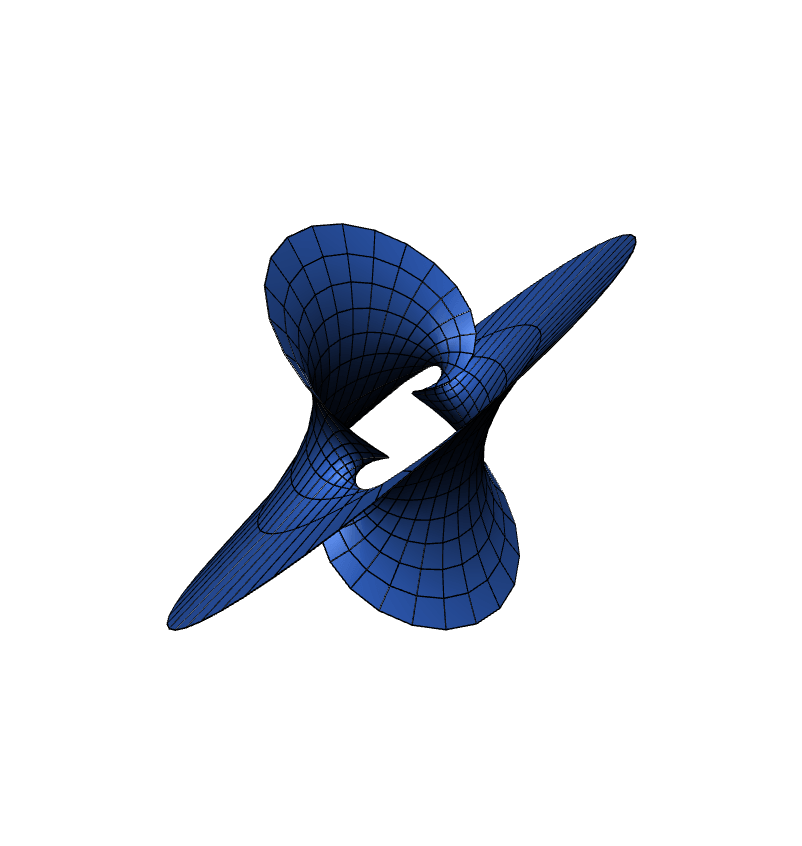}
\includegraphics[height=6.5cm]{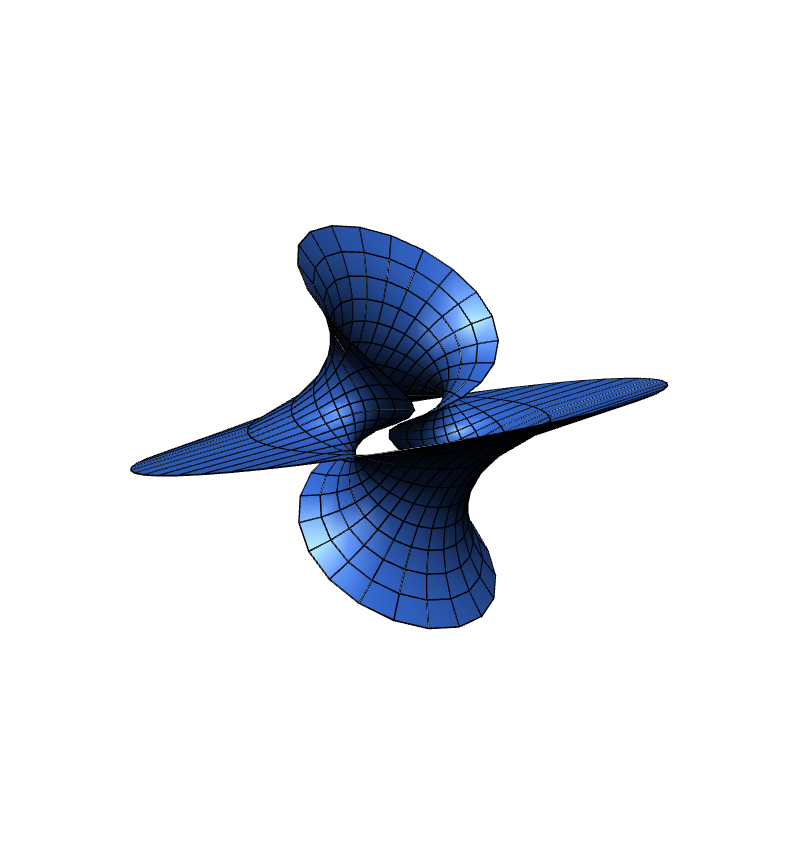}
\caption{Simple factor dressing of a minimal surface with one planar
  end, $l=1$, with parameters
  $\mu=-\frac i2$ and $\mu=-\frac 12+\frac i2$.}
\end{figure}

By Theorem \ref{thm:Goursat} the L\'opez--Ros deformation of $f$ with
parameter $\sigma =e^{s+it}\in\Cc_*$ is
\[
f_\sigma = \begin{pmatrix}
x\cos t  \,\left(- \frac 1{x^2+y^2} e^{-s}  +
\frac{3y^2-x^2}3e^s\right) - y  \sin t \,\left(- \frac 1{x^2+y^2} e^{-s} +
\frac{y^2-3x^2}3e^s \right) \\[.3cm]
x\sin t \, \left(- \frac 1{x^2+y^2} e^{-s} +
\frac{3y^2-x^2}3e^s\right)  + y \cos t \, \left(- \frac 1{x^2+y^2} e^{-s} +
\frac{y^2-3x^2}3e^s\right) \\[.3cm]
x
\end{pmatrix}\,.
\]
$f_\sigma$ has a vertical planar end at the puncture $z= x+iy
=0$. From the holomorphic null curves  of $f$ and $f_\sigma$,
$\sigma\not=1$,  we see
that $f_\sigma$ is not a reparametrisation of $f$.
\begin{figure}[H]
\includegraphics[height=6cm]{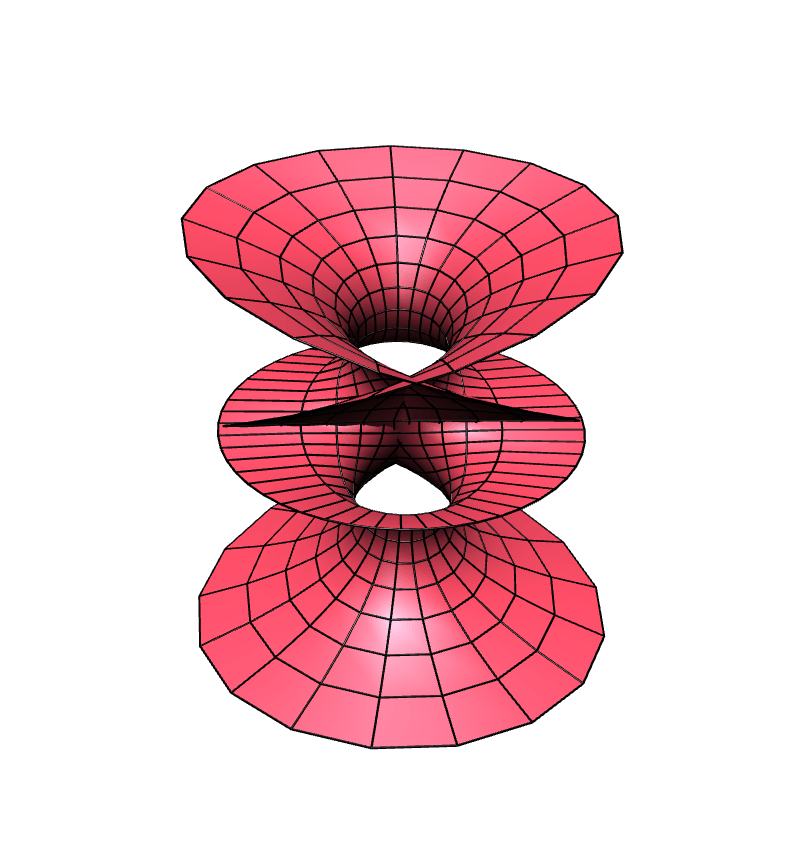}
\includegraphics[height=6cm]{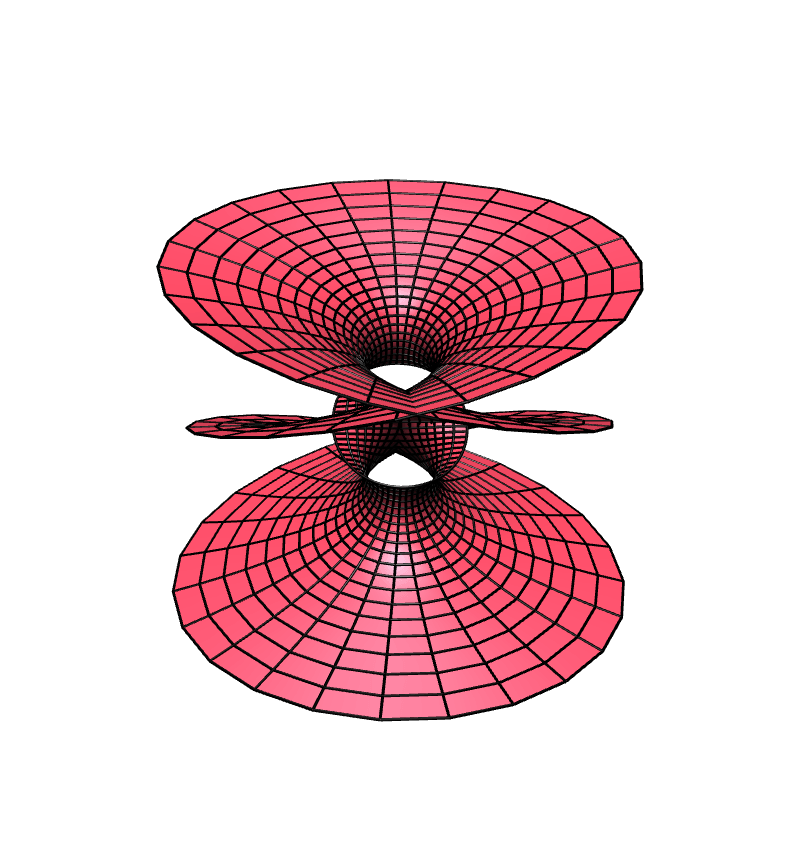}
\caption{L\'opez--Ros deformation of a minimal surface with one planar
  end, $l=1$, with parameters $\sigma =2$ and $\sigma=7$.}
\end{figure}

Finally, we include some pictures of the simple factor dressing for
more general parameters. Note that the surfaces are
single--valued for all parameters $(\mu, m, m)$, and have a planar end
at $z=0$.

\begin{figure}[H]
\includegraphics[height=6cm]{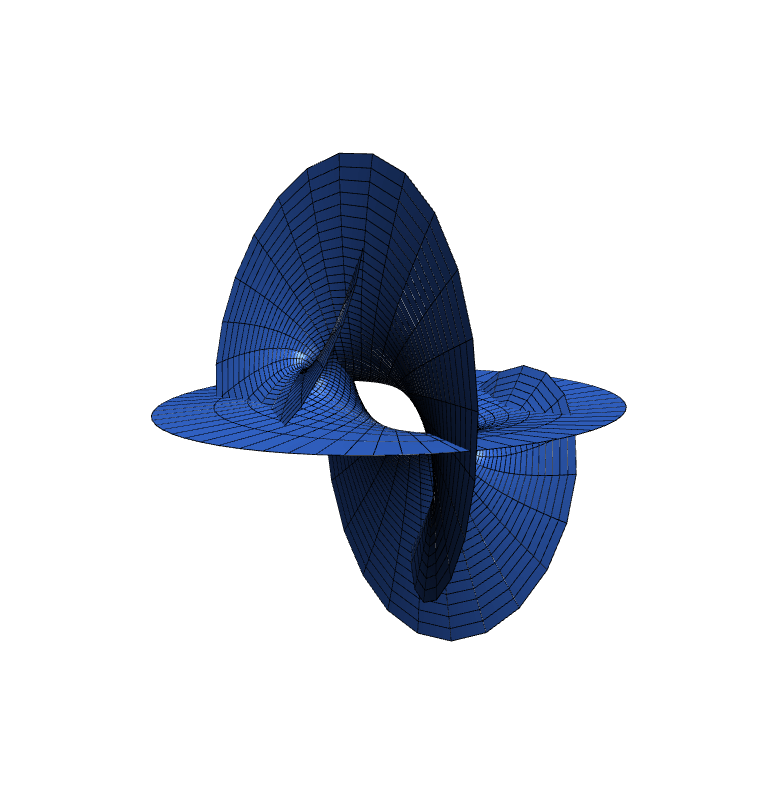}
\includegraphics[height=6cm]{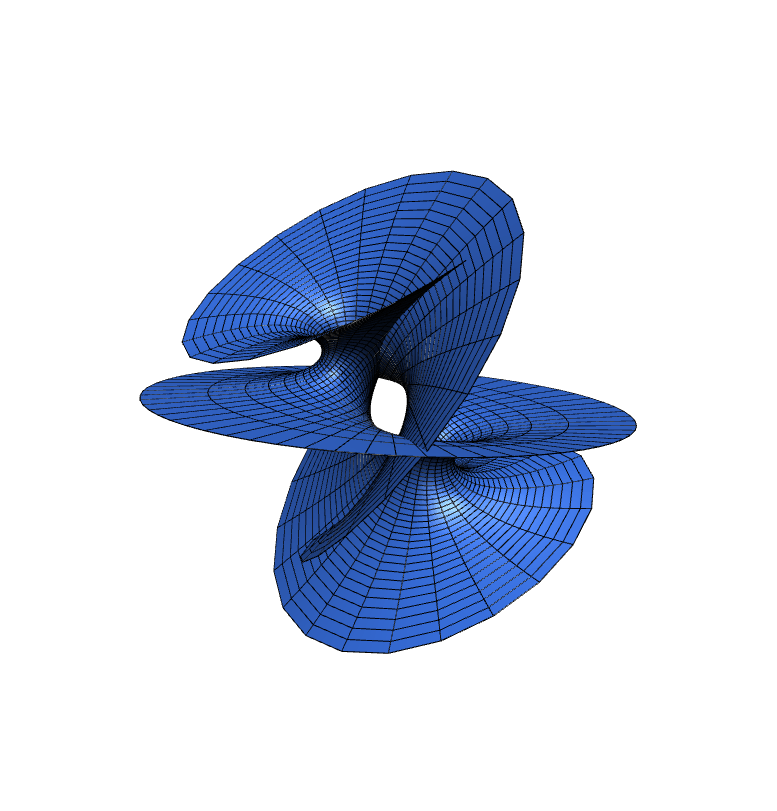}
\caption{Simple factor dressing of a minimal surface with one planar
  end, $l=1$, with parameters $(\mu, m, m)$ with
  $\mu= -\frac i2$ and $\mu=  -\frac 12 + \frac i2$, $m= \frac 12-k$.}
\end{figure}


\subsection{Scherk surfaces}

We will now consider the first Scherk surface given by
the Weierstrass data $g(z) = z, \omega(z) = -\frac{4}{z^4-1}dz$, that
is, $f =\Re\Phi$ is the real part
of the  (multi--valued)  holomorphic null curve
\[
\Phi(z) =(\ii\log\frac{z+\ii}{z-\ii}, \ii\log \frac{z+1}{z-1},
\log\frac{z^2+1}{z^2-1})\,, z\in\C\setminus\{\pm 1, \pm i\}\,,
\] 
identifying again $z= x+ \ii y$.

\begin{figure}[H]
\includegraphics[height=3.8cm]{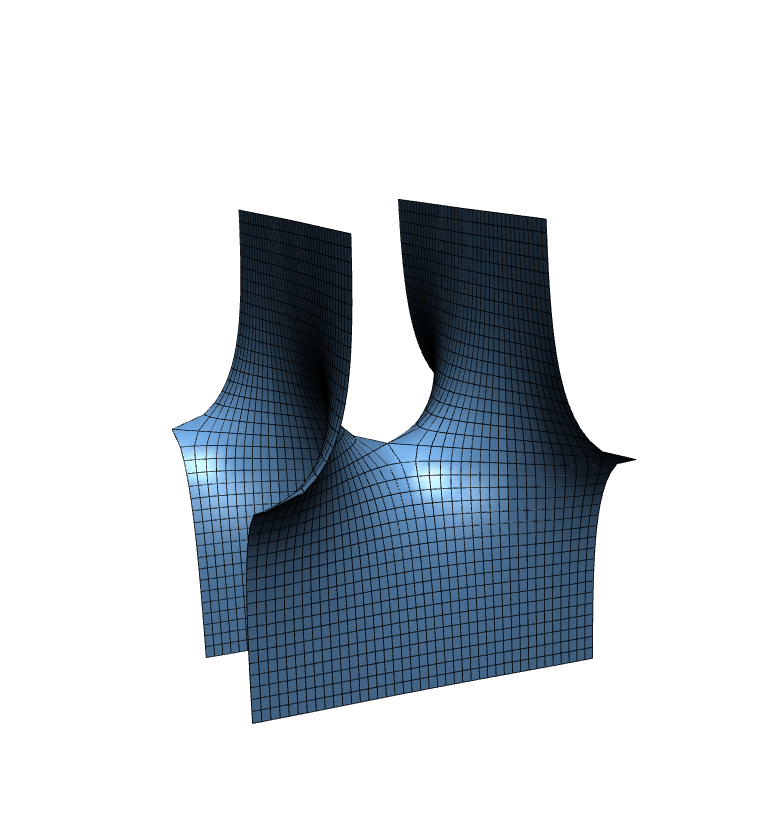}
\includegraphics[height=3.8cm]{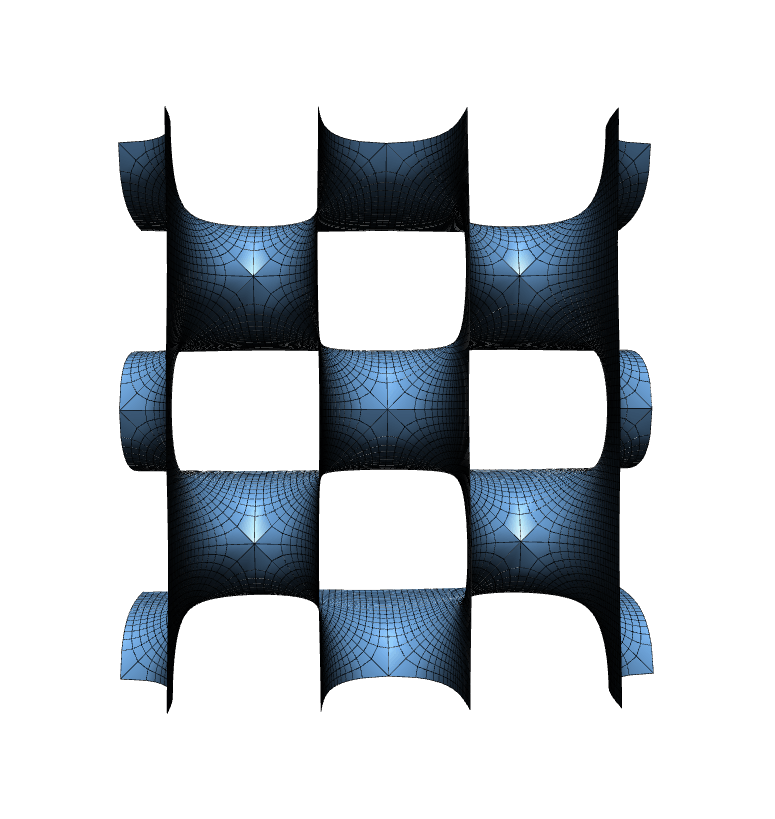}
\includegraphics[height=3.8cm]{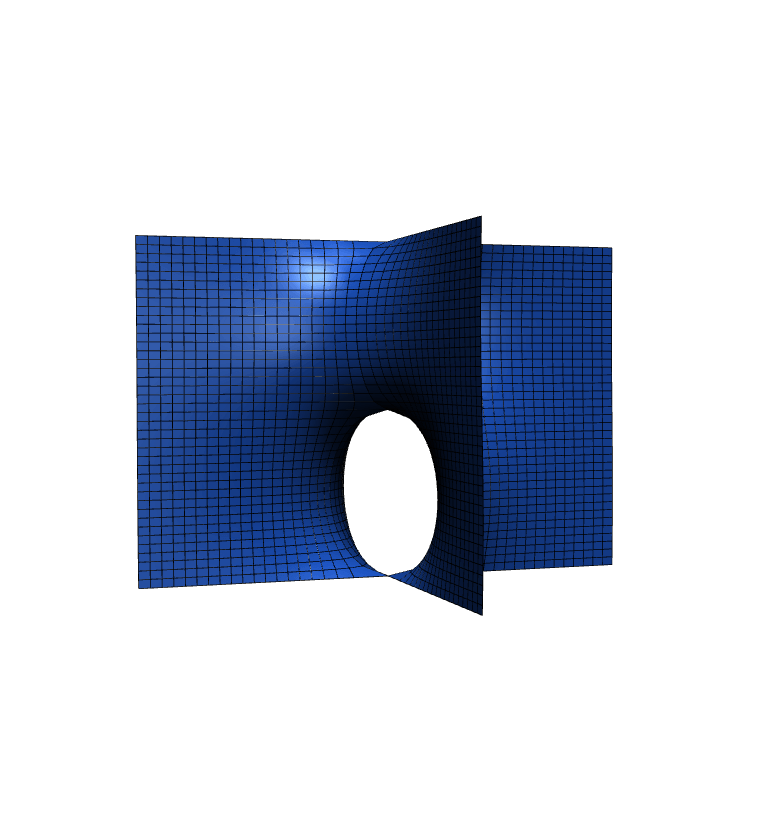}
\includegraphics[height=3.8cm]{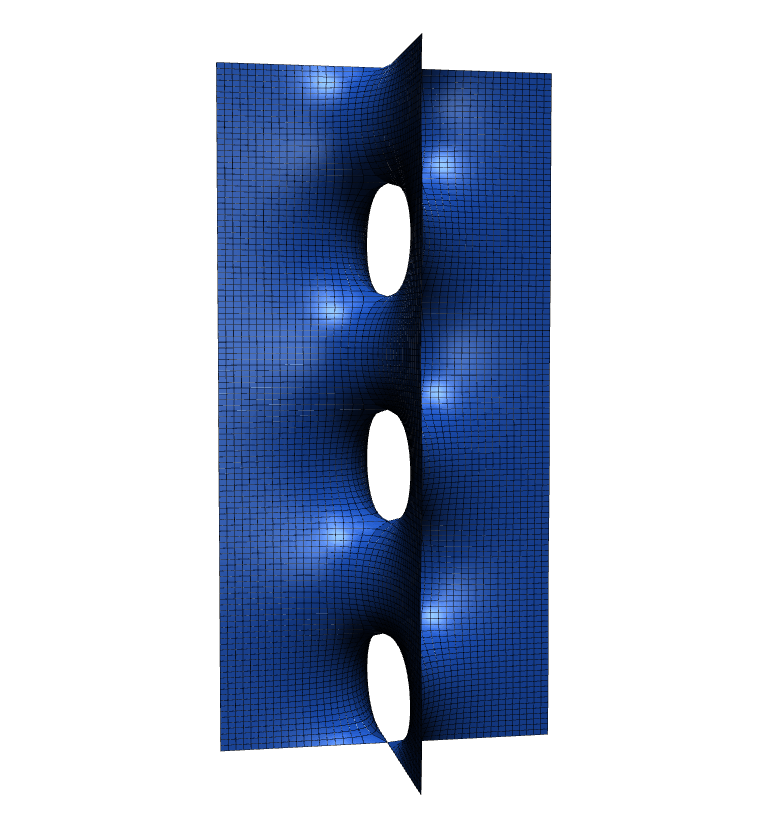}
\caption{Scherk's first surface and its conjugate, Scherk's fifth surface (parametrisation by height).}
\end{figure}

Denoting by $\gamma_p$ the positively oriented circle around
$p\in\{\pm 1, \pm i\}$ the periods $\gamma_p^*\Phi= \Phi + \tau_p +
\ii \tau_p^*$ of
$\Phi$ are given by 
\[
\tau_{\pm 1}  + \ii \tau^*_{\pm 1}= 2\pi(0,\pm 1, -\ii), \quad \tau_{\pm
  i} + \ii \tau^*_{\pm i} = 2\pi(\pm 1,0,\ii)\,.
\]
In particular,  the doubly--periodic first Scherk surface has periods
$
\gamma_p^*f = f +  \tau_p$ and its conjugate, 
the  simply--periodic fifth Scherk surface,  has periods  $\gamma_p^*f^*
= f^* +  \tau^*_p$.

\begin{figure}[H]
\includegraphics[height=6cm]{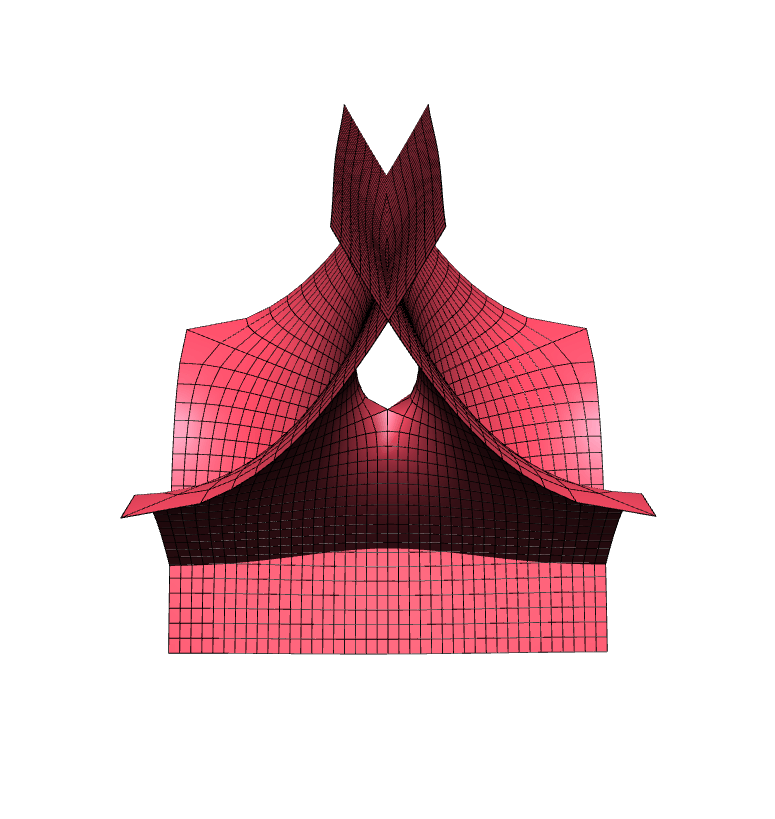}
\includegraphics[height=6cm]{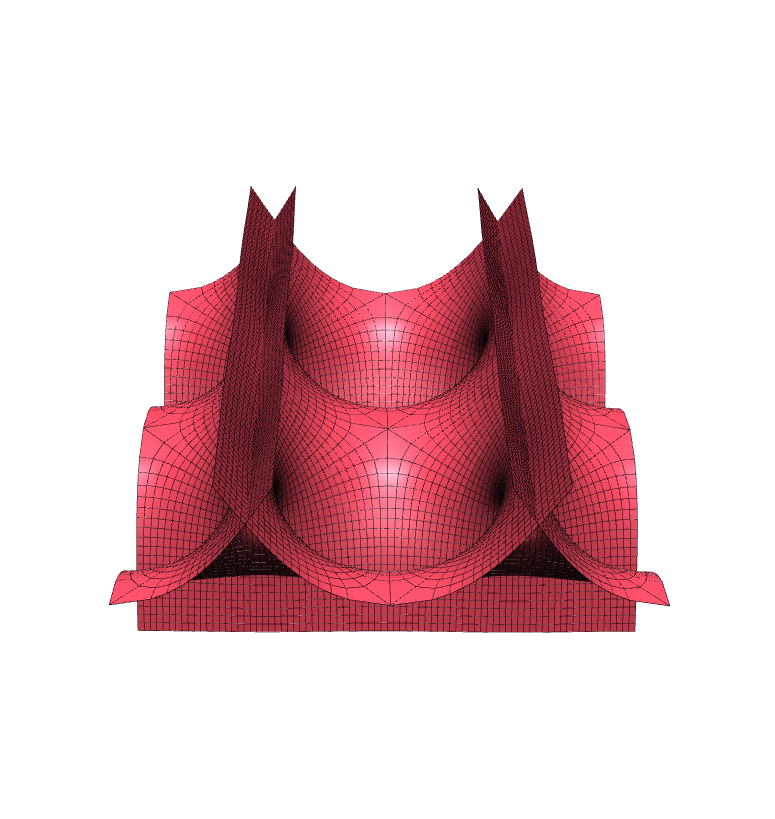}
\caption{L\'opez-Ros deformation of  Scherk's first surface with
  parameter $\sigma=0.6$}
\end{figure}

Since $  \tau^*_p \in \Span_{\R} \{k\}$ we see with Theorem
\ref{thm:Goursat} that the L\'opez-Ros deformation of $f$ with
parameter $\sigma$ is doubly--periodic with
\[
\gamma_p^*f_\sigma =f_\sigma + \begin{pmatrix} \cos t & - \sin t & 0 \\
  \sin t & \cos t & 0 \\ 0&0&1
\end{pmatrix} \tau_p \cosh s\,.
\]

\begin{figure}[H]
\includegraphics[height=6cm]{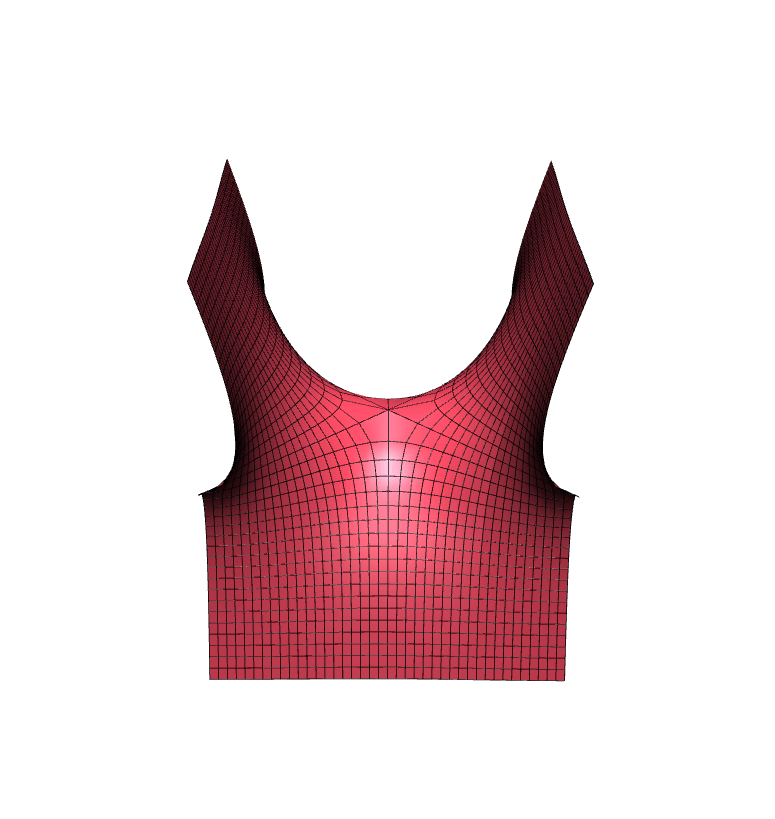}
\includegraphics[height=6cm]{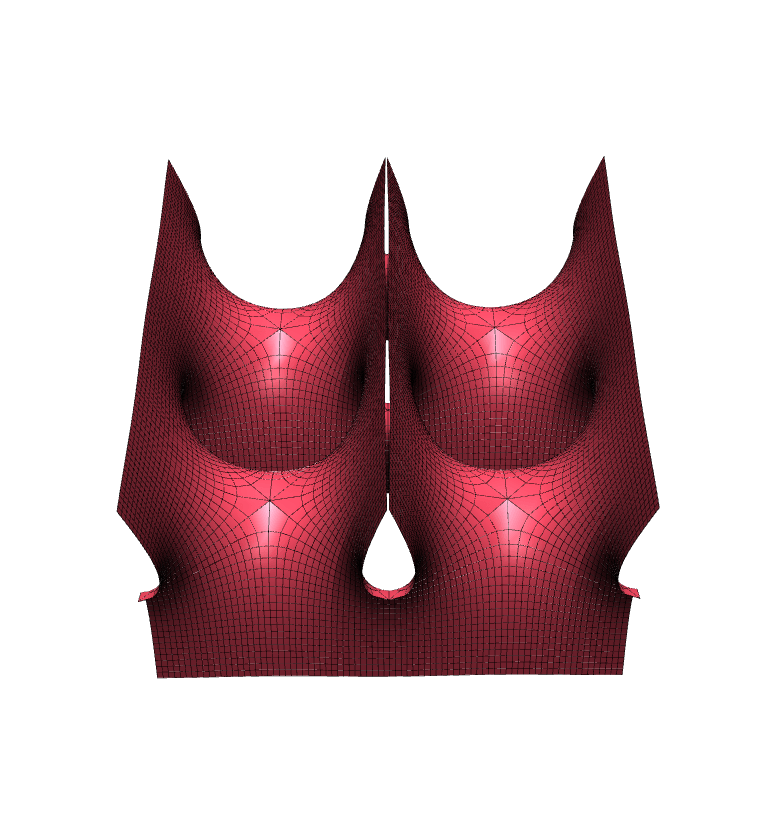}
\caption{L\'opez-Ros deformation of  Scherk's first surface with
  parameter  $\sigma = 1.4$.}
\end{figure}

Corollary \ref{cor: periods in 3 space} shows that the
periods of the simple factor dressing with parameter $\mu$  are given by
\[
\tau^\mu_{\pm 1} = \pm 2\pi e^{\pm s}\begin{pmatrix} 0 \\  \cos t  \\
  \sin t  
\end{pmatrix}\,,
\qquad \tau^\mu_{\pm i} = 
2\pi \begin{pmatrix}
\pm 1 \\ -\cos t \ \sinh s\\ -\sin t \ \sinh s
\end{pmatrix}\,,
\]
where $s= -\ln|\mu|,  t= \arg \frac{\bar\mu-1}{\bar\mu(1-\mu)}$; in
particular,   the
periods  cannot be simultaneously
closed. Moreover, since
  $\tau^\mu_1 + \tau^\mu_{-1} + \tau^\mu_{i} + \tau^\mu_{-i} =0$
and $\tau^\mu_1 = - \tau^\mu_{-1} e^{2s}$ we see  
  that $f^\mu$ is doubly--periodic with respect to the integer
lattice generated by $b\tau^\mu_1$ and $\tau^\mu_i$  if $s = \ln \sqrt
q$ with $q=\frac ab, a, b\in\N$.

In particular, for $q\in\N$ and $\mu = - \frac 1{\sqrt q}$ the simple
factor dressing with parameter $\mu$ is invariant under the integer lattice
$\Gamma=<\tau_1^\mu, \tau_i^\mu>$ since $q= \frac a b$ with $a=q,
b=1$, and thus $b\tau^\mu_1= \tau^\mu_1$.

\begin{figure}[H] 
\includegraphics[height=3.5cm]{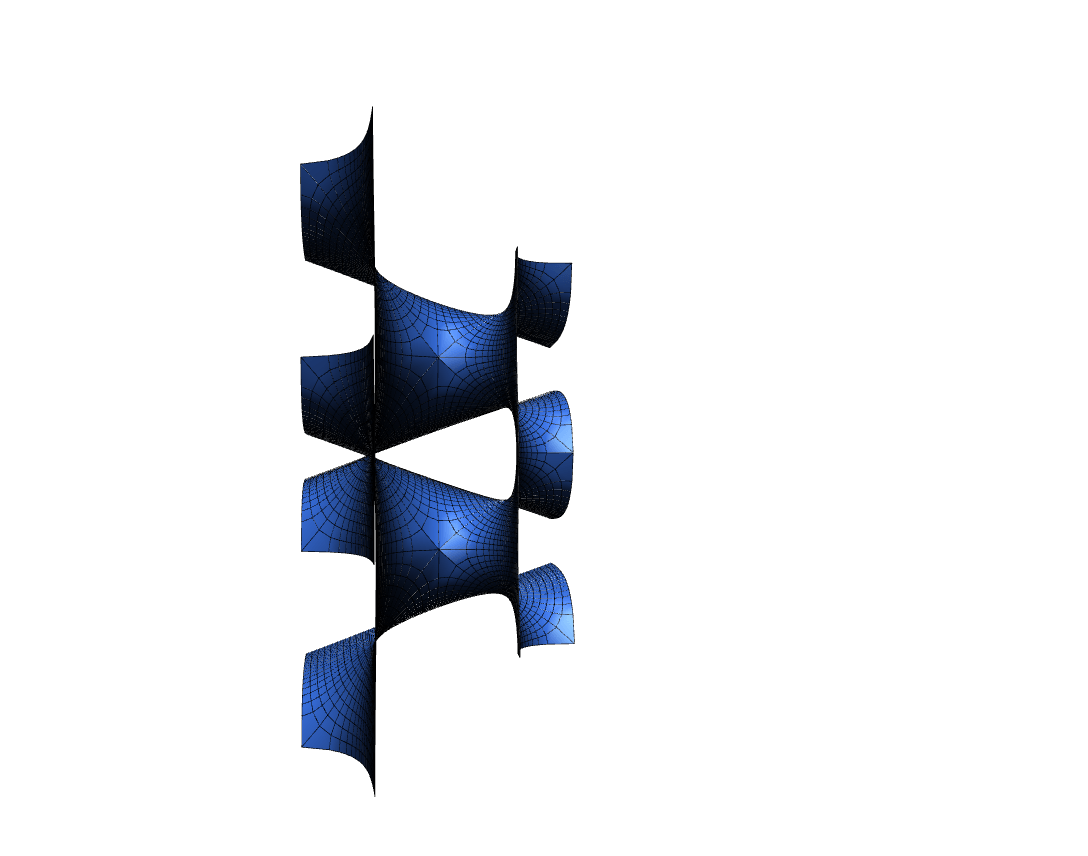}
\includegraphics[height=3.5cm]{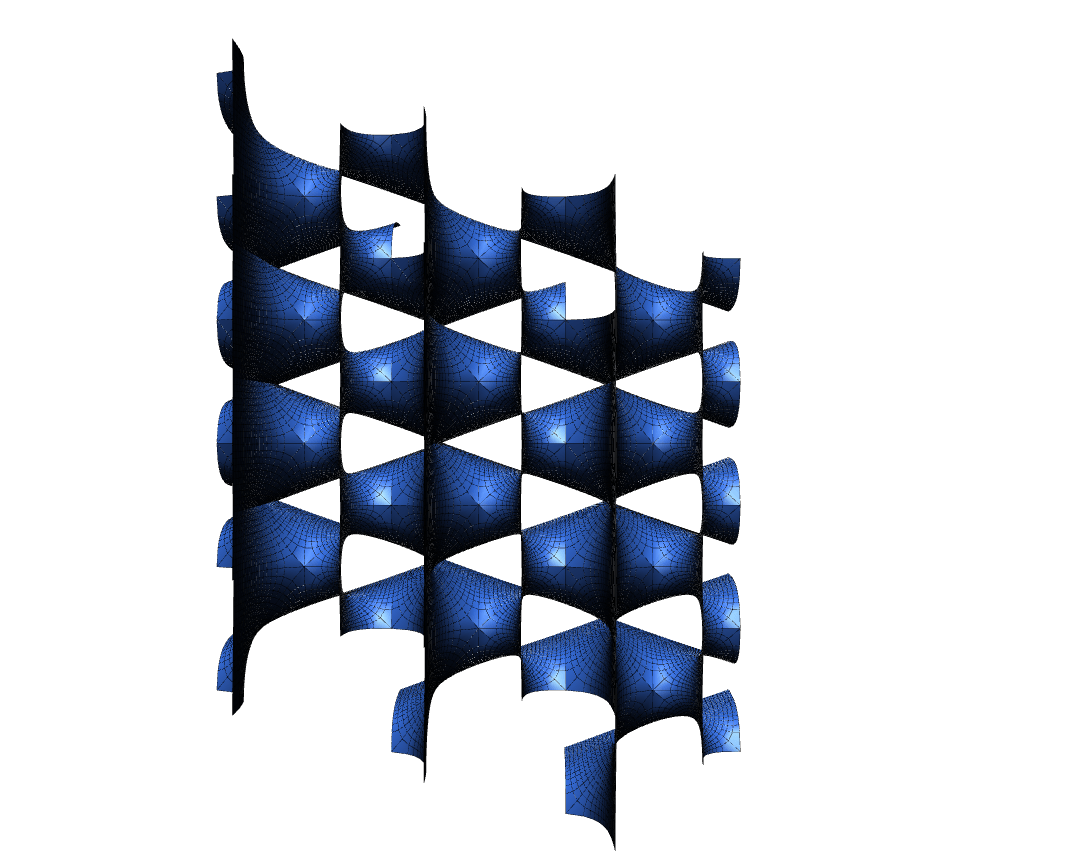}
\includegraphics[height=3.5cm]{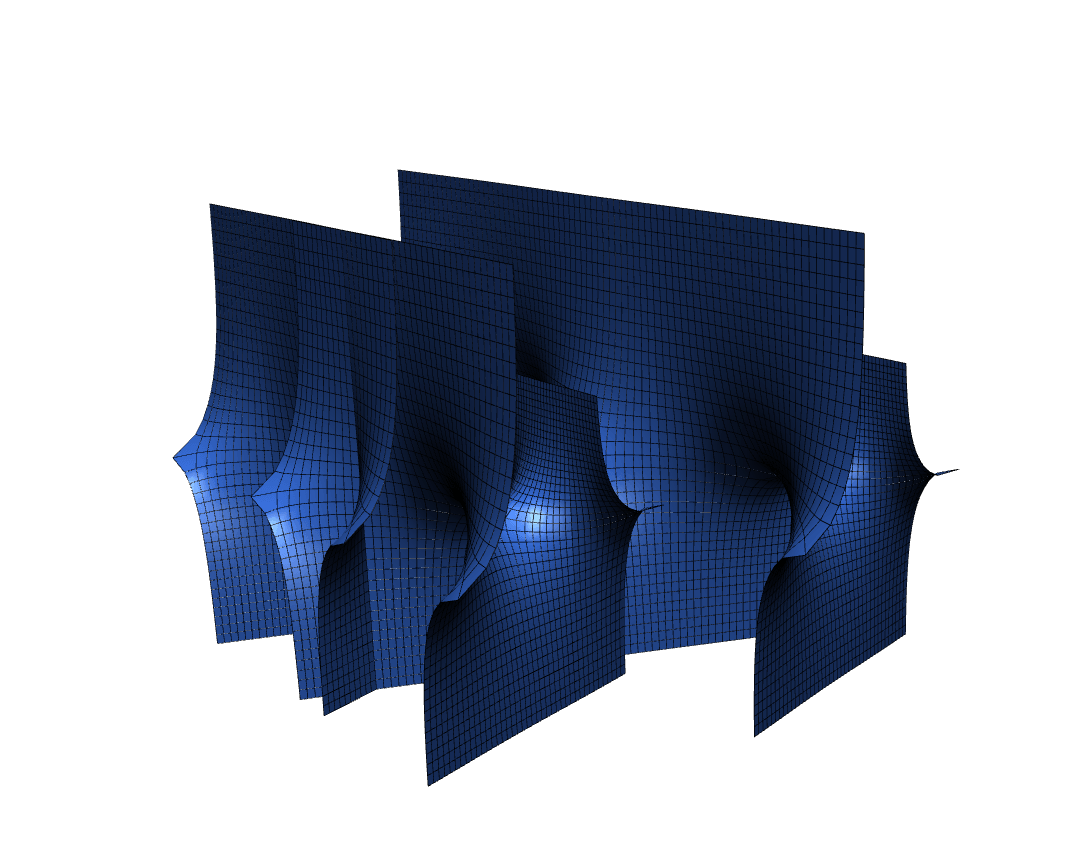}
\caption{Simple factor dressing with $\mu= -\frac 1{\sqrt 2}$  of
  Scherk's first surface: fundamental domain for the lattice $\Gamma =
  <2\pi \sqrt 2 j, 2\pi(i-\frac{\sqrt 2j}4)>$,
  larger piece of the surface, and side view.
}
\end{figure}

However, as already indicated by the pictures above, in this case  a simple factor dressing
is invariant under a smaller lattice: since $\tau_1^\mu =- q\tau_{-1}^\mu$
with $q\in\N$, we see that the simple factor dressing is invariant
under the integer lattice $\hat\Gamma=<\tau_{-1}^\mu, \tau_i^\mu>$.

\begin{figure}[H]
\includegraphics[height=3.5cm]{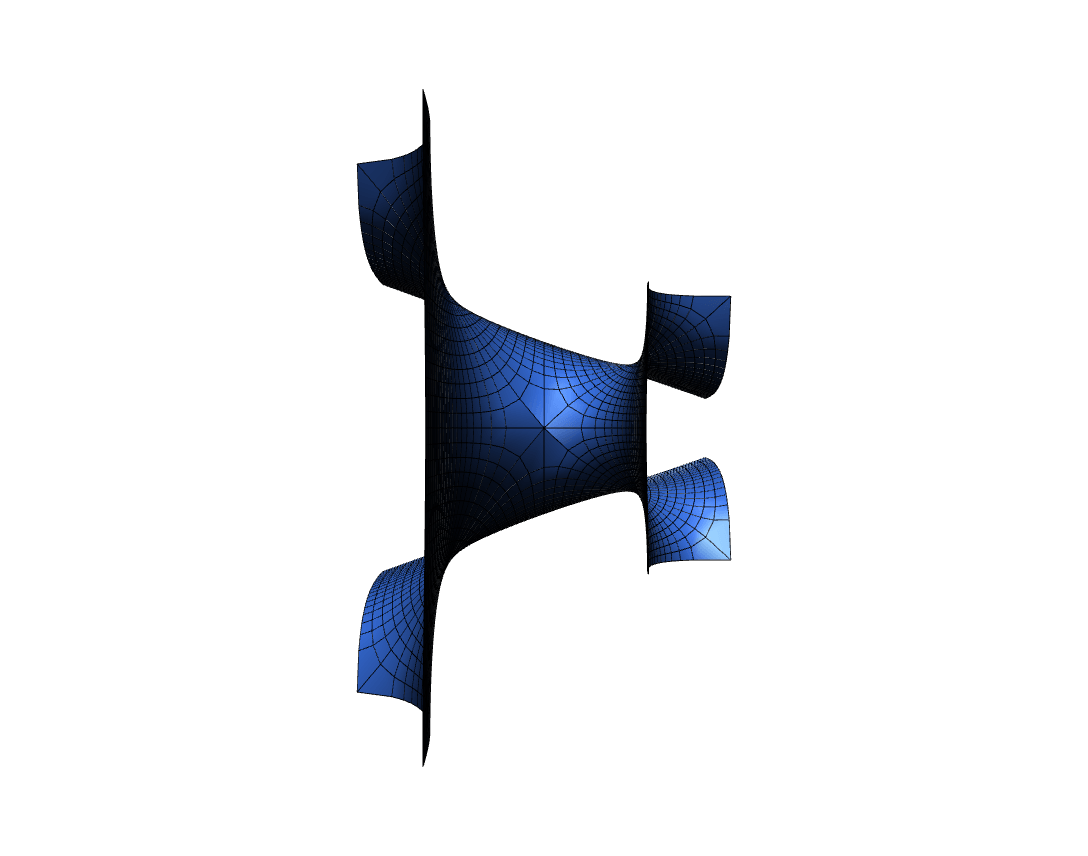}
\includegraphics[height=3.5cm]{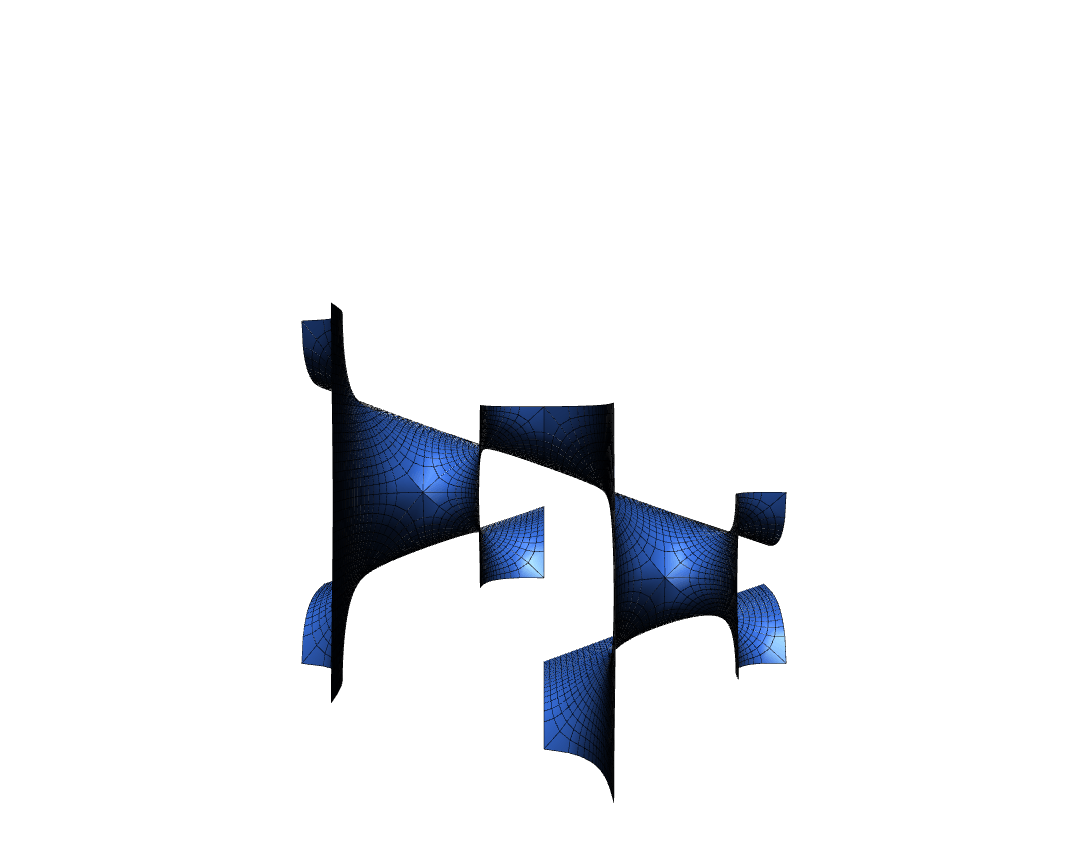}
\includegraphics[height=3.5cm]{SFDscherkfdtau1.png}
\caption{Simple factor dressing with $\mu= -\frac 1{\sqrt 2}$  of
  Scherk's first surfacee: fundamental domain  for the lattice $\hat\Gamma = <\tau^\mu_i , \tau^\mu_{-1}>$, translations by
  $\tau^\mu_{i} =  2\pi(i-\frac{\sqrt 2j}4)$ and $\tau^\mu_{-1} = -\pi \sqrt 2 j$.}
\end{figure}

We conclude the example of Scherk's first surface by providing the
pictures for  transforms with values in 4--space. The left-- and
right--associated family are minimal surfaces in $\R^4$

\begin{figure}[H]
\includegraphics[height=4.5cm]{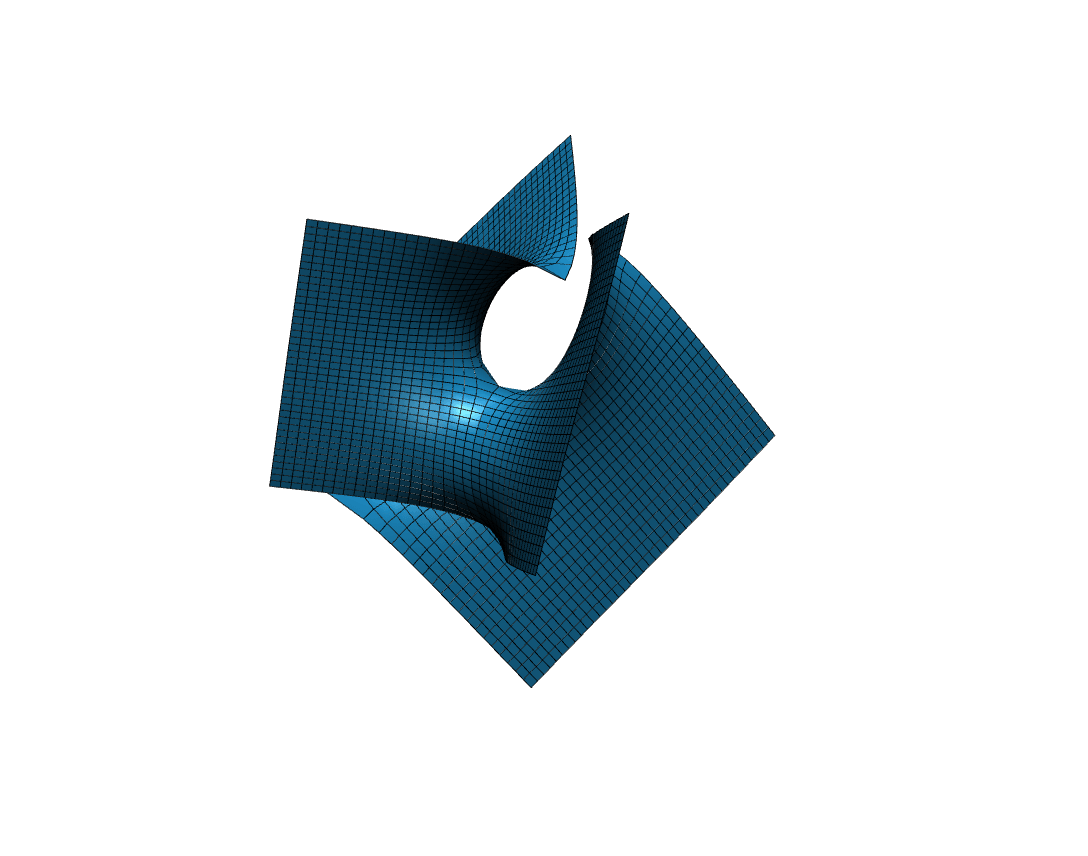}
\includegraphics[height=4.5cm]{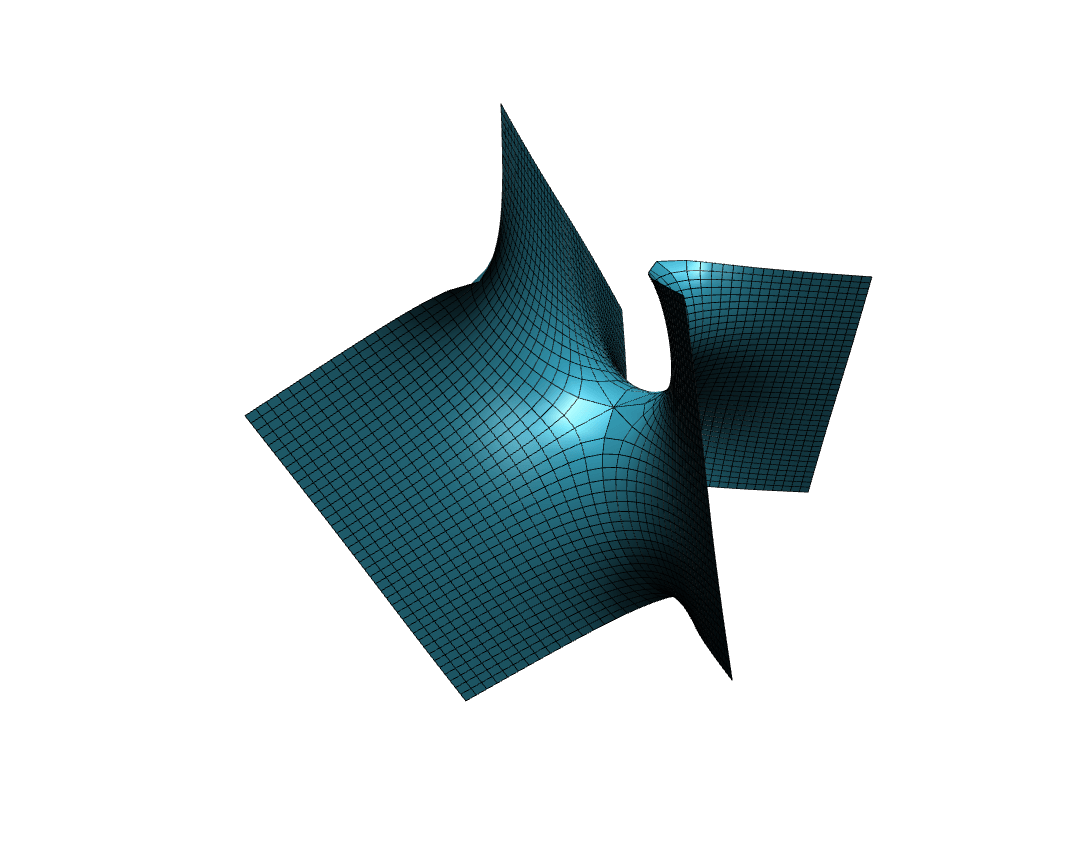}  
\caption{Elements   $f^{\frac 1{\sqrt 7}, \frac 2{\sqrt {7}} +\frac
    j{\sqrt 7} - \frac k{\sqrt 7}}$ and $f_{\frac 1{\sqrt 7}, \frac 2{\sqrt {7}} +\frac
    j{\sqrt 7} - \frac k{\sqrt 7}}$  of the left and right associated family of
  Scherk's first surface, orthogonally projected into $\R^3$.}
\end{figure}

whereas the associated Willmore surface

\begin{figure}[H]
\includegraphics[height=4cm]{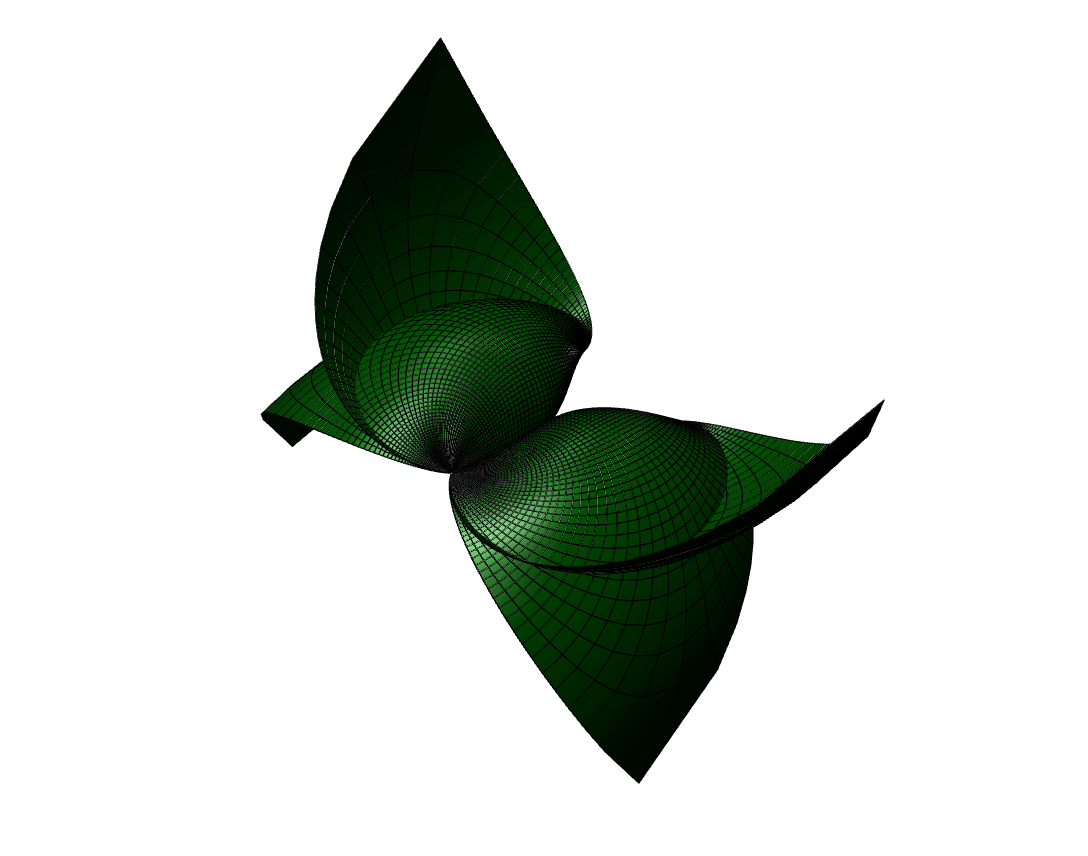}
\includegraphics[height=4cm]{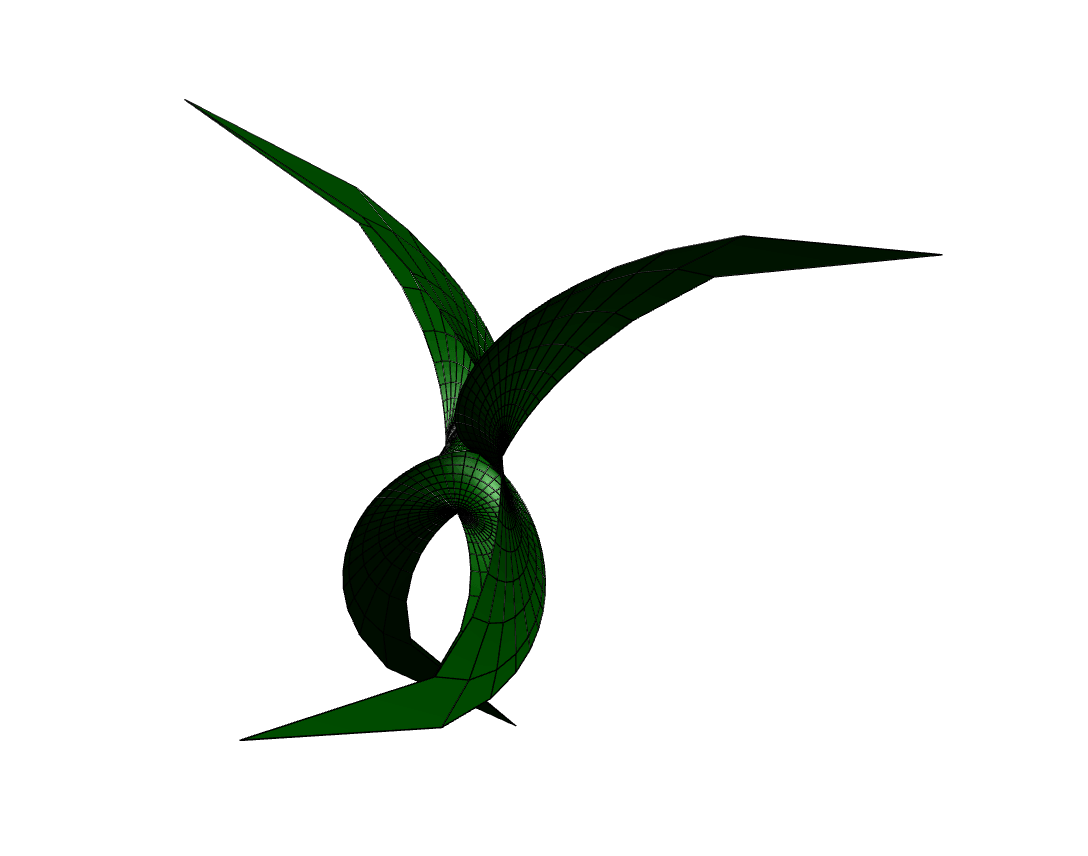}
\caption{The associated Willmore surface  of   Scherk's first surface
  orthogonally projected onto $\R^3$.}
\end{figure}

and the $\mu$--Darboux transforms are Willmore surfaces in $\R^4$.

\begin{figure}[H]
\includegraphics[height=5.5cm]{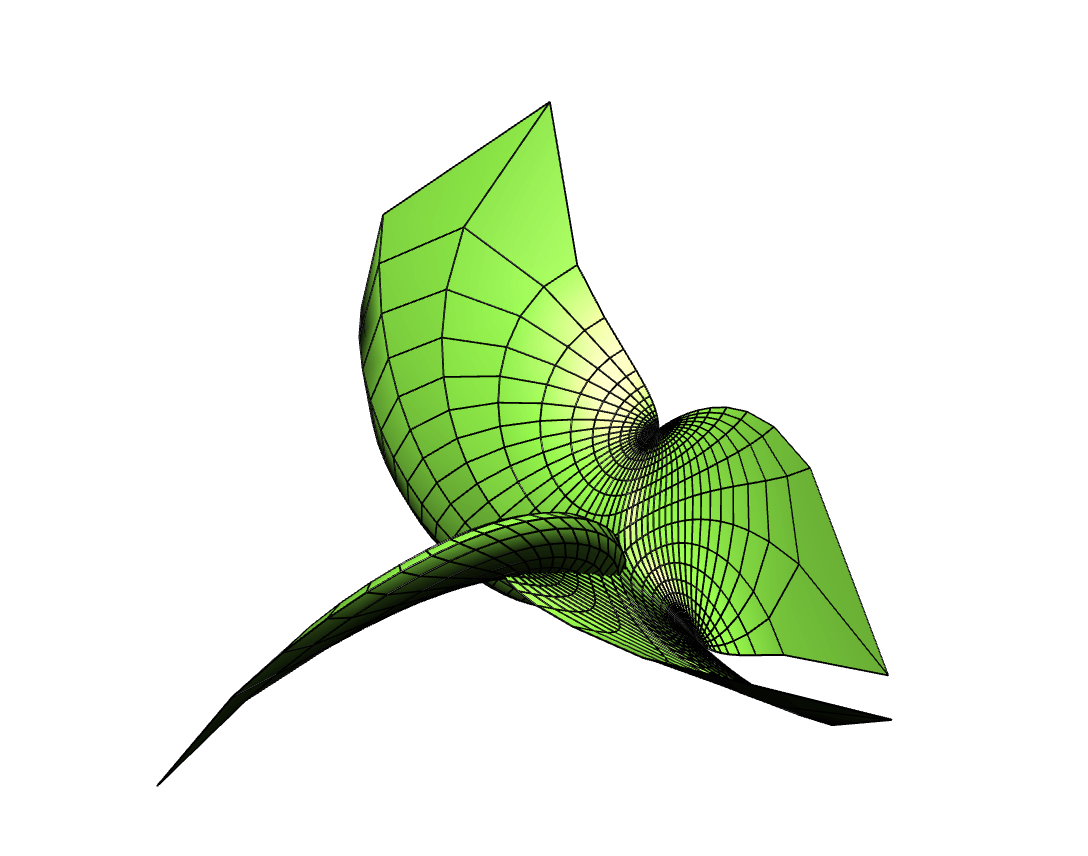}

\caption{The $\mu$--Darboux
  transform with $\mu=-\frac i2, m=1$, of Scherk's first surface, orthogonally projected into $\R^3$.}
\end{figure}


\subsection{Minimal torus}
Our final example is a minimal surface whose holomorphic null curve is
defined on a punctured torus. If  $\Lambda$ is the lattice in $\C$
over $\Z$ spanned by the two periods $\omega_1$ and $\omega_2$ then
the Weierstrass' elliptic function $\wp$  associated with $\Lambda$  satisfies
\[
(\wp')^2 = 4\wp^3- g_2\wp -g_3\,
\]
and thus
$
\wp''= 6\wp^2 -\frac 12g_2
$,
where $g_2, g_3$ are the invariants
\[
g_2(\omega_1, \omega_2) = 60 \sum_{(m,n)\not=(0,0)} (m\omega_1 + n\omega_2)^{-4}, 
\quad 
g_3(\omega_1, \omega_2) = 140 \sum_{(m,n)\not=(0,0)} (m\omega_1 + n\omega_2)^{-6}\,.
\]

\begin{figure}[H]
\includegraphics[height=3.8cm]{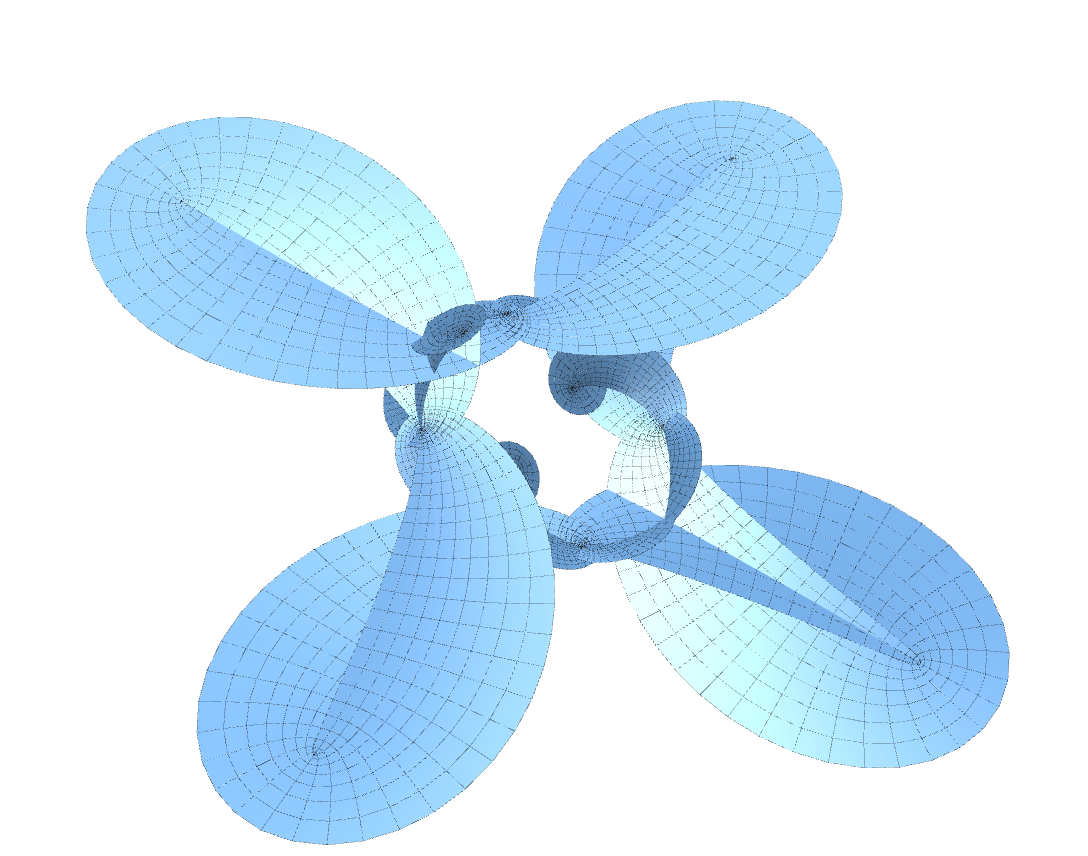}
\caption{Small--Weierstrass torus.}
\end{figure}

With this at hand, the minimal surface  in Example 6.2 in \cite{small} reads as
\[
f =\Re \Phi: M\to\R^3
\]
where $M = (\C \setminus \{0,\frac{\omega_1}2, \frac{\omega_2}2,
\frac{\omega_1+\omega_2}2\})/\Lambda$ and  
$\Phi=(\Phi_1, \Phi_2, \Phi_3)$ is the holomorphic null curve in $\Cc^3$ given by 
\begin{eqnarray*}
\Phi_1&=&  \frac{1}{8(\wp')^3} 
(-g_2^2 -8g_3^2-48g_3\wp - 12g_2g_3\wp -24g_2\wp^2 -3g_2^2\wp^2+64g_3\wp^3
\\
&& \qquad \qquad + 48\wp^4+24g_2\wp^4 + 16\wp^6 )\\
\Phi_2&=&
\frac{\ii}{8(\wp')^3} (-g_2^2 +8g_3^2-48g_3\wp + 12g_2g_3\wp -24g_2\wp^2 +3g_2^2\wp^2-64g_3\wp^3 \\
&& \qquad \qquad + 48\wp^4-24g_2\wp^4 - 16\wp^6) \\
\Phi_3&=&
\frac{1}{4(\wp')^3}(-2g_2g_3 -3g_2^2\wp -24g_3\wp^2 + 8g_2\wp^3 -48\wp^5)
\end{eqnarray*}
in $\C^3$.  As before, the Gauss map is given by the stereographic
projection
(\ref{eq:Gauss with g})  of
$
g = \frac{d\Phi_3}{d\Phi_1 - \ii d\Phi_2}$.

Since $\Phi$ is defined on the punctured torus so are both $f$ and
$f^*$. Again, this implies that all discussed transformes are defined
on the same punctured torus. The following pictures are made for a
square torus with invariants $g_2 = 100$ and $g_3=0$.

\begin{figure}[H]
\includegraphics[height=4cm]{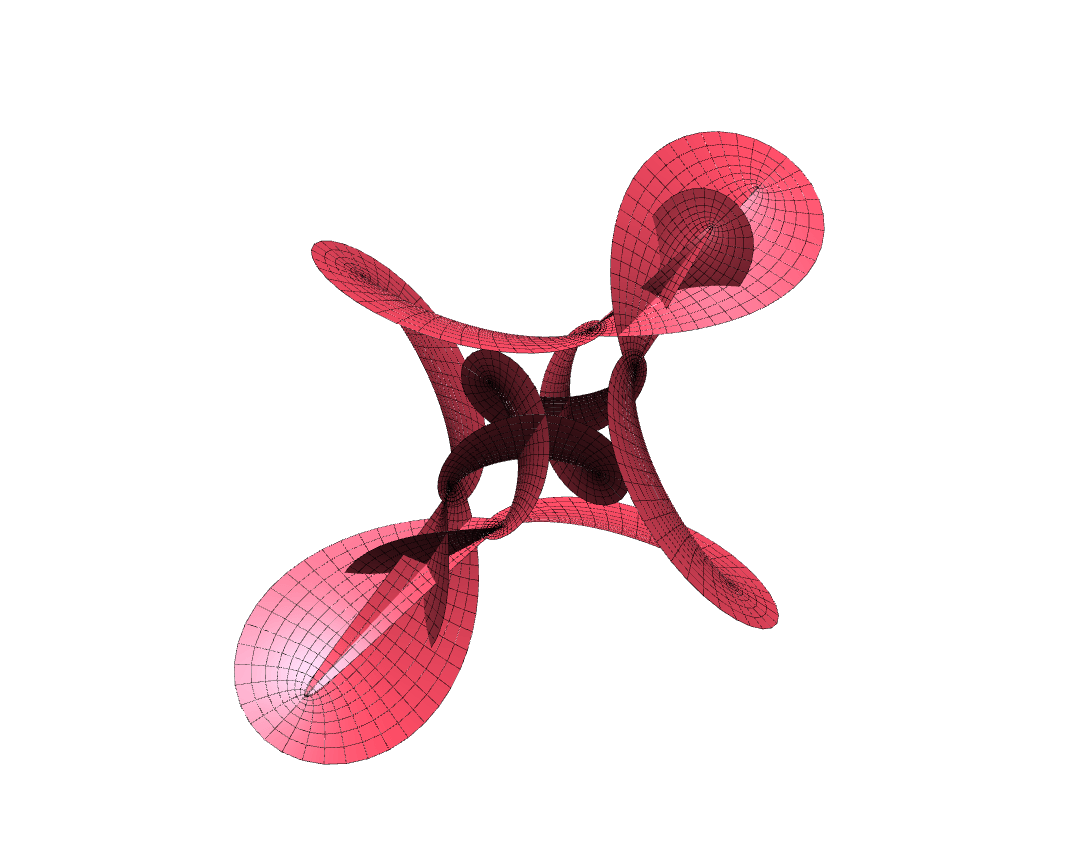}
\includegraphics[height=4cm]{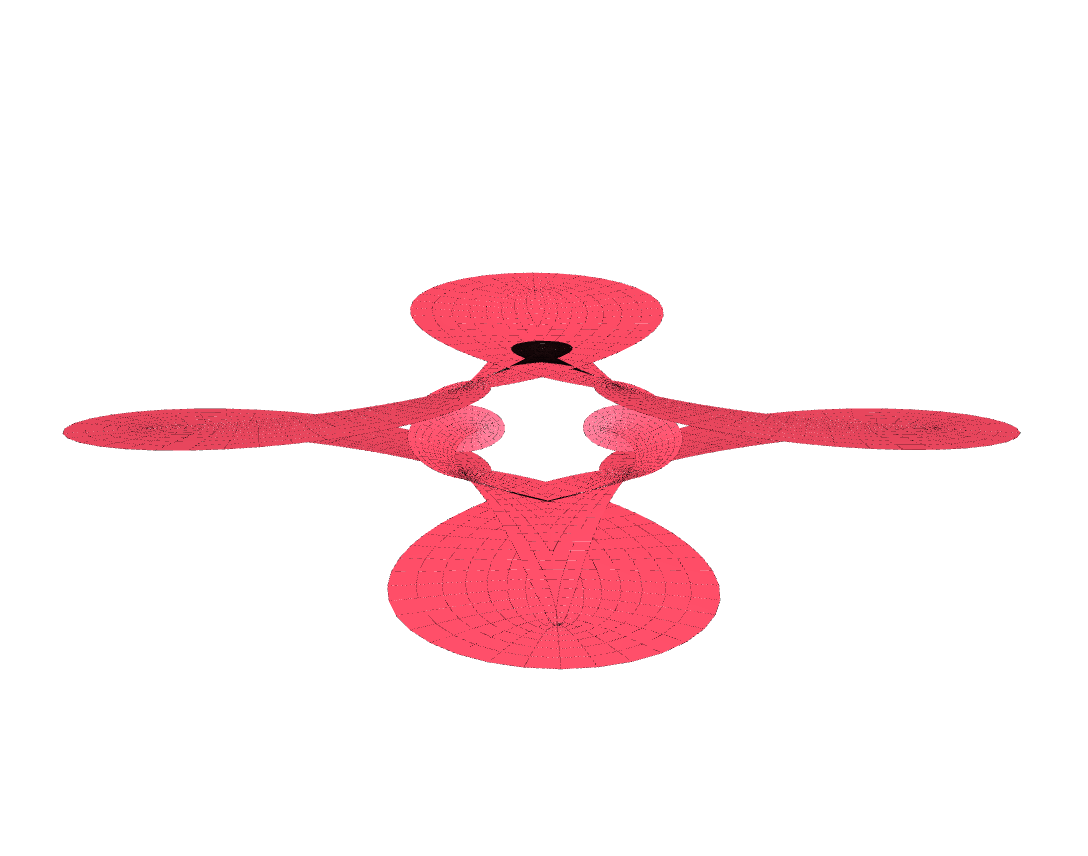}
\caption{L\'opez-Ros deformation of  a Small--Weierstrass torus with
  parameter $\sigma=0.3$ and $\sigma=1.4$}
\end{figure}
\vspace{-1cm}
\begin{figure}[H]
\includegraphics[height=4.5cm]{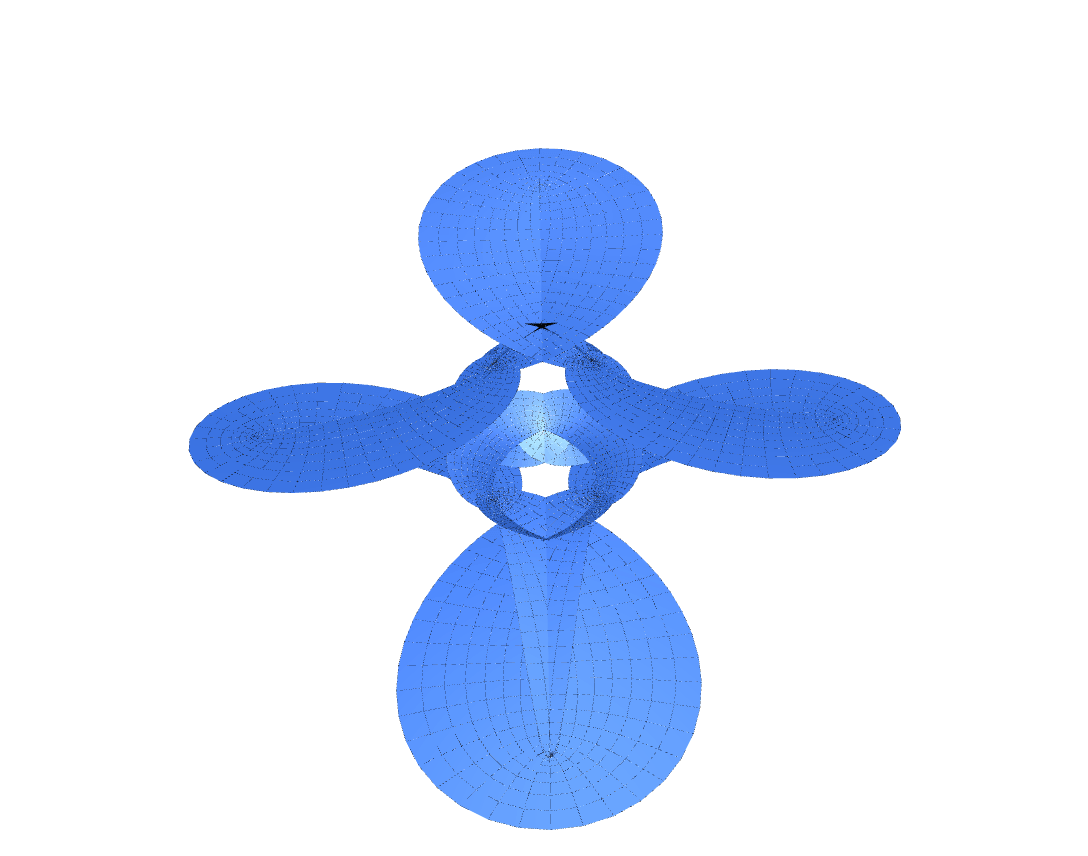}
\includegraphics[height=4.5cm]{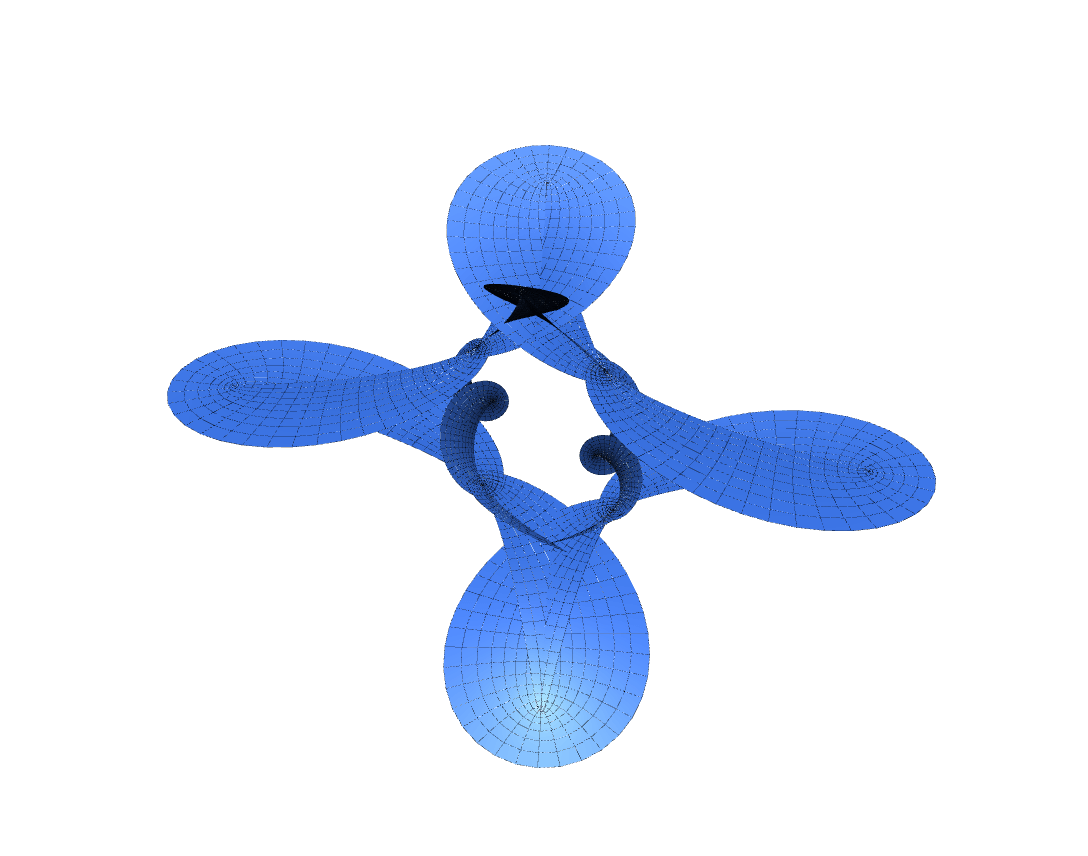}
\caption{Simple factor dressing of  a Small--Weierstrass torus  with
  parameter  $\mu = -\frac i2$ and parameters $(-\frac i2, m, m)$ with
  $m=1-i+j +2k$ respectively.}
\end{figure}
 \vspace{-1cm}
\begin{figure}[H]
\includegraphics[height=4.5cm]{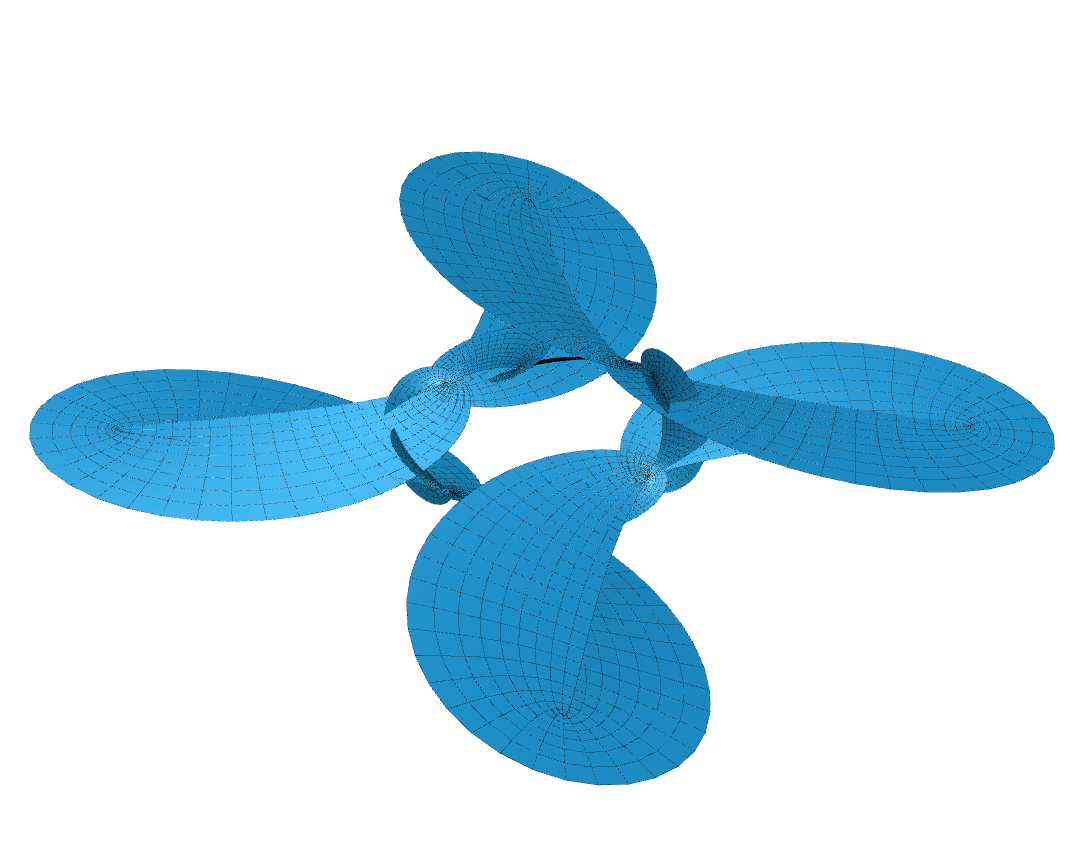}
\includegraphics[height=4.5cm]{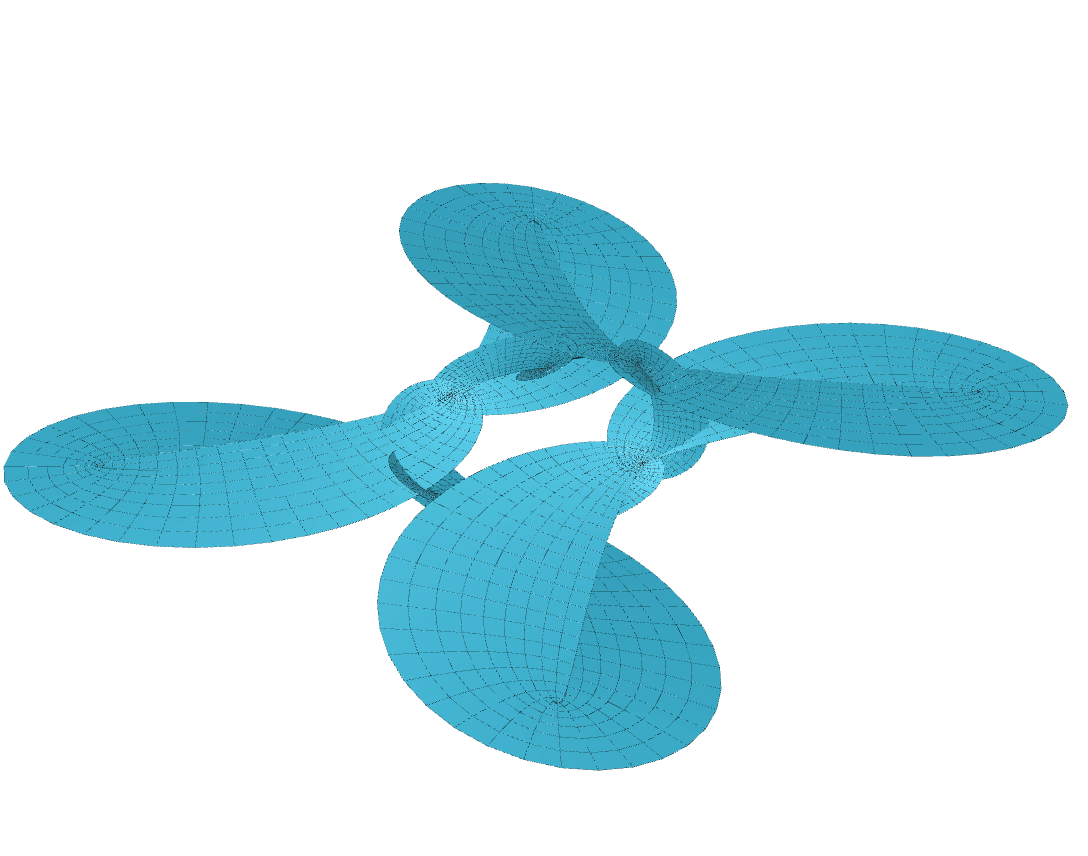}  
\caption{Elements   $f^{\frac 1{\sqrt 7}, \frac 2{\sqrt {7}} +\frac
    j{\sqrt 7} - \frac k{\sqrt 7}}$ and $f_{\frac 1{\sqrt 7}, \frac 2{\sqrt {7}} +\frac
    j{\sqrt 7} - \frac k{\sqrt 7}}$  of the left and right associated family of
 a Small--Weierstrass torus, orthogonally projected into $\R^3$.}
\end{figure}
\vspace{-1cm}
\begin{figure}[H]
\includegraphics[height=3.5cm]{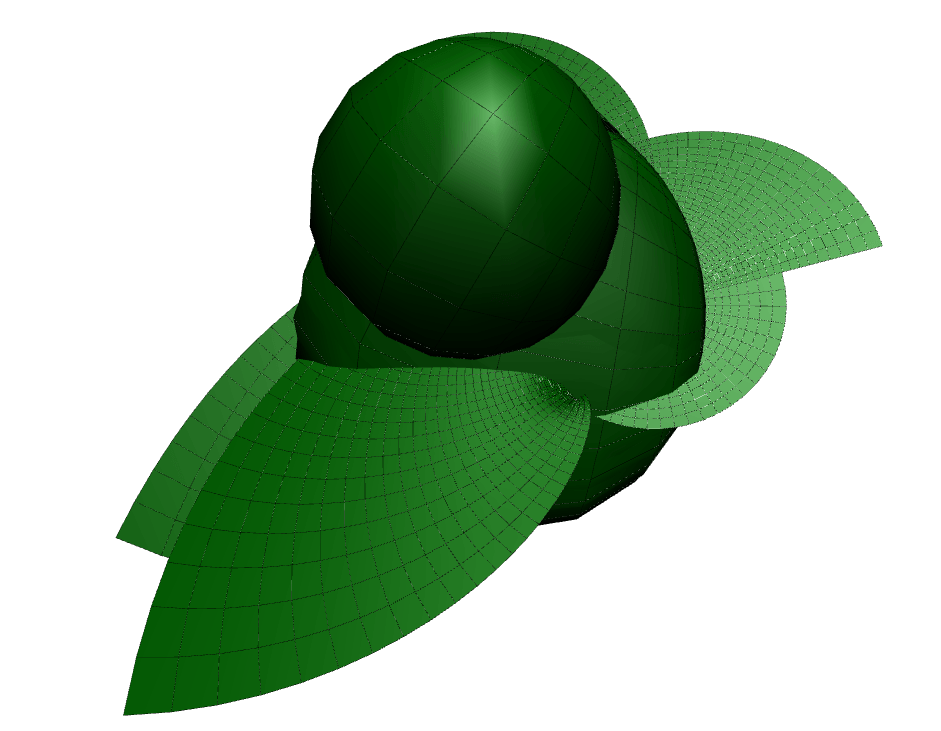}
\includegraphics[height=4cm]{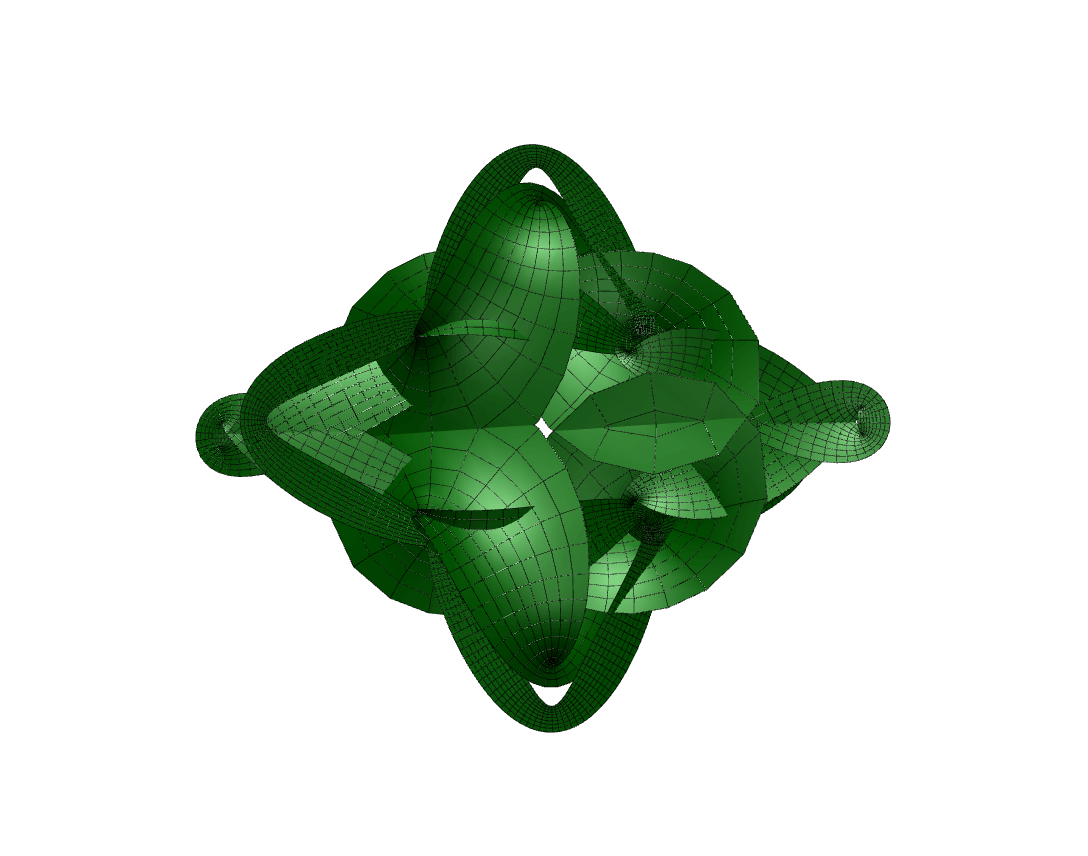}
\caption{The associated Willmore surface  of  a Small--Weierstrass torus,
  orthogonally projected onto $\R^3$.}
\end{figure}

\begin{figure}[H]
\includegraphics[height=5.5cm]{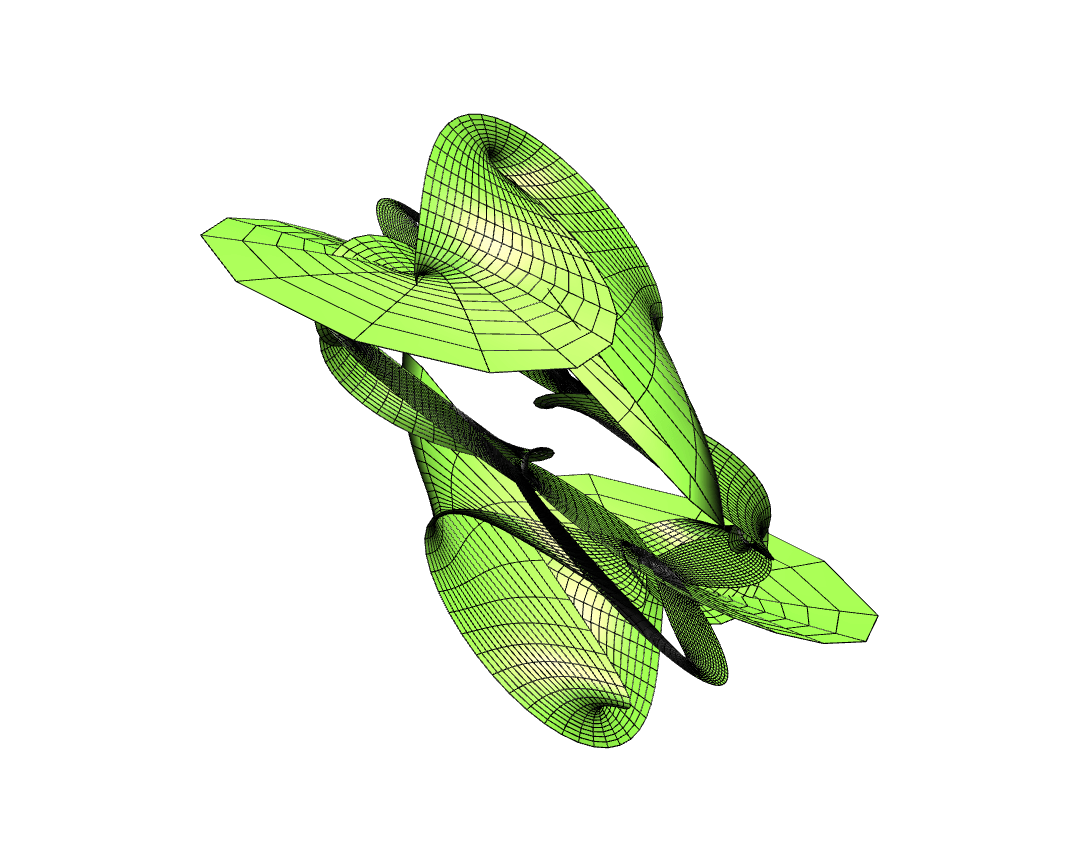}

\caption{The $\mu$--Darboux
  transform with $\mu=-\frac i2, m=1$, of a Small--Weierstrass torus, orthogonally projected into $\R^3$.}
\end{figure}


%
\todoo{\input{questions}}

\bibliographystyle{alpha}
\bibliography{doc}

\end {document}